\documentclass[reqno, a4paper]{amsart}
\usepackage{amsmath, amssymb, amsthm}
\usepackage[usenames]{color}
\usepackage[hmargin=3cm, vmargin=3cm]{geometry}

\newtheorem{thm}{Theorem}[section]

\newtheorem{lem}[thm]{Lemma}
\newtheorem{prop}[thm]{Proposition}
\theoremstyle{definition}
\newtheorem{defn}[thm]{Definition}
\theoremstyle{remark}
\newtheorem{rem}[thm]{Remark}
\numberwithin{equation}{section}
\newtheorem{slem}[thm]{Sublemma}

\newcommand{\bls}{\boldsymbol\lambda, \boldsymbol\sigma}
\newcommand{\bl}{\boldsymbol\lambda}
\newcommand{\bs}{\boldsymbol\sigma}

\newcommand{\dist}{\text{dist}}

\begin{document}

\title[Nonlinear elliptic problems with the fractional Laplacian]{Asymptotic behavior of solutions for nonlinear elliptic problems with the fractional Laplacian}

\author{Woocheol Choi}
\address[Woocheol Choi]{Department of Mathematical Sciences, Seoul National University, 1 Gwanak-ro, Gwanak-gu, Seoul 151-747, Republic of Korea}
\email{chwc1987@math.snu.ac.kr}

\author{Seunghyeok Kim}
\address[Seunghyeok Kim]{Department of Mathematical Sciences, KAIST, 291 Daehak-ro, Yuseong-gu, Daejeon 305-701, Republic of Korea}
\curraddr{Departamento de Matem\'{a}tica, Pontificia Universidad Cat\'{o}lica de Chile, Avenida Vicu\~{n}a Mackenna 4860, Santiago, Chile}
\email{shkim0401@gmail.com}

\author{Ki-Ahm Lee}
\address[Ki-Ahm Lee]{Department of Mathematical Sciences, Seoul National University, 1 Gwanak-ro, Gwanak-gu, Seoul 151-747, Republic of Korea \& Center for Mathematical Challenges, Korea Institute for Advanced Study, Seoul,130-722, Republic of Korea}
\email{kiahm@math.snu.ac.kr}


\begin{abstract}
In this paper we study the asymptotic behavior of least energy solutions and the existence of multiple bubbling solutions of nonlinear elliptic equations involving the fractional Laplacians and the critical exponents.
This work can be seen as a nonlocal analog of the results of Han (1991) \cite{H} and Rey (1990) \cite{R}.
\end{abstract}

\maketitle

\section{Introduction}
The aim of this paper is to study the nonlocal equations:
\begin{equation}\label{a12uu}
\left\{\begin{array}{ll} \mathcal{A}_{s} u = u^{p}+\epsilon u &\text{in}~\Omega,
\\
u>0 &\text{in} ~\Omega,
\\
u=0 & \text{on}~\partial \Omega,
\end{array}\right.
\end{equation}
where $0< s < 1$, $p := \frac{n+2s}{n-2s}$, $\epsilon > 0$ is a small parameter, $\Omega$ is a smooth bounded domain of $\mathbb{R}^n$ and $\mathcal{A}_{s}$ denotes the fractional Laplace operator $(-\Delta)^s$ in $\Omega$ with zero Dirichlet boundary values on $\partial \Omega$, defined in terms of the spectra of the Dirichlet Laplacian $-\Delta$ on $\Omega$. It can be understood as the nonlocal version of the Brezis-Nirenberg problem \cite{BN}.

The fractional Laplacian appears in diverse areas including physics, biological modeling and mathematical finances and partial differential equations involving the fractional Laplacian have attracted the attention of many researchers.
An important feature of the fractional Laplacian is its nonlocal property, which makes it difficult to handle.
Recently, Caffarelli and Silvestre \cite{CS} developed a local interpretation of the fractional Laplacian given in $\mathbb{R}^n$ by considering a Neumann type operator in the extended domain $\mathbb{R}^{n+1}_{+} := \{(x,t) \in \mathbb{R}^{n+1}: t > 0\}$.
This observation made a significant influence on the study of related nonlocal problems.
A similar extension was devised by Cabr\'e and Tan \cite{CT} and Capella, D\'{a}vila, Dupaigne, and Sire \cite{CDDS} (see Br\"andle, Colorado, de Pablo, and S\'anchez \cite{BCPS1} and Tan \cite{T2} also) for nonlocal elliptic equations on bounded domains with the zero Dirichlet boundary condition,
and by Kim and Lee \cite{KL} for singular nonlocal parabolic equations.

Based on these extensions, many authors studied nonlinear problems of the form $\mathcal{A}_{s} u = f(u)$, where $f:\mathbb{R}^n \rightarrow \mathbb{R}$ is a certain function.
Since it is almost impossible to describe all the works involving them, we explain only some results which are largely related to our problem.
When $s= {1 \over 2}$, Cabr\'e and Tan \cite{CT} established the existence of positive solutions for equations having nonlinearities with the subcritical growth, their regularity and the symmetric property.
They also proved a priori estimates of the Gidas-Spruck type by employing a blow-up argument along with a Liouville type result for the square root of the Laplacian in the half-space.
Moreover, Tan \cite{T} studied Brezis-Nirenberg type problems (see \cite{BN}) for the case $s = {1 \over 2}$, that is, when the nonlinearity is given by $f(u) = u^{\frac{n+1}{n-1}} + \epsilon u$ with $\epsilon >0$.
On the other hand, the first author of this paper gave a different proof for the Gidas-Spruck type estimates using the Pohozaev identity and applied them to the Lane-Emden type system involving $A_{1/2}$.
The work of Tan \cite{T} is extended to $0<s<1$ and $f(u) = u^{\frac{n+2s}{n-2s}}+ \lambda u^q$  for $0<q< \frac{n+2s}{n-2s}$ in \cite{BCPS2} .
See also \cite{BCPS1} which dealt with a subcritical concave-convex problem.
For $f(u) = u^q$ with the critical and supercritical exponents $q \geq \frac{n+2s}{n-2s}$, the nonexistence of solutions was proved in \cite{BCPS1, T, T2}
in which the authors devised and used the Pohozaev type identities.

\medskip
The aim of this paper is to study the problem \eqref{a12uu} when $p=\frac{n+2s}{n-2s}$ is the critical Sobolev exponent and $\epsilon >0$ is close to zero. During this study we develop some nonlocal techniques which also have their own interests.

\medskip
The first part is devoted to study least energy solutions of \eqref{a12uu}.
To state the result, we recall from \cite{CoT} that the sharp fractional Sobolev inequality for $n>2s$ and $s>0$
\[\left( \int_{\mathbb{R}^n} |f(x)|^{p+1} dx \right)^{1 \over p+1} \leq \mathcal{S}_{n,s} \left( \int_{\mathbb{R}^n} |\mathcal{A}_{s}^{1/2} f(x)|^2 dx \right)^{1 \over 2} \quad \text{for any } f \in H^s(\mathbb{R}^n)\]
which holds with the constant
\begin{equation}\label{ans}
\mathcal{S}_{n,s} =2^{-s} \pi^{-s/2} \biggl[\frac{\Gamma \left(\frac{n-2s}{2}\right)}{\Gamma \left( \frac{n+2s}{2}\right)}\biggr]^{1 \over 2} \biggl[ \frac{\Gamma(n)}{\Gamma(n/2)}\biggr]^{s \over n}.
\end{equation}
Our first result is the following.
\begin{thm}\label{thm-m-limit}
Assume $0 < s < 1$ and $n> 4s$.
For $\epsilon>0$, let $u_{\epsilon}$ be a solution of \eqref{a12uu} such that
\begin{equation}\label{2p1s0}
\lim_{\epsilon \rightarrow 0} \frac{ \int_{\Omega} |\mathcal{A}_{s}^{1/2} u_{\epsilon}|^2 dx }{\left( \int_{\Omega} |u_{\epsilon}|^{p+1} dx\right)^{2/(p+1)} } = \mathcal{S}_{n,s}.
\end{equation}
Then there exist a point $x_0 \in \Omega$ and a constant $\mathfrak{b}_{n,s} >0$ such that
\[u_{\epsilon} \rightarrow 0 \text{ in } \left\{\begin{array}{lll} C^{\alpha}_{\textrm{loc}}(\Omega \setminus \{x_0\})&\ \text{for all }  \alpha \in (0,2s) & \textrm{if } s \in (0,1/2],
\\
C^{1,\alpha}_{\textrm{loc}}(\Omega \setminus \{x_0\})&\ \text{for all } \alpha
\in (0,2s-1) & \textrm{if } s \in (1/2,1),
\end{array}
\right.\]
and
\[\| u_{\epsilon} (x) \|_{L^{\infty}} u_{\epsilon} (x) \rightarrow \mathfrak{b}_{n,s} G(x,{x_0}) \textrm{ in } \left\{\begin{array}{lll} C^{\alpha}_{\textrm{loc}}(\Omega \setminus \{x_0\}) &\ \text{for all } \alpha \in (0,2s) & \textrm{if } s \in (0,1/2],
\\
C^{1,\alpha}_{\textrm{loc}}(\Omega \setminus \{x_0\})&\ \text{for all } \alpha
\in (0,2s-1) & \textrm{if } s \in (1/2,1),
\end{array}
\right.\]
as $\epsilon$ goes to 0.
The constant $\mathfrak{b}_{n,s}$ is explicitly computed in Section \ref{sec_asymp} (see \eqref{eq-beta}).
\end{thm}
\noindent Here the function $G = G(x,y)$ for $x, y \in \Omega$ is Green's function of $\mathcal{A}_{s}$ with the Dirichlet boundary condition, which solves the equation
\begin{equation}\label{Green_Omega}
\mathcal{A}_{s} G(\cdot, y) = \delta_y \text{ in } \Omega \quad \text{and} \quad G(\cdot, y) = 0 \text{ on } \partial\Omega.
\end{equation}
The regular part of $G$ is given by
\begin{equation}\label{Robin_Omega}
H(x,y) = \frac{\mathfrak{a}_{n,s}}{|x-y|^{n-2s}} - G(x,y) \quad \text{where }  \mathfrak{a}_{n,s} = {1 \over |S^{n-1}|} \cdot {2^{1-2s}\Gamma({n-2s \over 2}) \over \Gamma({n \over 2})\Gamma(s)}.
\end{equation}
The diagonal part $\tau$ of the function $H$, namely, $\tau(x) := H(x,x)$ for $x \in \Omega$ is called the Robin function and it plays a crucial role for our problem.
\begin{thm}\label{thm-m-location}
Assume that $0 < s < 1$ and $n > 4s$.
Suppose $x_0 \in \Omega$ is a point given by Theorem \ref{thm-m-limit}. Then

\noindent (1) $x_0$ is a critical point of the function $\tau(x)$.

\noindent (2) It holds that
\[\lim_{\epsilon \rightarrow 0} \epsilon \| u_{\epsilon}\|_{L^{\infty}(\Omega)}^{2 \frac{n-4s}{n-2s}} = \mathfrak{d}_{n,s} |\tau(x_0)|\]
where the constant $\mathfrak{d}_{n,s}$ is computed in Section \ref{sec_blowup} (see \eqref{eq-J}).
\end{thm}
\noindent These two results are motivated by the work of Han \cite{H} and Rey \cite{R} on the classical local Brezis-Nirenberg problem, which dates back to Brezis and Peletier \cite{BP},
\begin{equation}\label{uun2n}
\left\{ \begin{array}{ll}
-\Delta u = u^{\frac{n+2}{n-2}} +{\epsilon}u &\text{in}~ \Omega,\\
u>0 & \text{in}~ \Omega,\\
u=0& \text{on}~\partial \Omega.
\end{array}\right.
\end{equation}

\medskip
On the other hand, in the latter part of his paper, Rey \cite{R} constructed a family of solutions for \eqref{uun2n} which asymptotically blow up at a nondegenerate critical point of the Robin function.
Moreover, this result was extended in \cite{MP}, where Musso and Pistoia obtained the existence of multi-peak solutions for certain domains.
In the second part of our paper, by employing the Lyapunov-Schmidt reduction method, we prove an analogous result to it for the nonlocal problem \eqref{a12uu}.
\begin{thm}\label{thm-m-onepeak}
Suppose that $0 < s < 1$ and $n > 4s$.
Let $\Lambda_1 \subset \Omega$ be a stable critical set of the Robin function $\tau$. Then, for small $\epsilon > 0$, there exists a family of solutions of \eqref{a12uu} which blow up and concentrate at the point $x_0 \in \Lambda_1$ as $\epsilon \to 0$.
\end{thm}

\noindent This result is an immediate consequence of the following result. Given any $k \in \mathbb{N}$, set
\begin{equation}\label{upsilon_brezis}
\Upsilon_k(\bls) = c_1^2 \left( \sum_{i=1}^k H (\sigma_i, \sigma_i) \lambda_i^{n-2s} -\sum_{\substack{i, h=1 \\ i \neq h}}^{k} G(\sigma_i, \sigma_h) (\lambda_i \lambda_h)^{\frac{n-2s}{2}} \right) - c_2\sum_{i=1}^k \lambda_i^{2s}
\end{equation}
for $(\bls) = (\lambda_1, \cdots, \lambda_k, \sigma_1, \cdots, \sigma_k) \in (0,\infty)^k \times \Omega^k$, where
\begin{equation}\label{constant_AB}
c_1 = \int_{\mathbb{R}^n} w_{1,0}^p(x)dx \quad \text{and} \quad c_2 = \int_{\mathbb{R}^n} w_{1,0}^2(x)dx
\end{equation}
with $w_{1,0}$ the function defined in \eqref{bubble} with $(\lambda, \xi) = (1,0)$. Then we have
\begin{thm}\label{thm-m-multipeak}
Assume $0 < s < 1$ and $n > 4s$.
Given $k \in \mathbb{N}$, suppose that $\Upsilon_k$ has a stable critical set $\Lambda_k$
such that
\[\Lambda_k \subset \left\{((\lambda_1, \cdots, \lambda_k), (\sigma_1,\cdots,\sigma_k)) \in (0,\infty)^k \times \Omega^k : \sigma_i \ne \sigma_j \text{ if } i \ne j \text{ and } i, j = 1, \cdots, k\right\}.\]
Then there exist a point $((\lambda_1^0, \cdots, \lambda_k^0),(\sigma^0_1,\cdots,\sigma^0_k)) \in \Lambda_k$ and a small number $\epsilon_0 > 0$ such that for $0 < \epsilon < \epsilon_0$,
there is a family of solutions $u_{\epsilon}$ of \eqref{a12uu} which concentrate at each point $\sigma^0_1, \cdots, \sigma^0_{k-1}$ and $\sigma^0_k$ as $\epsilon \to 0$.
\end{thm}

\noindent For the precise description of the asymptotic behavior of $u_{\epsilon}$, see the proof of Theorem \ref{thm-m-multipeak} in Subsection \ref{subsec_proof_reduction}.

Here we borrowed the notion of stable critical sets from \cite{Li2}.
As in the case $s = 1$ (see \cite{MP, EGP} for instance), we can prove that if the domain $\Omega$ is a dumbbell-shaped domain which consists of disjoint $k$-open sets and sufficiently narrow channels connecting them,
then $\Upsilon_k$ has a stable critical point for each $k \in \mathbb{N}$, thereby obtaining the following result.
\begin{thm}\label{domain-construction}
There exist contractible domains $\Omega$ such that, for $\epsilon > 0$ small enough, \eqref{a12uu} possesses a family of solutions which blow up at exactly $k$ different points of each domain $\Omega$ as $\epsilon$ converges to 0.
\end{thm}
\noindent For the detailed explanation, see Section \ref{sec_reduction}.

\medskip
In order to study the asymptotic behavior, we will use the fundamental observation of Caffarelli and Silvestre \cite{CS} and Cabr\'e and Tan \cite{CT} (see also \cite{ST, CDDS, BCPS1, T2}). In particular, we study the local problem on a half-cylinder $\mathcal{C}:= \Omega \times [0, \infty)$,
\begin{equation}\label{u0inc}
\left\{ \begin{array}{ll} \text{div}(t^{1-2s} \nabla U )= 0&\quad \text{in}~ \mathcal{C} = \Omega \times (0,\infty),\\
U>0 &\quad \text{in} ~\mathcal{C},\\
U = 0 &\quad \text{on} ~\partial_L \mathcal{C} : = \partial \Omega \times (0, \infty),\\
\partial_{\nu}^{s} { U} = f(U) & \quad \text{on} ~\Omega \times \{0\},
\end{array}
\right.
\end{equation}
where $\nu$ is the outward unit normal vector to $\mathcal{C}$ on $\Omega \times \{0\}$ and
\begin{equation}\label{pns}
\partial_{\nu}^{s}U(x,0):= -C_s^{-1} \left(\lim_{t \rightarrow 0+} t^{1-2s} \frac{\partial U}{\partial t}(x,t)\right) \quad \text{for } x \in \Omega
\end{equation}
where
\begin{equation}\label{cs}
C_s:=\frac{2^{1-2s}\Gamma(1-s)}{\Gamma(s)}
\end{equation}
Under appropriate regularity assumptions, the trace of a solution $U$ of \eqref{u0inc} on $\Omega \times \{0\}$ solves the nonlinear problem \eqref{a12uu}.

A key step of the proof for Theorem \ref{thm-m-limit} is to get a sort of the uniform bound after rescaling the solutions $\{ u_{\epsilon}: \epsilon>0\}$.
For this purpose, we will establish a priori $L^{\infty}$-estimates by using the Moser iteration argument.
Recently, such type of estimates have been established in \cite{GQ, TX, XY}.
However, they cannot be applied to our case directly, so we will derive a result which is adequate in our setting (refer to Lemmas \ref{lem-harnack-1} and \ref{lem-harnack-2}).
We remark that a similar argument to our proof appeared in \cite{GQ}.
One more thing which has to be stressed is that we need a bound of $\|u_{\epsilon}\|_{L^{\infty}}$ in terms of a certain negative power of $\epsilon >0$ (Lemma \ref{eq-mu-epsilon})
to apply the elliptic estimates (Lemma \ref{lem-harnack-2}).
For this, we will use an inequality which comes from a local version of Pohozaev identity on the extended domain (see Proposition \ref{prop-sub-estimate}).
We refer to Section \ref{sec_asymp} for the details.

\medskip
We also study problems having nonlinearities of slightly subcritical growth
\begin{equation}\label{eqtn-subcritical}
\left\{\begin{array}{ll} \mathcal{A}_{s} u = u^{p-\epsilon} &\text{in}~\Omega,
\\
u>0 &\text{in} ~\Omega,
\\
u=0 & \text{on}~\partial \Omega.
\end{array}\right.
\end{equation}
In particular, the following two theorems will be obtained.
\begin{thm}\label{thm-m-limit-sub}
Assume that $0 < s < 1$ and $n > 2s$.
For $\epsilon>0$, let $u_{\epsilon}$ be a solution of \eqref{eqtn-subcritical} satisfying \eqref{2p1s0}.
Then, there exist a point $x_0 \in \Omega$ and a constant $\mathfrak{b}_{n,s} > 0$ such that
\[u_{\epsilon} \rightarrow 0 \text{ in } \left\{\begin{array}{lll} C^{\alpha}_{\textrm{loc}}(\Omega \setminus \{x_0\})&\ \text{for all }  \alpha \in (0,2s) & \textrm{if } s \in (0,1/2],
\\
C^{1,\alpha}_{\textrm{loc}}(\Omega \setminus \{x_0\})&\ \text{for all } \alpha
\in (0,2s-1) & \textrm{if } s \in (1/2,1),
\end{array}
\right.\]
and
\[\| u_{\epsilon} (x) \|_{L^{\infty}} u_{\epsilon} (x) \rightarrow \mathfrak{b}_{n,s} G(x,{x_0}) \textrm{ in } \left\{\begin{array}{lll} C^{\alpha}_{\textrm{loc}}(\Omega \setminus \{x_0\}) &\ \text{for all } \alpha \in (0,2s) & \textrm{if } s \in (0,1/2],
\\
C^{1,\alpha}_{\textrm{loc}}(\Omega \setminus \{x_0\})&\ \text{for all } \alpha
\in (0,2s-1) & \textrm{if } s \in (1/2,1),
\end{array}
\right.\]
as $\epsilon \to 0$. Moreover,

\noindent (1) $x_0$ is a critical point of the function $\tau(x)$.

\noindent (2) We have
\[\lim_{\epsilon \rightarrow 0} \epsilon \| u_{\epsilon}\|_{L^{\infty}(\Omega)}^{2} = \mathfrak{g}_{n,s} |\tau(x_0)|.\]
Here $\mathfrak{b}_{n,s}$ is the same constant to one given in Theorem \ref{thm-m-limit} and $\mathfrak{g}_{n,s}$ is computed in Section \ref{sec_subcrit} (see \eqref{eq-constant-L}).
\end{thm}

Like \eqref{upsilon_brezis}, we define
\begin{equation}\label{upsilon_subcritical}
\widetilde{\Upsilon}_{k}(\bls) = c_1^2 \left( \sum_{i=1}^k H (\sigma_i, \sigma_i) \lambda_i^{n-2s} -\sum_{\substack{i, h=1 \\ i \neq h}}^{k} G(\sigma_i, \sigma_h) (\lambda_i \lambda_h)^{\frac{n-2s}{2}} \right) - {c_1(n-2s)^2 \over 4n} \log ( \lambda_1 \cdots \lambda_k)
\end{equation}
for $(\bls) = (\lambda_1, \cdots, \lambda_k, \sigma_1, \cdots, \sigma_k) \in (0,\infty)^k \times \Omega^k$, where $c_1 > 0$ is defined in \eqref{constant_AB}.
Then we have
\begin{thm}\label{thm-m-multipeak2}
Assume $0 < s < 1$ and $n > 2s$.
Given $k \in \mathbb{N}$, suppose that $\widetilde{\Upsilon}_k$ has a stable critical set $\Lambda_k$
such that
\[\Lambda_k \subset \left\{((\lambda_1, \cdots, \lambda_k), (\sigma_1,\cdots,\sigma_k)) \in (0,\infty)^k \times \Omega^k : \sigma_i \ne \sigma_j \text{ if } i \ne j \text{ and } i, j = 1, \cdots, k\right\}.\]
Then there exist a point $((\lambda_1^0, \cdots, \lambda_k^0),(\sigma^0_1,\cdots,\sigma^0_k)) \in \Lambda_k$ and a small number $\epsilon_0 > 0$ such that for $0 < \epsilon < \epsilon_0$,
there is a family of solutions $u_{\epsilon}$ of \eqref{eqtn-subcritical} which concentrate at each point $\sigma^0_1, \cdots, \sigma^0_{k-1}$ and $\sigma^0_k$ as $\epsilon \to 0$.
\end{thm}

\noindent Most of the steps in the proof for Theorem \ref{thm-m-limit} and Theorem \ref{thm-m-location} can be adapted in proving Theorem \ref{thm-m-limit-sub}.
However the order of the proof for Theorem \ref{thm-m-limit-sub} is different from that of previous theorems and some new observations have to be made.
We refer to Section \ref{sec_subcrit} for the details.

Regarding Theorem \ref{thm-m-limit-sub}, it would be interesting to consider whether we can obtain a further description on the asymptotic behavior of a least energy solution of \eqref{eqtn-subcritical} (i.e. a solution satisfying \eqref{2p1s0}) as in \cite{FW},
where Flucher and Wei found that a least energy solution concentrates at a minimum of the Robin function in the local case ($s = 1$).

Moreover, we believe that even in the nonlocal case ($s \in (0,1)$)
there exist solutions of \eqref{eqtn-subcritical} (with the nonlinearity changed into $|u|^{p-1-\epsilon}u$) which can be characterized as sign-changing towers of bubbles.
See the papers e.g. \cite{DDM, PW, MP2, GMP} which studied the existence of bubble-towers for the related local problems.

\medskip
Before concluding this introduction, we would like to mention some related results to our problem.
In \cite{DDW}, the authors took into account the singularly perturbed nonlinear Schr\"{o}dinger equations
\begin{equation}\label{nse}
\begin{cases}
\epsilon^{2s}\mathcal{A}_s u + V u - u^p = 0 &\text{in } \mathbb{R}^n, \\
u > 0 &\text{in } \mathbb{R}^n,\\
u \in H^{2s}(\mathbb{R}^n) &
\end{cases}
\end{equation}
where $\epsilon > 0$ is sufficiently small, $0<s<1$, $p \in (1, {n+2s \over n-2s})$ and $V$ is a positive bounded $C^{1,\alpha}$ function whose value is away from 0.
In particular, employing the nondegeneracy result of \cite{FLS}, they deduced the existence of various types of spike solutions, like multiple spikes and clusters, such that each of the local maxima concentrates on a critical point of $V$.
See also the result of \cite{CZ} in which a single peak solution is found under stronger assumptions on \eqref{nse} than those of \cite{DDW} (in particular, it is assumed that $s \in (\max\{{1 \over 2}, {n \over 4}\}, 1)$ in \cite{CZ}).
As far as we know, these works are the first results to investigate concentration phenomena for singularly perturbed equations with the fractional operator $\mathcal{A}_s$ by utilizing the Lyapunov-Schmidt reduction method.

On the other hand, in \cite{SV} and \cite{SV2}, the Brezis-Nirenberg problem is also considered when the fractional Laplace operator is defined as in a different way:
\[(-\Delta)^su(x) = c_{n,s} P.V. \int_{\mathbb{R}^n} {u(x) - u(y) \over |x-y|^{n+2s}} dy \quad \text{for } x \in \Omega\]
where $\Omega$ is bounded and $c_{n,s}$ is a normalization constant.
(Here, we refer to an interesting paper \cite{MN} which compares two different notions of the fractional Laplacians.)
It turns out that a similar result can be deduced to one in \cite{T} and \cite{BCPS2}, the papers aforementioned in this introduction.
In this point of view, it would be interesting to obtain results for this operator corresponding to ours.
As a matter of fact, we suspect that concentration points of solutions for \eqref{a12uu} and \eqref{eqtn-subcritical} are governed by Green's function of the operator in this case too.

\medskip
This paper is organized as follows.
In section \ref{sec_prelim}, we review certain notions related to the fractional Laplacian and study the regularity of Green's function of $\mathcal{A}_{s}$.
Section \ref{sec_asymp} is devoted to prove Theorem \ref{thm-m-limit}.
In section \ref{sec_blowup}, we show Theorem \ref{thm-m-location} by finding some estimates for Green's function.
In Section \ref{sec_reduction}, multi-peak solutions is constructed by the Lyapunov-Schmidt reduction method, giving the proof of Theorem \ref{thm-m-multipeak} and Theorem \ref{domain-construction}.
On the other hand, the Lane-Emden equation \eqref{eqtn-subcritical} whose nonlinearity has slightly subcritical growth is considered in Section \ref{sec_subcrit}, and the proof of Theorem \ref{thm-m-limit-sub} and Theorem \ref{thm-m-multipeak2} is presented there.
In Appendix \ref{sec_appen_a} and Appendix \ref{sec_appen_b}, we give the proof of Proposition \ref{prop-sub-estimate} and \eqref{eq-main-limit}, respectively,
while we exhibit some necessary computations for the construction of concentrating solutions in Appendix \ref{sec_appen_c}.

\bigskip
\noindent \textbf{Notations.}

\medskip
\noindent Here we list some notations which will be used throughout the paper.

\noindent - The letter $z$ represents a variable in the $\mathbb{R}^{n+1}$. Also, it is written as $z = (x,t)$ with $x \in \mathbb{R}^n$ and $t \in \mathbb{R}$.

\noindent - Suppose that a domain $D$ is given and $\mathcal{T} \subset \partial D$. If $f$ is a function on $D$, then the trace of $f$ on $\mathcal{T}$ is denoted by $\text{tr}|_\mathcal{T} f$ whenever it is well-defined.

\noindent - For a domain $D \subset \mathbb{R}^n$, the map $\nu = (\nu_1, \cdots, \nu_n): \partial D \to \mathbb{R}^n$ denotes the outward pointing unit normal vector on $\partial D$.

\noindent - $dS$ stands for the surface measure. Also, a subscript attached to $dS$ (such as $dS_x$ or $dS_z$) denotes the variable of the surface.

\noindent - $|S^{n-1}| = 2\pi^{n/2}/\Gamma(n/2)$ denotes the Lebesgue measure of $(n-1)$-dimensional unit sphere $S^{n-1}$.

\noindent - For a function $f$, we set $f_+ = \max\{f, 0\}$ and $f_- = \max\{-f, 0\}$.

\noindent - Given a function $f = f(x)$, $\nabla_x f$ means the gradient of $f$ with respect to the variable $x$.

\noindent - We will use big $O$ and small $o$ notations to describe the limit behavior of a certain quantity as $\epsilon \to 0$.

\noindent - $C > 0$ is a generic constant that may vary from line to line.

\noindent - For $ k \in \mathbb{N}$, we denote by $B_{k} (x_0,r)$ the  ball $\{ x \in \mathbb{R}^{k}: |x-x_0| < r \}$ for each $x_0 \in \mathbb{R}^k$ and $r>0$.

\section{Preliminaries}\label{sec_prelim}
In this section we first recall the backgrounds of the fractional Laplacian. We refer to \cite{BCPS1, CT, CS, CDDS, T2, KL} for the details.
In particular, the latter part of this section is devoted to prove a $C^{\infty}$ regularity property of Green's function for the fractional Laplacian with zero Dirichlet boundary condition.

\subsection{Fractional Sobolev spaces, fractional Laplacians and $s$-harmonic extensions}\label{subsec_frac_Sob}
Let $\Omega$ be a smooth bounded domain of $\mathbb{R}^n$.
Let also $\{ \lambda_k, \phi_k\}_{k=1}^{\infty}$ be a sequence of the eigenvalues and corresponding eigenvectors of the Laplacian operator $-\Delta$ in $\Omega$ with the zero Dirichlet boundary condition on $\partial \Omega$,
\[\left\{ \begin{array}{ll}
- \Delta \phi_k = \lambda_k \phi_k &\text{in}~ \Omega,\\
\phi_k = 0 &\text{on}~ \partial \Omega,\\
\end{array}\right.\]
such that $\|\phi_k\|_{L^2(\Omega)} = 1$ and $\lambda_1 < \lambda_2 \le \lambda_3 \le \cdots$.
Then we set the fractional Sobolev space $H_0^s (\Omega)$ $(0 < s < 1)$ by
\begin{equation}\label{H_0^s}
H_0^s (\Omega) = \left\{ u = \sum_{k=1}^{\infty} a_k \phi_k \in L^2 (\Omega) : \sum_{k=1}^{\infty} a_k^2 \lambda_k^{s} < \infty \right\},
\end{equation}
which is a Hilbert space whose inner product is given by
\[\left\langle \sum_{k=1}^{\infty} a_k \phi_k, \sum_{k=1}^{\infty} b_k \phi_k \right\rangle_{H_0^s(\Omega)} = \sum_{k=1}^{\infty} a_k b_k\lambda_k^s \qquad \text{if } \sum_{k=1}^{\infty} a_k \phi_k,\ \sum_{k=1}^{\infty} b_k \phi_k \in H_0^s(\Omega). \]
Moreover, for a function in $H_0^{s}(\Omega)$, we define the fractional Laplacian $\mathcal{A}_{s} : H_0^s (\Omega) \rightarrow H_0^s (\Omega) \simeq H_0^{-s} (\Omega)$ as
\[\mathcal{A}_{s} \left ( \sum_{k=1}^{\infty} a_k \phi_k \right) = \sum_{k=1}^{\infty} a_k \lambda_k^s \phi_k.\]
We also consider the square root $\mathcal{A}_{s}^{1/2}: H_0^s(\Omega) \to L^2(\Omega)$ of the positive operator $\mathcal{A}_{s}$ which is in fact equal to $\mathcal{A}_{s/2}$. Note that by the above definitions, we have
\[\left\langle u, v \right\rangle_{H_0^s(\Omega)} = \int_{\Omega} \mathcal{A}_{s}^{1/2}u \cdot \mathcal{A}_{s}^{1/2}v = \int_{\Omega} \mathcal{A}_{s} u \cdot v \quad \text{for } u, v \in H_0^s(\Omega). \]

\medskip
If the domain $\Omega$ is the whole space $\mathbb{R}^n$, the space $H^s(\mathbb{R}^n)$ $(0<s<1)$ is given as
\[H^s(\mathbb{R}^n) = \left\{u \in L^2(\mathbb{R}^n): \|u\|_{H^s(\mathbb{R}^n)} := \left(\int_{\mathbb{R}^n} (1+|2\pi\xi|^{2s}) |\hat{u}(\xi)|^2 d\xi\right)^{1 \over 2} < \infty \right\} \]
where $\hat{u}$ denotes the Fourier transform of $u$, and the fractional Laplacian $\mathcal{A}_{s}: H^s(\mathbb{R}^n) \to H^{-s}(\mathbb{R}^n)$ is defined to be
\[ \widehat{\mathcal{A}_{s} u}(\xi) = |2\pi\xi|^{2s} \hat{u}(\xi) \quad \text{for any } \xi \in \mathbb{R}^n  \text{ given } u \in H^s(\mathbb{R}^n).\]

\medskip
Regarding \eqref{u0inc} (see also \eqref{s-extension} below), we need to introduce some more function spaces on $\mathcal{C} = \Omega \times (0, \infty)$ where $\Omega$ is either a smooth bounded domain or $\mathbb{R}^n$.
If $\Omega$ is bounded, the function space $H_{0,L}^s(\mathcal{C})$ is defined as the completion of
\[ C_{c,L}^{\infty}(\mathcal{C})
:= \left\{ U \in C^{\infty}\left(\overline{\mathcal{C}}\right) : U = 0 \text{ on } \partial_L\mathcal{C} = \partial \Omega \times (0,\infty) \right\}\]
with respect to the norm
\begin{equation}\label{weighted_norm}
\|U\|_{\mathcal{C}} = \left( \int_{\mathcal{C}} t^{1-2s} |\nabla U|^2 \right)^{1 \over 2}.
\end{equation}
Then it is a Hilbert space endowed with the inner product
\[(U,V)_{\mathcal{C}} = \int_{\mathcal{C}} t^{1-2s} \nabla U \cdot \nabla V \quad \text{for} \quad U, \ V \in H_{0,L}^{s}(\mathcal{C}).\]
In the same manner, we define the space $H_{0,L}^{s}(\mathcal{C}_{\epsilon})$ and $C_{c,L}^{\infty}(\mathcal{C}_{\epsilon})$ for the dilated problem \eqref{equation-dilated}.
Moreover, $\mathcal{D}^s(\mathbb{R}^{n+1}_+)$ is defined as the completion of
$C_c^{\infty}\left(\overline{\mathbb{R}^{n+1}_+}\right)$ with respect to the norm $\|U\|_{\mathbb{R}^{n+1}_+}$ (defined by putting $\mathcal{C} = \mathbb{R}^{n+1}_+$ in \eqref{weighted_norm} above).
Recall that if $\Omega$ is a smooth bounded domain, it is verified that
\begin{equation}\label{eq_Sobo_trace}
H_0^s(\Omega) = \{u = \text{tr}|_{\Omega \times \{0\}}U: U \in H^s_{0,L}(\mathcal{C})  \}
\end{equation}
in \cite[Proposition 2.1]{CS} and \cite[Proposition 2.1]{CDDS} and \cite[Section 2]{T2}.
Furthermore, it holds that
\[\|U(\cdot, 0)\|_{H^s(\mathbb{R}^n)} \le C \|U\|_{\mathbb{R}^{n+1}_+}\]
for some $C > 0$ independent of $U \in \mathcal{D}^s(\mathbb{R}^{n+1}_+)$.

\medskip
Now we may consider the fractional harmonic extension of a function $u$ defined in $\Omega$, where $\Omega$ is $\mathbb{R}^n$ or a smooth bounded domain.
By the celebrated results of Caffarelli-Silvestre \cite{CS} (for $\mathbb{R}^n$) and Cabr\'e-Tan \cite{CT} (for bounded domains, see also \cite{ST, CDDS, BCPS1, T2}),
if we set $U \in H^s_{0,L}(\mathcal{C})$ (or $\mathcal{D}^s(\mathbb{R}^{n+1}_+)$) as a unique solution of the equation
\begin{equation}\label{s-extension}
\left\{ \begin{array}{ll} \text{div}(t^{1-2s} \nabla U) = 0 &~\text{in}~ \mathcal{C},\\
U = 0 & ~\text{on}~ \partial_L \mathcal{C},\\
U(x,0)= u(x) &~ \text{for}~ x \in \Omega,
\end{array}\right.
\end{equation}
for some fixed function $u \in H^s_0(\Omega)$ (or $H^s(\mathbb{R}^n)$),
then $\mathcal{A}_{s}u = \partial_{\nu}^s U|_{\Omega \times \{0\}}$ where the operator $u \mapsto \partial_{\nu}^s U|_{\Omega \times \{0\}}$ is defined in \eqref{pns}.
(If $\Omega = \mathbb{R}^n$, we set $\partial_L\mathcal{C} = \emptyset$.)
We call this $U$ the $s$-harmonic extension of $u$.
We remark that an \textit{explicit} description of $U$ is obtained in \cite{BCPS1, T2} if $\Omega$ is bounded.

\subsection{Sharp Sobolev and trace inequalities}\label{subsec_sobolev_trace}
Given any $\lambda > 0$ and $\xi \in \mathbb{R}^n$, let
\begin{equation}\label{bubble}
w_{\lambda, \xi} (x) = \mathfrak{c}_{n,s} \left( \frac{\lambda}{\lambda^2+|x-\xi|^2}\right)^{\frac{n-2s}{2}} \quad \text{for } x \in \mathbb{R}^n,
\end{equation}
where
\begin{equation}\label{cns}
\mathfrak{c}_{n,s} = 2^{\frac{n-2s}{2}} \left( \frac{\Gamma \left( \frac{n+2s}{2}\right)}{\Gamma \left( \frac{n-2s}{2}\right)}\right)^{\frac{n-2s}{4s}}.
\end{equation}
Then the sharp Sobolev inequality
\[\left( \int_{\mathbb{R}^n} |u|^{p+1} dx \right)^{\frac{1}{p+1}} \leq \mathcal{S}_{n,s} \left( \int_{\mathbb{R}^n} |\mathcal{A}_{s}^{1/2} u|^2 dx \right)^{1 \over 2}\]
gets the equality if and only if $u(x) = c w_{\lambda,\xi}(x)$ for any $c > 0,\ \lambda>0$ and $\xi \in \mathbb{R}^n$, given $\mathcal{S}_{n,s}$ the value defined in \eqref{ans} (refer to \cite{L2, CL, FL}).
Furthermore, it was shown in \cite{CLO, L1, L3} that if a suitable decay assumption is imposed, then $\{ w_{\lambda, \xi}(x): \lambda>0, \xi \in \mathbb{R}^n \}$ is the set of all solutions for the problem
\begin{equation}\label{entire_nonlocal}
\mathcal{A}_{s} u = u^p,\quad u > 0 \quad \text{in } \mathbb{R}^n\quad \textrm{and}\quad \lim_{|x|\rightarrow \infty} u(x) = 0.
\end{equation}
We use $W_{\lambda,\xi} \in \mathcal{D}^s(\mathbb{R}^{n+1}_+)$ to denote the (unique) $s$-harmonic extension of $w_{\lambda,\xi}$ so that $W_{\lambda, \xi}$ solves
\begin{equation}\label{wlyxt}
\left\{\begin{array}{ll}
\text{div}(t^{1-2s}W_{\lambda,\xi}(x,t)) = 0 &\quad \text{in } \mathbb{R}^{n+1}_+,\\
W_{\lambda,\xi}(x,0) = w_{\lambda,\xi}(x) &\quad \text{for } x \in \mathbb{R}^n.
\end{array}\right.
\end{equation}
It follows that for the Sobolev trace inequality
\begin{equation}\label{eq-sharp-trace}
\left( \int_{\mathbb{R}^n} |U(x,0)|^{p+1} dx \right)^{\frac{1}{p+1}} \leq \frac{\mathcal{S}_{n,s}}{\sqrt{C_s}} \left( \int_0^{\infty}\int_{\mathbb{R}^n} t^{1-2s} |\nabla U(x,t)|^2 dx dt \right)^{1 \over 2},
\end{equation}
the equality is attained by some function $U \in \mathcal{D}^s(\mathbb{R}^{n+1}_+)$ if and only if $U(x,t) = c W_{\lambda,\xi}(x,t)$ for any $c > 0,\ \lambda>0$ and $\xi \in \mathbb{R}^n$,
where $C_s > 0$ is the constant defined in \eqref{cs} (see \cite{X}).
In what follows, we simply denote $w_{1,0}$ and $W_{1,0}$ by $w_1$ and $W_1$, respectively.

\subsection{Green's functions and the Robin function}
Let $G$ be Green's function of the fractional Laplacian $\mathcal{A}_{s}$ with the zero Dirichlet boundary condition (see \eqref{Green_Omega}).
Then it can be regarded as the trace of Green's function $G_{\mathcal{C}} = G_{\mathcal{C}}(z,x)$ ($z \in \mathcal{C}$, $x \in \Omega$) for the extended Dirichlet-Neumann problem which satisfies
\begin{equation}\label{Green_extended}
\left\{ \begin{array}{ll} \text{div} (t^{1-2s} \nabla G_{\mathcal{C}}(\cdot, x)) =0&\quad \text{in}~ \mathcal{C},\\
G_{\mathcal{C}}(\cdot, x) = 0 &\quad \text{on}~ \partial_L \mathcal{C},\\
\partial_{\nu}^{s} G_{\mathcal{C}}(\cdot, x) = \delta_x &\quad \text{on}~ \Omega\times \{0\}.
\end{array}\right.
\end{equation}
In fact, if a function $U$ in $\mathcal{C}$ solves
\[\left\{ \begin{array}{ll} \text{div} (t^{1-2s} \nabla U) =0&\quad \text{in}~ \mathcal{C},\\
U = 0 &\quad \text{on}~ \partial_L \mathcal{C},\\
\partial_{\nu}^{s} U = g &\quad \text{on } \Omega\times \{0\},\\
\end{array}\right.\]
for some function $g$ on $\Omega \times \{0\}$, then we can see that $U$ has the expression
\[U(z) = \int_{\Omega} G_{\mathcal{C}} (z, y) g(y) dy = \int_{\Omega} G_{\mathcal{C}} (z,y) \mathcal{A}_{s}u(y)dy, \quad z \in \mathcal{C},\]
where $u = \text{tr}|_{\Omega \times \{0\}}U$.
Then, by plugging $z=(x,0)$ in the above equalities, we obtain
\[u(x) = \int_{\Omega} G_{\mathcal{C}} ((x,0),y) \mathcal{A}_{s} u(y) dy,\]
which implies that $G_{\mathcal{C}} ((x,0),y)= G(x,y)$ for any $x, y \in \Omega$.

Green's function $G_{\mathcal{C}}$ on the half cylinder $\mathcal{C}$ can be partitioned to the singular part and the regular part.
The singular part is given by Green's function
\begin{equation}\label{Green_half}
G_{\mathbb{R}^{n+1}_{+}}((x,t),y) := \frac{ {{\mathfrak{a}}_{n,s}}}{|(x-y,t)|^{n-2s}}
\end{equation}
on the half space $\mathbb{R}^{n+1}_{+}$ satisfying
\[\left\{ \begin{array}{ll} \text{div} \left(t^{1-2s} \nabla_{(x,t)} G_{\mathbb{R}^{n+1}_{+}} ((x,t),y)\right) = 0 &\quad \text{in}~ \mathbb{R}^{n+1}_{+},\\
\partial_{\nu}^{s} G_{\mathbb{R}^{n+1}_{+}} ((x,0),y)= \delta_{y}(x)&\quad \text{on}~ \Omega \times \{0\},
\end{array}
\right.\]
for each $y \in \mathbb{R}^n$.
Note that ${\mathfrak{a}}_{n,s}$ is the constant defined in \eqref{Robin_Omega}.
The regular part is given by the function $H_{\mathcal{C}} : \mathcal{C} \to \mathbb{R}$ which satisfies
\[\left\{ \begin{array}{ll}
\text{div}\left(t^{1-2s} \nabla_{(x,t)}H_{\mathcal{C}}((x,t),y) \right)= 0 &\quad \text{in}~\mathcal{C},\\
H_{\mathcal{C}}(x,t,y) = \frac{\mathfrak{a}_{n,s}}{|(x-y,t)|^{n-2s}} &\quad \text{on}~ \partial_L \mathcal{C},\\
\partial_{\nu}^{s}H_{\mathcal{C}} ((x,0),y)=0&\quad\text{on}~ \Omega \times \{0\}.
\end{array}
\right.\]
The existence of such a function $H_{\mathcal{C}}$ can be proved using a variational method (see Lemma \ref{lem-H-existence} below).
We then have
\begin{equation}\label{eq-green-decompose}
G_{\mathcal{C}}((x,t),y) = G_{\mathbb{R}^{n+1}_{+}} ((x,t),y) - H_{\mathcal{C}}((x,t),y).
\end{equation}
Accordingly, the Robin function $\tau$ which was defined in the paragraph after Theorem \ref{thm-m-limit} can be written as $\tau (x):= H_{\mathcal{C}}((x,0),x)$.
As we will see, the function $\tau$ and the relation \eqref{eq-green-decompose} turn out to be very important throughout the paper.

\subsection{Maximum principle}
Here we prove a maximum principle which serves as a valuable tool in studying properties of Green's function $G$ of $\mathcal{A}_s$.
\begin{lem}\label{lem-maximum}
Suppose that $V$ is a weak solution of the following problem
\[\left\{ \begin{array}{ll} \textnormal{div}(t^{1-2s} \nabla V) = 0&\quad \textnormal{in}~\mathcal{C},
\\
V(x,t) = B(x,t)&\quad \textnormal{on}~\partial_L\mathcal{C},
\\
 \partial_{\nu}^{s} V(x,0) =0 &\quad \textnormal{on}~\Omega \times\{0\}.
\end{array}
\right.\]
for some function $B$ on $\partial_L\mathcal{C}$.
Then we have
\[\sup_{(x,t)\in \mathcal{C}} |V(x,t)| \leq \sup_{(x,t) \in \partial_{L} \mathcal{C}} |B(x,t)|.\]
\end{lem}
\begin{proof}
Let $S^{+} = \sup_{(x,t) \in \partial_L \mathcal{C}} B(x,t)$. Consider the function $Y(x,t):= S^{+} - V(x,t)$, which satisfies
\[\left\{ \begin{array}{ll} \textrm{div}(t^{1-2s} \nabla Y) = 0&\quad \textrm{in}~\mathcal{C},
\\
Y(x,t) \geq 0 &\quad \textrm{on}~\partial_L\mathcal{C}
\\
 \partial_{\nu}^{s} Y(x,0) =0 &\quad \textrm{on}~\Omega \times\{0\}.
\end{array}
\right.\]
Note that $Y_{-}(x,t) = 0$ on $\partial_L \mathcal{C}$. Then we get
\[0 = \int_{\mathcal{C}} t^{1-2s} \nabla Y(x,t) \cdot \nabla Y_{-}(x,t) dx dt = -\int_{\mathcal{C}}t^{1-2s} |\nabla Y_{-}(x,t)|^2 dx dt.\]
It proves that $Y_{-} \equiv 0$. Thus we have $S^{+}\geq V(x,t)$ for all $(x,t) \in \mathcal{C}$.

Similarly, if we set $S^{-} = \inf_{(x,t) \in \partial_L \mathcal{C}} B(x,t)$ and define the function $Z(x,t)= V(x,t)-S^{-}$, we may deduce that $V(x,t) \geq S_{-}$ for all $(x,t) \in \mathcal{C}$. Consequently, we have
\[S^{-} \leq V(x,t) \leq S^{+}\quad \text{for all } (x,t) \in \mathcal{C}.\]
It completes the proof.
\end{proof}

\subsection{Properties of the Robin function}
We study more on the property of the function $H_{\mathcal{C}}$ by using the maximum principle obtained in the previous subsection.
We first prove the existence of the function $H_{\mathcal{C}}$.
\begin{lem}\label{lem-H-existence}
For each point $y \in \Omega$ the function $H_{\mathcal{C}}((\cdot, \cdot), y)$ is the minimizer of the problem
\begin{equation}\label{eq-mini}
\min_{V \in S} \int_{\mathcal{C}} t^{1-2s} |\nabla V(x,t)|^2 dx dt,
\end{equation}
where
\[S = \left\{ V : \int_{\mathcal{C}} t^{1-2s} |\nabla V(x,t)|^2 dx dt < \infty\ \quad\textrm{and}\quad V(x,t) = G_{\mathbb{R}^{n+1}_{+}}(x,t,y) \textrm{ on } \partial_L \mathcal{C} \right\}.\]
Here the derivatives are defined in a weak sense.
\end{lem}
\begin{proof}
Let $\eta \in C^{\infty} (\mathbb{R}^{n+1})$ be a function such that $\eta(z) =0$ for $|z| \leq 1$ and $\eta (z) =1$ for $|z| \geq 2$.
Assuming without loss of any generality that $B_{n+1}((y,0),2) \cap \mathbb{R}^{n+1}_+ \subset \mathcal{C}$, let $V_{0}$ be the function defined in $\mathcal{C}$ by
\[V_{0}(x,t) = G_{\mathbb{R}^{n+1}_{+}}(x,t,y) \eta (x-y,t).\]
Then it is easy to check that
\[\int_{\mathcal{C}} t^{1-2s} |\nabla V_0 (x,t)|^2 dx dt < \infty.\]
Thus $S$ is nonempty and we can find a minimizing function $V$ of the problem \eqref{eq-mini} in $S$. Then, for any $\Phi \in C^{\infty}(\mathcal{C})$ such that $\Phi = 0 $ on $\partial_L \mathcal{C}$, we have
\[\int_{\mathcal{C}} t^{1-2s} \nabla V(x,t) \cdot \nabla \Phi (x,t) dx dt = 0.\]
Hence it holds that
\[\left\{ \begin{array}{ll} \textrm{div}(t^{1-2s} \nabla V(x,t)) =0 &\quad \text{for } (x,t) \in \mathcal{C},
\\
V(x,t) = G_{\mathbb{R}^{n+1}_{+}}(x,t,y) &\quad \text{for } (x,t) \in \partial_L \mathcal{C},
\\
\partial_{\nu}^{s} V(x,0) = 0& \quad \text{for } x \in \Omega.
\end{array}
\right.\]
in a weak sense. This completes the proof.
\end{proof}
In the same way, for a fixed point $y = (y^1, \cdots, y^n) \in \Omega$ and any multi-index $I= (i_1, i_2,\cdots, i_n) \in (\mathbb{N} \cup \{0\})^n$, we find the function  $\mathcal{H}_{\mathcal{C}}^{I}((\cdot, \cdot), y)$ satisfying
\begin{equation}\label{eq-derive-h}
\left\{ \begin{array}{ll}
\text{div}\left(t^{1-2s} \nabla_{(x,t)}\mathcal{H}^{I}_{\mathcal{C}}(x,t,y) \right)= 0 &\quad \text{in}~\mathcal{C},\\
\mathcal{H}^{I}_{\mathcal{C}}(x,t,y) = \partial_y^{I} G_{\mathbb{R}^{n+1}_{+}}(x,t,y) &\quad \text{on}~ \partial_L \mathcal{C},\\
\partial_{\nu}^{s}\mathcal{H}^{I}_{\mathcal{C}} (\cdot,\cdot,y)=0&\quad\text{on}~ \Omega \times \{0\},
\end{array}
\right.
\end{equation}
where $\partial_y^I = \partial_{y^1}^{i_1}\cdots\partial_{y^n}^{i_n}$.
In the below we shall show that, for any $(x,t) \in \mathcal{C}$, the function $H_{\mathcal{C}}(x,t,y)$ is $C_{\textrm{loc}}^{\infty}(\Omega)$ and that $\partial_{y}^{I} H_{\mathcal{C}}(x,t,y) = \mathcal{H}_{\mathcal{C}}^{I} (x,t,y)$.

\begin{lem}
For each $(x,t) \in \mathcal{C}$ the function $H_{\mathcal{C}}(x,t,y)$ is continuous with respect to $y$.
Moreover, such continuity is uniform on $(x,t,y) \in \mathcal{C}\times \mathcal{K}$ for any compact subset $\mathcal{K}$ of $\Omega$.
\end{lem}
\begin{proof}
Take points $y_1$ and $y_2$ in a compact subset $\mathcal{K}$ of $\Omega$, sufficiently close to each other.
If we apply Lemma \ref{lem-maximum} to the function $H_{\mathcal{C}}(x,t,y_1) - H_{\mathcal{C}}(x,t,y_2)$, then we get
\begin{align*}
\sup_{(x,t) \in \mathcal{C}}|H_{\mathcal{C}}(x,t,y_1) - H_{\mathcal{C}}(x,t,y_2)| &\leq \sup_{(x,t) \in \partial_L\mathcal{C}} |H_{\mathcal{C}}(x,t,y_1)- H_{\mathcal{C}}(x,t,y_2)|
\\
&= \sup_{(x,t) \in \partial_L\mathcal{C}} \left|\frac{\mathfrak{c}_{n,s}}{|(x-y_1, t)|^{n-2s}}- \frac{\mathfrak{c}_{n,s}}{|(x-y_2,t)|^{n-2s}}\right|
\\
&\le C(\mathcal{K})|y_1 -y_2|,
\end{align*}
where $C(\mathcal{K}) > 0$ is constant relying only on $\mathcal{K}$.
It proves the lemma.
\end{proof}

The next lemma provides a regularity property of the function $H_{\mathcal{C}}$.
We recall that the result of Fabes, Kenig, and Serapioni \cite{FKS} which gives that $(x,t,y) \mapsto H_{\mathcal{C}}(x,t,y)$ is $C^{\alpha}$ for some $0<\alpha<1$.

\begin{lem}\label{lem-property-H}
(1) For each $(x,t) \in \mathcal{C}$, the function $y \rightarrow H_{\mathcal{C}}(x,t,y)$ is a $C^{\infty}$ function. Moreover, for each multi-index $I \in (\mathbb{N} \cup \{0\})^{n}$, we have
\begin{equation}\label{eq-derivative-H}
\partial_{y}^{I} H_{\mathcal{C}}(x,t,y) = \mathcal{H}_{\mathcal{C}}^{I} (x,t,y)
\end{equation}
and $\partial_y^{I}H_{\mathcal{C}}(x,t,y)$ is bounded on $(x,t,y) \in \mathcal{C}\times \mathcal{K}$ for any compact set $\mathcal{K}$ of $\Omega$.

\noindent (2) For each $y \in \Omega$, the function $x \in \Omega \mapsto H_{\mathcal{C}}(x,0,y)$ is a $C^{\infty}$ function.
Moreover, for each multi-index $I \in (\mathbb{N} \cup \{0\})^n$, the derivative $\partial_x^{I} H_{\mathcal{C}}(x,0,y)$ is bounded on  $(x,y) \in  \mathcal{K}\times \Omega$ for any compact set $\mathcal{K}$ of $\Omega$.
\end{lem}
\begin{proof}
For two points $y_1$ and $y_2$ in a compact subset $\mathcal{K}$ of $\Omega$ chosen to be close enough to each other, we apply Lemma \ref{lem-maximum} to the function
\[H_{\mathcal{C}}(x,t,y_2) - H_{\mathcal{C}}(x,t,y_1) - (y_2 -y_1) \cdot (\mathcal{H}^{I_1}_{\mathcal{C}}, \cdots, \mathcal{H}^{I_n}_{\mathcal{C}})(x,t, y_1)\]
where $I_j$ is the multi-index in $(\mathbb{N} \cup \{0\})^n$ such that the $j$-th coordinate is 1 and the other coordinates are 0 for $1 \le j \le n$. Then we obtain
\begin{align*}
&\ \sup_{(x,t)\in\mathcal{C}}\left|H_{\mathcal{C}}(x,t,y_2) - H_{\mathcal{C}}(x,t,y_1) - (y_2 - y_1) \cdot (\mathcal{H}^{I_1}_{\mathcal{C}}, \cdots, \mathcal{H}^{I_n}_{\mathcal{C}})(x,t, y_1) \right|
\\
&\leq \sup_{(x,t) \in \partial_L \mathcal{C}}\left|\frac{\mathfrak{c}_{n,s}}{|(x-y_2, t)|^{n-2s}} - \frac{\mathfrak{c}_{n,s}}{|(x-y_1, t)|^{n-2s}} - (y_2 - y_1) \cdot \frac{\mathfrak{c}_{n,s}(n-2s)(x-y_1)}{|(x-y_1, t)|^{n-2s+2}}\right| \\
&\le C(\mathcal{K}) |y_1 -y_2|^2
\end{align*}
for some $C(\mathcal{K}) > 0$ independent of the choice of $y_1$ and $y_2$.
This shows that
$\nabla_{y} H_{\mathcal{C}}(x,t,y) = (\mathcal{H}^{I_1}_{\mathcal{C}}, \cdots, \mathcal{H}^{I_n}_{\mathcal{C}})(x,t,y)$ proving \eqref{eq-derivative-H} for $|I|=1$.
We can adapt this argument inductively, which proves the first statement of the lemma.

Since $H_{\mathcal{C}}(x,0,y) = H_{\mathcal{C}}(y,0,x)$ holds for any $(x,y) \in \Omega\times \Omega$, the second statement follows directly from the first statement.
\end{proof}
Given the above results, we can prove a lemma which is essential when we deduce certain regularity properties of a sequence $u_{\epsilon}$ in the statement of Theorems \ref{thm-m-limit} and \ref{thm-m-limit-sub}. See Section \ref{sec_asymp}.
\begin{lem}\label{lem-Arzela}
Suppose that the functions $\tilde{u}_{\epsilon}$ for $\epsilon >0$ defined in $\Omega$ are given by
\[\tilde{u}_{\epsilon}(x) = \int_{\Omega} G(x,y) \tilde{v}_{\epsilon}(y) dy,\]
where the set of functions $\{\tilde{v}_{\epsilon}: \epsilon>0\}$ satisfies $\sup_{\epsilon >0} \sup_{x \in \Omega} |\tilde{v}_{\epsilon}(x)| < \infty$.
Then $\{\tilde{u}_{\epsilon}: \epsilon >0\}$ are equicontinuous on any compact set.
\end{lem}
\begin{proof}
Suppose that $x_1$ and $x_2$ are contained in a compact set $\mathcal{K}$ of $\Omega$.
We have
\[\tilde{u}_{\epsilon}(x) = \int_{\Omega} G_{\mathcal{C}}(x,0,y) \tilde{v}_{\epsilon}(y) dy
=\int_{\Omega} G_{\mathbb{R}^{n+1}_{+}}(x,0,y) \tilde{v}_{\epsilon}(y) dy - \int_{\Omega} H_{\mathcal{C}}(x,0,y) \tilde{v}_{\epsilon}(y) dy\]
for any $x \in \Omega$. Take any number $\eta>0$.
It is well-known that the first term of the right-hand side is $C^{\alpha}$ for any $\alpha <2s$ if $s \in (0,1/2]$ and $C^{1,\alpha}$ for any $\alpha < 2s-1$ if $s \in (1/2,1)$.
Let us denote the last term by $R_{\epsilon}$. Then we have
\[|R_{\epsilon}(x_1) - R_{\epsilon}(x_2)| \le \int_{\Omega} \left|H_{\mathcal{C}}(x_1,0,y)- H_{\mathcal{C}}(x_2,0,y) \right| |\tilde{v}_{\epsilon}(y)| dy.\]
By Lemma \ref{lem-property-H} (2), we can find $\eta >0$ such that if  $|x_1-x_2| < \eta$ and $(x_1,x_2) \in \mathcal{K}\times \mathcal{K}$, then
\[\sup_{y \in \Omega} |H_{\mathcal{C}}(x_1,0,y) - H_{\mathcal{C}}(x_2,0,y)| \leq C\eta.\]
From this, we derive that
\[|R_{\epsilon}(x_1) -R_{\epsilon}(x_2)| \leq C\eta|\Omega|.\]
It proves that $\{\tilde{u}_{\epsilon}: \epsilon>0\}$ are equicontinuous on any compact set.
\end{proof}

\section{The asymptotic behavior}\label{sec_asymp}
Here we prove Theorem \ref{thm-m-limit} by studying the normalized functions $B_{\epsilon}$ of the $s$-harmonic extension $U_{\epsilon}$ of solutions $u_{\epsilon}$ for \eqref{a12uu}, given $\epsilon > 0$ sufficiently small.
We first find a pointwise convergence of the functions $B_{\epsilon}$.
Then we will prove that the functions $B_{\epsilon}$ are uniformly bounded by a certain function,
which is more difficult part to handle.
To obtain this result, we apply the Kelvin transform in the extended problem \eqref{u0inc}, and then attain $L^{\infty}$-estimates for its solution.
In addition we also need an argument to get a bound of the supremum $\|u_{\epsilon}\|_{L^{\infty}(\Omega)}$ in terms of $\epsilon>0$.
It involves a local version of the Pohozaev identity (see  Proposition \ref{prop-sub-estimate}).

\subsection{Pointwise convergence}
Set $U_{\epsilon}$ be the $s$-harmonic extension of $u_{\epsilon}$ to the half cylinder $\Omega \times [0,\infty)$, that is, $U_{\epsilon}$ satisfies $\textrm{tr}|_{\Omega \times \{0\}} U_{\epsilon} = u_{\epsilon}$ and it is a solution to the problem
\begin{equation}\label{dpuu}
\left\{ \begin{array}{ll}
\textrm{div}(t^{1-2s} \nabla U_{\epsilon}) = 0 & \quad \textrm{in}~ \mathcal{C} = \Omega \times (0,\infty),
\\
U_{\epsilon} >0 &\quad \textrm{in}~ \mathcal{C},
\\
U_{\epsilon}=0& \quad \textrm{on}~\partial_L \mathcal{C} = \partial \Omega \times [0,\infty),\\
\partial_{\nu}^{s}  U_{\epsilon} = U_{\epsilon}^{p}+{\epsilon}U_{\epsilon} &\quad \textrm{in} ~\Omega \times \{0\}.
\end{array}
\right.
\end{equation}
First we note the following identity
\begin{align*}
\int_{\mathcal{C}} t^{1-2s}|\nabla U_{\epsilon} (x,t)|^2 dx dt & =C_s \int_{\Omega \times \{0\}} \partial_{\nu}^{s} U_{\epsilon} (x,0) U_{\epsilon} (x,0) dx
\\
&=C_s \int_{\Omega \times \{0\}}\mathcal{A}_{s} u_{\epsilon} (x) u_{\epsilon} (x) dx
\\
&=C_s \int_{\Omega \times \{0\}} \left|\mathcal{A}_{s}^{1/2} u_{\epsilon} (x)\right|^2 dx.
\end{align*}
Using this with \eqref{2p1s0}, we have
\[\frac{(\int_{\Omega} |U_{\epsilon} (x,0)|^{p+1} dx)^{1/{(p+1)}} }{(\int_{\mathcal{C}} t^{1-2s}|\nabla U_{\epsilon}(x,t)|^2 dx dt)^{1/2}}= {\mathcal{S}_{n,s} \over \sqrt{C_s}}+ o (1)\quad \textrm{as} \quad \epsilon \rightarrow 0.\]
Also, by \eqref{a12uu}, it holds that
\begin{align*}
\int_{\mathcal{C}} t^{1-2s}|\nabla U_{\epsilon} (x,t)|^2 dxdt
&= C_s \int_{\Omega \times \{0\}} \left|\mathcal{A}_{s}^{1/2} u_{\epsilon} (x)\right|^2 dx = C_s \int_{\Omega \times \{0\}} \mathcal{A}_{s} u_{\epsilon}(x) u_{\epsilon}(x) dx\\
&= C_s \int_{\Omega \times\{0\}} u_{\epsilon}^{p+1}(x) dx + \epsilon C_s \int_{\Omega \times \{0\}} u_{\epsilon}^2 (x) dx.
\end{align*}
The two equalities above give
\[(\mathcal{S}_{n,s} + o (1) )^2\left(\| U_{\epsilon}( \cdot, 0)\|_{L^{p+1}(\Omega)}^{p+1} + \epsilon \| U_{\epsilon} (\cdot,0)\|_{L^2 (\Omega)}^2\right) =  \| U_{\epsilon} (\cdot,0)\|_{L^{p+1}(\Omega)}^2.\]
From $\|U_{\epsilon}(\cdot,0) \|_{L^2 (\Omega)} \leq C(\Omega) \|U_{\epsilon} (\cdot,0)\|_{L^{p+1}(\Omega)}$ we obtain
\[(\mathcal{S}_{n,s} + o (1) )^2\| U_{\epsilon}( \cdot, 0)\|_{L^{p+1}(\Omega)}^{p+1}  = \| U_{\epsilon} (\cdot,0)\|_{L^{p+1}(\Omega)}^2,\]
which turns to be
\begin{equation}\label{eq-up-limit}
\lim_{\epsilon \rightarrow 0} \int_{\Omega} U_{\epsilon}(x,0)^{p+1} dx = \mathcal{S}_{n,s}^{-\frac{n}{s}}.
\end{equation}

We set
\begin{equation}\label{eq-I-omega}
\mathcal{I}(\Omega,{r}) = \{ x  \in \Omega : \textrm{dist}(x,\partial \Omega) \geq r \} \quad \textrm{for}~ r>0
\end{equation}
and
\begin{equation}\label{eq-O-omega}
\mathcal{O}(\Omega,r) = \{x \in \Omega : \textrm{dist}(x, \partial \Omega ) < r\} \quad \textrm{for}~r>0.
\end{equation}
The following lemma presents a uniform bound of the solutions near the boundary.
\begin{lem}\label{lem-boundary-sup}
Let $u$ be a bounded solution of \eqref{a12uu} with $p>1$ and $0<\epsilon< \lambda_1^s$, where $\lambda_1$ is the first eigenvalue of $-\Delta$ with the zero Dirichlet condition. Then, for any $r>0$ there exists a number $C(r,\Omega)>0$ such that
\begin{equation}\label{upxc}
\int_{\mathcal{I}(\Omega,{r})} u~ dx \leq C(r,\Omega).
\end{equation}
Moreover, there is a constant $C> 0$ such that
\begin{equation}\label{xsux}
\sup_{x \in \mathcal{O}(\Omega,r)} u(x) \leq C.
\end{equation}
\end{lem}
\begin{proof}
Let $\phi_1$ be a first eigenfunction of the Dirichlet Laplacian $-\Delta$ in $\Omega$ such that $\phi_1 > 0$ in $\Omega$.
We have
\[{\lambda_1^s} \int_{\Omega} \phi_1 u dx  = \int_{\Omega} (\mathcal{A}_{s} \phi_1) u dx = \int_{\Omega} \phi_1 (\mathcal{A}_{s} u) dx = \int_{\Omega} \phi_1 u^p dx + \epsilon \int_{\Omega}\phi_1 u dx.\]
Using the Jensen inequality we get the estimate
\[C \left(\int_{\Omega} \phi_1 u dx \right)^p \le \int_{\Omega} \phi_1 u^p dx = (\lambda_1^s - \epsilon) \int_{\Omega} \phi_1 u dx,\]
and hence
\[\int_{\Omega} \phi_1 u ~dx \leq \left({{\lambda_1^s - \epsilon} \over C}\right)^{\frac{1}{p-1}}.\]
Because $\phi_1 \geq C$ on $ \mathcal{I}(\Omega,{r})$, we have
\begin{equation}\label{phi1u}
C \int_{ \mathcal{I}(\Omega,r)} u ~dx \leq \left({{\lambda_1^s - \epsilon} \over C}\right)^{\frac{1}{p-1}}.
\end{equation}
This completes the derivation of the estimate \eqref{upxc}.

If $\Omega$ is strictly convex, the moving plane argument,
which is given in the proof of \cite[Theorem 7.1]{CT} for $s=1/2$ and can be extended to any $s \in (0,1)$ with \cite[Lemma 3.6]{T2} and \cite[Corollary 4.12]{CS2},
yields the fact that the solution $u$ increases along an arbitrary straight line toward inside of $\Omega$ emanating from a point on $\partial \Omega$.
Then, by borrowing an averaging argument from \cite[Lemma 13.2]{QS} or \cite{H}, which heavily depends on this fact,
we can bound $\sup_{x \in \mathcal{O}(\Omega,r)}u(x)$ by a constant multiple of $\int_{\mathcal{I}(\Omega,r)} u(x) dx$.
In short, estimate \eqref{phi1u} gives the uniform bound \eqref{xsux} near the boundary.
The general cases can be proved using the Kelvin transformation in the extended domain (see \cite{C}).
\end{proof}

\begin{lem}\label{lem_mu_e}
Let
\begin{equation}\label{mue}
\mu_{\epsilon} = \mathfrak{c}_{n,s}^{-1} \sup_{x \in \Omega} u_{\epsilon}(x)
\end{equation}
where the definition of $\mathfrak{c}_{n,s}$ is provided in \eqref{cns}.
(Its finiteness comes from \cite[Proposition 5.2]{BCPS2}.)
If a point $x_{\epsilon} \in \Omega$ satisfies $\mu_{\epsilon} = \mathfrak{c}_{n,s}^{-1}u_{\epsilon}(x_{\epsilon})$, then we have
\[\lim_{\epsilon \rightarrow 0} \mu_{\epsilon} = \infty,\]
and $x_{\epsilon}$ converges to an interior point $x_0$ of $\Omega$ along a subsequence.
\end{lem}
\begin{proof}
Suppose that $u_{\epsilon}$ has a bounded subsequence.
As before, we let $U_{\epsilon}$ be the extension of $u_{\epsilon}$ (see \eqref{dpuu}).
By Lemma \ref{lem-Arzela}, $u_{\epsilon}$ are equicontinuous, and thus the Arzela-Ascoli theorem implies that $u_{\epsilon}$ converges to a function $v$ uniformly on any compact set.
We denote by $V$ the extension of $v$.
Then we see that $\lim_{\epsilon \rightarrow 0} \nabla U_{\epsilon} (x,t) = \nabla V (x,t)$ for any $(x,t) \in \mathcal{C}$ from the Green's function representation.
Thus we have
\begin{align*}
\int_{\mathcal{C}} t^{1-2s} |\nabla V|^2 dx dt = \int_{\mathcal{C}} t^{1-2s} \liminf_{\epsilon \rightarrow 0} |\nabla U_{\epsilon}|^2 dx dt &\leq \liminf_{\epsilon \rightarrow 0} \int_{\mathcal{C}} t^{1-2s} |\nabla U_{\epsilon}|^2 dx dt
\\
&= \liminf_{\epsilon \rightarrow 0} C_s \int_{\Omega} (u_{\epsilon}^{p+1} + \epsilon u_{\epsilon}^2) dx
\\
&= C_s \int_{\Omega} v^{p+1} dx.
\end{align*}
Meanwhile, using \eqref{eq-up-limit}, we obtain
\[\left( \int_{\mathcal{C}} t^{1-2s}|\nabla V|^2 dx dt\right)^{1 \over 2} \leq \frac{C_s^{1/2}}{\mathcal{S}_{n,s}} \left( \int_{\Omega} V^{p+1}(x,0) dx\right)^{1 \over p+1}.\]
Hence the function $V$ attains the equality in the sharp Sobolev trace inequality \eqref{eq-sharp-trace},
so we can deduce that $V = cW_{\lambda,\xi}$ for some $c, \lambda > 0$ and $\xi \in \mathbb{R}^n$ (see Subsection \ref{subsec_sobolev_trace}).
However, the support of $V$ is $\mathcal{C}$ by its own definition.
Consequently, a contradiction arises and the supremum $\mu_{\epsilon} = \mathfrak{c}_{n,s}^{-1}u_{\epsilon} (x_\epsilon)$ diverges.
Since Lemma \ref{lem-boundary-sup} implies $u_{\epsilon}$ is uniformly bounded near the boundary for all small $\epsilon >0$, the point $x_{\epsilon}$ converges to an interior point passing to a subsequence.
\end{proof}

Now, we normalize the solutions $u_{\epsilon}$ and their extensions $U_{\epsilon}$, that is, we set
\begin{equation}\label{aexme}
b_{\epsilon} (x):= \mu_{\epsilon}^{-1} u_{\epsilon} \big(\mu_{\epsilon}^{-\frac{2}{n-2s}} x + x_{\epsilon}\big),\quad x \in \Omega_{\epsilon}:=   \mu_{\epsilon}^{\frac{2}{n-2s}}(\Omega-x_{\epsilon}),
\end{equation}
and
\begin{equation}\label{aezme}
B_{\epsilon} (z) := \mu_{\epsilon}^{-1} U_{\epsilon} \big(\mu_{\epsilon}^{-\frac{2}{n-2s}} z + x_{\epsilon}\big),\quad z \in \mathcal{C}_{\epsilon} :=   \mu_{\epsilon}^{\frac{2}{n-2s}}(\mathcal{C}-(x_{\epsilon},0))
\end{equation}
with the value $\mu_{\epsilon}$ defined in \eqref{mue}.
It satisfies $ b_{\epsilon}(0) =\mathfrak{c}_{n,s}$ and  $0 \leq b_{\epsilon} \leq \mathfrak{c}_{n,s}$, and the domain $\Omega_{\epsilon}$ converges to $\mathbb{R}^{n}$ as $\epsilon$ goes to zero. The function $B_{\epsilon}$ satisfies
\[\left\{ \begin{array}{ll}
\textrm{div}(t^{1-2s} \nabla B_{\epsilon}) =0&\quad \textrm{in}~ \mathcal{C}_{\epsilon},
\\
B_{\epsilon}> 0&\quad \textrm{in}~\mathcal{C}_{\epsilon},
\\
B_{\epsilon}=0&\quad \textrm{on}~ \partial_L \mathcal{C}_{\epsilon},
\\
\partial_{\nu}^{s} B_{\epsilon}=B_{\epsilon}^{p} + \epsilon \mu_{\epsilon}^{-p+1} B_{\epsilon}&\quad \textrm{in} ~\Omega_{\epsilon} \times \{0\}.
\end{array}
\right.\]
We have
\begin{lem}\label{lem-convergence}
The function $b_{\epsilon}$ converges to the function $w_1$ uniformly on any compact set in a subsequence.
\end{lem}
\begin{proof}
Let $B$ be the weak limit of $B_{\epsilon}$ in $H_{0,L}^{s}(\mathcal{C})$ and $b = \text{tr}|_{\Omega \times \{0\}} B$.
Then it satisfies $b(0) = \max_{x \in \mathbb{R}^n} b(x) = \mathfrak{c}_{n,s}$ and
\[\left\{ \begin{array}{ll}
\textrm{div}(t^{1-2s} \nabla B) =0&\quad \textrm{in}~\mathbb{R}^{n+1}_{+},
\\
B> 0&\quad \textrm{in}~\mathbb{R}^{n+1}_+,
\\
\partial_{\nu}^{s} B=B^{p} &\quad \textrm{in} ~\mathbb{R}^n \times \{0\},
\end{array}
\right.\]
as well as $B$ is an extremal function of the Sobolev trace inequality \eqref{eq-sharp-trace} (see Subsection \ref{subsec_sobolev_trace}). Therefore $B (x,t) = W_{1}(x,t)$.
By Lemma \ref{lem-Arzela}, the family of functions $\{b_{\epsilon}(x): \epsilon>0\}$ are equicontinuous on any compact set in $\mathbb{R}^n$,
so by the Arzela-Ascoli theorem $b_{\epsilon}$ converges to a function $v$ on any compact set.
The function $v$ should be equal to the weak limit function $w_1$. It proves the lemma.
\end{proof}

\subsection{Uniform boundedness}
The previous lemma tells that the dilated solution $b_{\epsilon}$ converges to the function $w_1$ uniformly on each compact set of $\Omega_{\epsilon}$.
However it is insufficient for proving our main theorems and in fact we need a refined uniform boundedness result.
\begin{prop}\label{prop-uniform-bound}
There exists a constant $C>0$ independent of $\epsilon >0$ such that
\begin{equation}\label{aexcu}
b_{\epsilon} (x) \leq C w_1 (x).
\end{equation}
By rescaling, it can be shown that it is equivalent to
\begin{equation}\label{uexcu}
u_{\epsilon}(x) \leq C w_{\mu_{\epsilon}^{-{2 \over n-2s}},x_{\epsilon}} (x).
\end{equation}
\end{prop}
\noindent The proof of this result follows as a combination of the Kelvin transformation, a priori $L^{\infty}$-estimates, and an inequality which comes from a local Pohozaev identity for the solutions of \eqref{u0inc}.

\medskip
We set the Kelvin transformation
\begin{equation}\label{eq-kelvin-b}
d_{\epsilon}(x) = |x|^{-(n-2s)} b_{\epsilon} \left( \kappa(x) \right)\quad \text{for } x \in \Omega_{\epsilon},
\end{equation}
and
\begin{equation}\label{eq-kelvin-B}
D_{\epsilon} (z) = |z|^{-(n-2s)} B_{\epsilon} \left( \kappa(z) \right)\quad \text{for } z \in \mathcal{C}_{\epsilon},
\end{equation}
where $\kappa(x) = \frac{x}{|x|^2}$ is the inversion map.
Then, inequality \eqref{aexcu} is equivalent to that $d_{\epsilon}(x)\leq C$ for all $x \in \kappa(\Omega_{\epsilon})$.
Because $0 < b_{\epsilon}(x) \le \mathfrak{c}_{n,s}$ for $x \in \Omega_{\epsilon}$, it is enough to find a constant $C>0$ and a radius $r>0$ such that
\begin{equation}\label{eq-kelvin-bound}d_{\epsilon}(x) \leq C \quad \textrm{for}\quad x \in B_{n}(0,r) \cap \kappa (\Omega_{\epsilon}) \quad \text{for all } \epsilon>0.
\end{equation}
After making elementary but tedious computations, we find that the function $D_{\epsilon}$ satisfies
\[\textrm{div}(t^{1-2s} \nabla D_{\epsilon}) =0 \quad \textrm{in} ~ \kappa(\mathcal{C}_{\epsilon}).\]
Also we have
\begin{align*}
\partial_{\nu}^{s} D_{\epsilon}(x,0) &= \lim_{t \rightarrow 0} t^{1-2s} \frac{\partial}{\partial \nu} \left[ |z|^{-(n-2s)} B_{\epsilon}\left(\frac{z}{|z|^2 }\right) \right]
\\
&= \lim_{t \rightarrow 0} t^{1-2s} |z|^{-(n-2s+2)} \frac{\partial}{\partial \nu} B_{\epsilon} \left(\frac{z}{|z|^2}\right)
\\
&= \lim_{t \rightarrow 0} |z|^{-n-2s} \lim_{t \rightarrow 0} \left[ \left(\frac{t}{|z|^2}\right)^{1-2s} \frac{\partial}{\partial \nu} B_{\epsilon} \left(\frac{z}{|z|^2}\right)\right]
\\
&= |x|^{-n-2s} B_{\epsilon}^{p} \left(\frac{x}{|x|^2}\right) + \epsilon  \mu_{\epsilon}^{-p+1} |x|^{-n-2s} B_{\epsilon}^{p} \left(\frac{x}{|x|^2}\right)
\\
&= D_{\epsilon}^{p}(x,0) + \epsilon \mu_{\epsilon}^{-p+1} |x|^{-4s} D_{\epsilon} (x,0) \qquad \text{for } x \in \kappa(\Omega_{\epsilon}).
\end{align*}
Hence the function $D_{\epsilon}$ satisfies
\begin{equation}\label{eq-B-problem}
\left\{ \begin{array}{ll}
\textrm{div}(t^{1-2s} \nabla  D_{\epsilon}) (z) = 0 &\quad \textrm{in} ~\kappa( \mathcal{C}_{\epsilon}),
\\
D_{\epsilon} > 0 &\quad \textrm{in}~ \kappa(\mathcal{C}_{\epsilon}),
\\
D_{\epsilon} =0 &\quad \textrm{on}~\kappa( \partial_L \mathcal{C}_{\epsilon}),
\\
\partial_{\nu}^{s} D_{\epsilon} =  D_{\epsilon}^{p} + \epsilon \mu_{\epsilon}^{-p+1} |x|^{-4s} D_{\epsilon} &\quad \textrm{on}~ \kappa(\Omega_{\epsilon} \times \{0\}).
\end{array}\right.
\end{equation}
Here we record that
\begin{equation}\label{eq-e-bound-1}
\| \mu_{\epsilon}^{-p+1} |x|^{-4s}\|_{L^{\frac{n}{2s}} (B_{n}(0,1) \cap \kappa (\Omega_{\epsilon}))}
\leq \left( \mu_{\epsilon}^{-\frac{2n}{n-2s}} \int_{\big\{|x| \geq \mu_{\epsilon}^{-\frac{p-1}{2s}}\big\}} |x|^{-2n} dx\right)^{2s \over n} = C.
\end{equation}
In order to show \eqref{eq-kelvin-bound}, we shall prove two regularity results for the problem \eqref{eq-B-problem} in Lemma \ref{lem-harnack-1} and Lemma~\ref{lem-harnack-2} below.

In fact, to make \eqref{eq-B-problem} satisfy the conditions that Lemma \ref{lem-harnack-2} can be applicable,
we need a higher order integrability of the term $\epsilon \mu_{\epsilon}^{-p+1} |x|^{-4s}$ than that in \eqref{eq-e-bound-1}.
Note that  for $\delta>0$ we have
\begin{equation}\label{cepmu}
c \epsilon \mu_{\epsilon}^{\frac{8s^2 \delta}{n + 2s \delta}} \leq \| \epsilon \mu_{\epsilon}^{-p+1} |x|^{-4s} \|_{L^{\frac{n}{2s}+\delta}(B_{n}(0,1)\cap \kappa (\Omega_{\epsilon}))} \leq C \epsilon \mu_{\epsilon}^{\frac{8s^2 \delta}{n + 2s \delta}},
\end{equation}
for some constants $C>0$ and $c>0$.
Thus it is natural to find a bound of $\mu_{\epsilon}$ in terms of a certain positive power of $\epsilon^{-1}$.
It will be achieved later by using Lemma \ref{lem-harnack-3} and an inequality derived from a local version of the Pohozaev identity (see Lemma \ref{eq-mu-epsilon}).

\medskip
In what follows, whenever we consider a family of functions whose domains of definition are a set $D \subset \mathbb{R}^k$, we will denote $\int_{B_k (0,r)} f = \int_{B_k (0,r) \cap D} f$ for any ball $B_k (0,r)\subset \mathbb{R}^k $ for each $r > 0$ and $k \in \mathbb{N}$.

\begin{lem}\label{lem-harnack-1}
Let $V$ be a bounded solution of the equations:
\[\left\{ \begin{array}{ll} \textnormal{div}(t^{1-2s} \nabla V )(z) = 0 &\quad \textnormal{in}~\kappa(\mathcal{C}_{\epsilon}),
\\
V>0 &\quad \textnormal{in}~\kappa(\mathcal{C}_{\epsilon}),
\\
V=0 &\quad \textnormal{on}~\kappa(\partial_L \mathcal{C}_{\epsilon}),
\\
\partial_{\nu}^{s} V (x,0) = g(x)V(x,0) &\quad \textnormal{on}~\kappa(\Omega_{\epsilon} \times \{0\}).
\end{array}
\right.\]
Fix $\beta \in (1, \infty)$. Suppose that there is a constant $r>0$ such that
\begin{equation}\label{cond-w}
\left\| g\right\|_{L^{\frac{n}{2s}} (\kappa(\Omega_{\epsilon} \times \{0\}) \cap B_{n}(0,2r))}
\leq \frac{\beta}{2\mathcal{S}_{n,s}^2(\beta+1)^2},
\end{equation}
and
\[\int_{B_{n+1}(0,2r)} t^{1-2s} V(x,t)^{\beta+1} dx dt \leq Q.\]
Then, there exists a constant $C=C(\beta, r,Q)>0$ such that
\[\int_{B_{n}(0,r)} V(x,0)^{ \frac{(\beta+1)(p+1)}{2}} dx \leq C. \]
\end{lem}

\begin{rem}
Here we imposed the condition that $V$ is bounded for the simplicity of the proof.
This is a suitable assumption for our case, because we will apply it to the function $D_{\epsilon}$ which is already known to be bounded for each $\epsilon >0$.
However, this lemma holds without the assumption on the boundedness.
To prove this, one may use a truncated function $V_L := V \cdot 1_{\{|v|\leq L\}}$ with for large $L>0$ where the function $1_D$ for any set $D$ denotes the characteristic function on $D$. See the proof of Lemma \ref{lem_linear_nondeg}.
\end{rem}
\begin{proof}
Choose a smooth function $\eta \in C_c^{\infty}(\mathbb{R}^{n+1}, [0,1])$ supported on $B_{n+1}(0,2r) \subset \mathbb{R}^{n+1}$ satisfying $\eta = 1$ on $B_{n+1}(0,r)$.
Multiplying the both sides of
\[\textrm{div}(t^{1-2s} \nabla V) = 0 \quad \text{in } \kappa (\mathcal{C}_{\epsilon})\]
by $\eta^2 V^{\beta}$ and using that $V =0$ on $\kappa (\partial_L \mathcal{C}_{\epsilon})$, we discover that
\begin{equation}\label{2vbdz}
C_s \int_{\kappa(\Omega_{\epsilon} \times \{0\})} g(x) V^{\beta+1}(x,0) \eta^2 (x,0) dx = \int_{\kappa(\mathcal{C}_{\epsilon})}t^{1-2s} (\nabla V) \cdot \nabla (\eta^2 V^{\beta}) dz.
\end{equation}
Also, we can employ Young's inequality to get
\begin{equation}\label{2e2dz}
\begin{aligned}
\int_{\kappa(\mathcal{C}_{\epsilon})} t^{1-2s}(\nabla V) \cdot \nabla (\eta^2 V^\beta) dz &= \int_{\kappa(\mathcal{C}_{\epsilon})} \beta t^{1-2s} \eta^2 V^{\beta-1} |\nabla V|^2 +  2t^{1-2s} V^{\beta} \eta (\nabla V) \cdot (\nabla \eta)dz
\\
&= \int_{\kappa(\mathcal{C}_{\epsilon})} t^{1-2s} \beta |V^{\frac{\beta-1}{2}} \eta (\nabla V)|^2 dz + 2 \int_{\kappa(\mathcal{C}_{\epsilon})} t^{1-2s} V^{\beta} \eta (\nabla V) \cdot (\nabla \eta) dz
\\
&\geq \frac{\beta}{2} \int_{\kappa(\mathcal{C}_{\epsilon})} t^{1-2s}|V^{\frac{\beta-1}{2}} \eta (\nabla V)|^2 dz - \frac{2}{\beta} \int_{{\kappa(\mathcal{C}_{\epsilon})}} t^{1-2s}| V^{\frac{\beta+1}{2}} (\nabla \eta)|^2 dz.
\end{aligned}
\end{equation}
On the other hand, applying the identity
\[\nabla ( V^{\frac{\beta+1}{2}} \eta) = \frac{\beta+1}{2} V^{\frac{\beta-1}{2}}\eta (\nabla V) + V^{\frac{\beta+1}{2}} (\nabla \eta),\]
we obtain
\[2 \left(\frac{\beta+1}{2}\right)^2 |V^{\frac{\beta-1}{2}} \eta (\nabla V)|^2 + 2 |V^{\frac{\beta+1}{2}} (\nabla \eta)|^2 \geq |\nabla (V^{\frac{\beta+1}{2}} \eta)|^2.\]
This gives
\[|V^{\frac{\beta-1}{2}} \eta (\nabla V)|^2 \geq \frac{2}{(\beta+1)^2} \left\{ |\nabla ( V^{\frac{\beta+1}{2}} \eta)|^2 - 2 |V^{\frac{\beta+1}{2}} (\nabla \eta)|^2\right\}.\]
Combining this with \eqref{2vbdz} and \eqref{2e2dz}, and using the Sobolev trace inequality, we deduce that
\begin{equation}\label{koeoh}
\begin{aligned}
&C_s \int_{\kappa(\Omega_{\epsilon}\times \{0\})} g(x) V^{\beta+1}(x,0) \eta^2 (x,0) dx
\\
&\geq \frac{\beta}{2} \frac{2}{(\beta+1)^2} \int_{\kappa(\mathcal{C}_{\epsilon})} t^{1-2s}|\nabla (V^{\frac{\beta+1}{2}} \eta)|^2 dz - \left( \frac{2}{\beta} + \frac{2\beta}{(\beta+1)^2}\right) \int_{\kappa(\mathcal{C}_{\epsilon})}t^{1-2s} | V^{\frac{\beta+1}{2}} (\nabla \eta)|^2 dz
\\
&\geq \frac{C_s \beta }{\mathcal{S}_{n,s}^2 (\beta+1)^2} \left( \int_{\kappa(\Omega_{\epsilon} \times \{0\})} \left( V^{\frac{\beta+1}{2}} \eta\right)^{p+1} dx \right)^{\frac{2}{p+1}}\\
&\hspace{170pt} - \left( \frac{2}{\beta} + \frac{2 \beta}{(\beta+1)^2}\right) \int_{\kappa(\mathcal{C}_{\epsilon})} t^{1-2s} | V^{\frac{\beta+1}{2}} (\nabla \eta)|^2 dz.
\end{aligned}
\end{equation}
Moreover, we use the assumption \eqref{cond-w} to get
\begin{multline}\label{eq-w-holder}
\int_{\kappa(\Omega_{\epsilon}\times \{0\})} g(x) V^{\beta+1}(x,0) \eta^2 (x,0) dx \leq \left( \int_{\kappa (\Omega_{\epsilon} \times \{0\})} (\eta V^{\frac{\beta+1}{2}})^{p+1} dx \right)^{\frac{2}{p+1}} \left\| g \right\|_{L^{\frac{p+1}{p-1}}(\kappa(\Omega_{\epsilon}\times \{0\}) \cap B_n(0,2r))}\\
\leq \frac{\beta}{2 \mathcal{S}_{n,s}^2 (\beta+1)^2} \left( \int_{\kappa (\Omega_{\epsilon} \times \{0\})} (\eta V^{\frac{\beta+1}{2}})^{p+1} dx \right)^{\frac{2}{p+1}}.
\end{multline}
Using this estimate, we can derive from \eqref{koeoh} that
\begin{align*}
&\frac{C_s \beta}{2 \mathcal{S}_{n,s}^2 (\beta+1)^2} \left( \int_{\kappa(\Omega_{\epsilon} \times\{0\})} (V^{\beta+1}\eta^2)^{\frac{p+1}{2}} dx \right)^{\frac{2}{p+1}}
\\
&\geq ~\frac{C_s \beta}{\mathcal{S}_{n,s}^2 (\beta+1)^2} \left( \int_{\kappa(\Omega_{\epsilon} \times \{0\})} (V^{\frac{\beta+1}{2} }\eta)^{p+1} dx\right)^{\frac{2}{p+1}}- \left( \frac{2}{\beta}
+ \frac{2\beta}{(\beta+1)^2} \right) \int_{\kappa(\mathcal{C}_{\epsilon})}t^{1-2s} |V^{\frac{\beta+1}{2}} (\nabla \eta) |^2 dz.
\end{align*}
We now have
\begin{align*}
\int_{\kappa(\Omega_{\epsilon}) \cap B_{n}(0,r)} (V^{\frac{\beta+1}{2}} )^{p+1} dx
&\leq C \left(\int_{\kappa(\mathcal{C}_{\epsilon})} t^{1-2s} |V^{\frac{\beta+1}{2}}\nabla \eta|^2 dz\right)^{p+1 \over 2}\\
&\leq C \left(\int_{\kappa(\mathcal{C}_{\epsilon}) \cap B_{n+1}(0,2r)} t^{1-2s}|V|^{\beta+1} dz\right)^{p+1 \over 2} \leq C.
\end{align*}
This completes the proof.
\end{proof}
Next, we prove the $L^{\infty}$-estimate by applying the Moser iteration technique.
For the proof of Lemma \ref{lem-harnack-2}, we utilize the Sobolev inequality on weighted spaces which appeared in Theorem 1.3 of \cite{FKS} as well as the Sobolev trace inequality \eqref{eq-sharp-trace}.
Such an approach already appeared in the proof of Theorem 3.4 in \cite{GQ}.
\begin{prop}\cite[Theorem 1.3]{FKS}
Let $\Omega$ be an open bounded set in $\mathbb{R}^{n+1}$.
Then there exists a constant $C = C(n, s, \Omega) > 0$ such that
\begin{equation}\label{eq-sobolev-weight}
\left( \int_{\Omega} |t|^{1-2s} |U(x,t)|^{\frac{2(n+1)}{n}} dx dt\right)^{\frac{n}{2(n+1)}}
\leq C \left( \int_{\Omega} |t|^{1-2s} |\nabla U(x,t)|^2 dx dt \right)^{1 \over 2}
\end{equation}
holds for any function $U$ whose support is contained in $\Omega$ whenever the right-hand side is well-defined.
\end{prop}

\begin{lem}\label{lem-harnack-2}
Let $V$ be a bounded solution of the equations
\[\left\{\begin{array}{ll}
\textnormal{div}(t^{1-2s} \nabla V)  = 0 &\quad \textnormal{in} ~\kappa (\mathcal{C}_{\epsilon}),
\\
V>0&\quad \textnormal{in}~\kappa(\mathcal{C}_{\epsilon}),
\\
V=0  &\quad \textnormal{on}~ \kappa (\partial_L \mathcal{C}_{\epsilon}),
\\\partial_{\nu}^{s} V (x,0)= g(x) V(x,0)
&\quad  \textnormal{on} ~\kappa (\Omega_{\epsilon} \times \{0\}).
\end{array}
\right.\]
Fix $\beta_0 \in (1, \infty)$. Suppose that
\[\int_{B_{n+1}(0,r)} t^{1-2s} V(x,t)^{\beta_0+1} dx dt +  \int_{B_{n}(0,r)} V(x,0)^{\beta_0+1} dx \leq Q_1\]
and
\[\int_{\kappa(\Omega_{\epsilon}\times\{0\}) \cap B_{n}(0,r)} |g(x)|^{q} dx \leq Q_2\]
for some $r > 0$  and $q> \frac{n}{2s}$.
Then there exists a constant $C= C(\beta_0, r, Q_1, Q_2)>0$ such that
\[\| V(\cdot,0)\|_{L^{\infty} (B_{n}(0,r/2))} \leq C.\]
\end{lem}
\begin{proof}
Let $\eta \in C_c^{\infty}(\mathbb{R}^{n+1})$. Then the same argument as \eqref{2vbdz}-\eqref{koeoh} in the proof of the previous lemma gives
\begin{equation}\label{gxe2v}
\begin{aligned}
&C_s \int_{\kappa(\Omega_{\epsilon} \times \{0\})} g \eta^2 V^{\beta+1} dx
\\
&\qquad \geq \frac{\beta}{2} \frac{2}{(\beta+1)^2} \int_{\kappa(\mathcal{C}_{\epsilon})} t^{1-2s}|\nabla (V^{\frac{\beta+1}{2}} \eta)|^2 dz - \left( \frac{2}{\beta} + \frac{2\beta}{(\beta+1)^2}\right) \int_{\kappa(\mathcal{C}_{\epsilon})} t^{1-2s}| V^{\frac{\beta+1}{2}} \nabla \eta|^2 dz.
\end{aligned}
\end{equation}
First, we use H\"older's inequality to estimate the left-hand side by
\begin{align*}
\int_{\kappa(\Omega_{\epsilon} \times \{0\})} g \eta^2 V^{\beta+1} dx
&\leq \left( \int_{\kappa(\Omega_{\epsilon} \times \{0\})} (V^{\beta+1} \eta^2 )^{q'} dx \right)^{1 \over q'} \left( \int_{\kappa(\Omega_{\epsilon} \times\{0\})} |g|^q dx \right)^{1 \over q}\\
&\leq C \left( \int_{\kappa(\Omega_{\epsilon} \times \{0\})} (V^{\beta+1} \eta^2)^{q'} dx \right)^{1 \over q'}
\end{align*}
where $q'$ denotes the H\"older conjugate of $q$, i.e., $q' = {q \over q-1}$.
Since $q > \frac{p+1}{p-1}$, we have $q'< \frac{p+1}{2}$ and so the following interpolation inequality holds.
\begin{align*}
&\left( \int_{\kappa(\Omega_{\epsilon} \times \{0\})} (V^{\beta+1} \eta^2)^{q'} dx \right)^{1 \over q'}
\\
&\quad\quad \leq \left( \int_{\kappa(\Omega_{\epsilon} \times\{0\})} (V^{\beta+1} \eta^2)^{\frac{p+1}{2}} dx\right)^{2\theta \over p+1}
\left( \int_{\kappa(\Omega_{\epsilon} \times\{0\})} (V^{\beta+1}\eta^2 ) dx \right)^{1-\theta}
\\
&\quad\quad \leq \delta^{1 \over \theta} \theta \left( \int_{\kappa(\Omega_{\epsilon} \times \{0\})} (V^{\beta+1} \eta^2)^{\frac{p+1}{2}} dx \right)^{2 \over p+1} + \delta^{-{1 \over 1-\theta}} (1-\theta) \int_{\kappa(\Omega_{\epsilon} \times\{0\})} (V^{\beta+1} \eta^2 ) dx,
\end{align*}
where $\theta \in (0,1)$ and $\delta > 0$ satisfy respectively
\[{\frac{2\theta}{p+1}+ (1-\theta) = \frac{1}{q'}} \quad \text{and} \quad \delta = \left({1 \over \theta C} \cdot {\beta \over 2(\beta+1)^2}\right)^{\theta}\]
for an appropriate number $C > 0$. Then \eqref{gxe2v} gives
\begin{equation}\label{t1t2b}
\begin{aligned}
&\frac{\beta}{2(\beta+1)^2} \int_{\kappa(\mathcal{C}_{\epsilon})} t^{1-2s} | \nabla (V^{\frac{\beta+1}{2}} \eta)|^2 dz
\\
&\leq C \beta^{\theta \over 1-\theta} \int_{\kappa(\Omega_{\epsilon} \times\{0\})} (V^{\beta+1} \eta^2) dx
+ \left( \frac{2}{\beta}+ \frac{2\beta}{(\beta+1)^2}\right) \int_{\kappa(\mathcal{C}_{\epsilon})} t^{1-2s} |V^{\frac{\beta+1}{2}} \nabla \eta|^2 dz.
\end{aligned}
\end{equation}
Consequently the weighted Sobolev inequality \eqref{eq-sobolev-weight}, the trace inequality \eqref{eq-sharp-trace} and \eqref{t1t2b} yield that
\begin{equation}\label{vb12e}
\begin{aligned}
&\ \left( \int_{\kappa(\Omega_{\epsilon}\times\{0\})} |V^{\frac{\beta+1}{2}} \eta|^{p+1} dx \right)^{2 \over p+1} + \left( \int_{\kappa(\mathcal{C}_{\epsilon})} t^{1-2s}|V^{\frac{\beta+1}{2}} \eta|^{2(n+1) \over n} dxdt\right)^{n \over n+1}
\\
& \leq C \int_{\kappa(\mathcal{C}_{\epsilon})} t^{1-2s} \left(|\nabla (V^{\frac{\beta+1}{2}} \eta)|^2 + |V^{\beta+1} \eta^2|\right) dxdt
\\
& \leq C\beta^{1 \over 1-\theta}\left[ \int_{\kappa(\mathcal{C}_{\epsilon})} t^{1-2s} V^{\beta+1} \left(|\nabla \eta|^2 + \eta^2\right) dxdt  +
\int_{\kappa(\Omega_{\epsilon} \times \{0\})} |V^{\beta+1} \eta^2 | dx\right].
\end{aligned}
\end{equation}
Now, for each $0 < r_1 < r_2$, we take a function $\eta \in C_{c}^{\infty}(\mathbb{R}^{n+1}, [0,1])$ supported on $B_{n+1}(0,r_2)$ such that $\eta =1$ on $B_{n+1}(0,r_1)$.
Then the above estimate \eqref{vb12e} implies
\begin{equation}
\begin{aligned}\label{eq-moser}
&\left( \int_{B_{n}(0,r_1)} V^{(\beta+1)\frac{p+1}{2}} dx \right)^{2 \over p+1} + \left( \int_{B_{n+1}(0,r_1)} t^{1-2s} V^{{(\beta+1)}\frac{n+1}{n}}  dz \right)^{n \over n+1}
\\
&\leq \frac{C \beta^{1 \over 1-\theta}}{(r_2 - r_1)^2} \left[ \left( \int_{B_{n}(0,r_2)} V^{\beta+1} dx\right) + \left( \int_{B_{n+1}(0,r_2)} t^{1-2s} V^{\beta+1} dz\right)  \right].
\end{aligned}
\end{equation}

We will use this inequality iteratively.
We denote $\theta_0 = \min\{{p+1 \over 2}, {n+1 \over n}\} >1$ and set $\beta_k + 1= (\beta_0+1) \theta_0^k$ and $R_k = r/2 + r/2^k$ for $k \in \mathbb{N} \cup \{0\}$.
By applying the inequality $a^{\gamma} + b^{\gamma} \ge (a+b)^{\gamma}$ for any $a, b > 0$ and $\gamma \in (0,1]$ with H\"older's inequality, and then taking $\beta = \beta_k$ in \eqref{eq-moser}, we obtain
\begin{align*}
&\left( \int_{B_{n}(0,R_{k+1})} V^{\beta_{k+1}+1} dx + \int_{B_{n+1}(0,R_{k+1})} t^{1-2s} V^{\beta_{k+1}+1} dz \right)^{1 \over \beta_{k+1}+1}
\\
&\qquad \leq C^{1 \over (\beta_0+1)\theta_0^k} \left[ \theta_0^{\frac{k}{1-\theta}}2^{2k}\right]^{1 \over (\beta_0+1)\theta_0^k} \left(\int_{B_{n}(0,R_k)} V^{\beta_k + 1} dx + \int_{B_{n+1}(0,R_k)} t^{1-2s} V^{\beta_k + 1} dz \right)^{1 \over \beta_k + 1}.
\end{align*}
Set
\[A_k (V) = \left( \int_{B_{n}(0,R_{k})} V^{\beta_k + 1} dx + \int_{B_{n+1}(0,R_k)} t^{1-2s} V^{\beta_k + 1} dz \right)^{1 \over \beta_k + 1}.\]
Then, for $D := (4 \theta_0^{\frac{1}{1-\theta}})^{1 \over \beta_0+1}$, we have
\[A_{k+1} \leq C^{1 \over \theta_0^k} D^{k \over \theta_0^k} A_k.\]
Using this we get
\[A_k \leq C^{\sum_{j=1}^{\infty} {1 \over \theta_0^j}} D^{\sum_{j=1}^{\infty} \frac{j}{\theta_0^{j}}} A_0 \leq C A_0,\]
from which we deduce that
\[\sup_{x \in B_{n}(0,r/2)} V(x,0) = \lim_{k \rightarrow \infty} \left( \int_{B_{n+1}(0,r/2)} V^{\beta_k+1} (x,0)dx \right)^{1 \over \beta_k + 1} \leq \sup_{k\in \mathbb{N}} A_k \leq C.\]
This concludes the proof.
\end{proof}

As we mentioned before, we cannot use the above result to the function {$D_{\epsilon}$} directly because the estimate \eqref{eq-e-bound-1} is not enough to employ this result.
To overcome this difficulty, we will seek a refined estimation of the term $\epsilon\mu_{\epsilon}^{-p+1}|x|^{-4s}$ than \eqref{eq-e-bound-1},
and in particular we will try to bound $\mu_{\epsilon}$ by a constant multiple of $\epsilon^{-\alpha}$ having \eqref{cepmu} in mind where $\alpha > 0$ is a sufficiently small number.
We deduce the next result, which is a local invariant of the previous lemma, as the first step for this objective.
\begin{lem}\label{lem-harnack-3}
Let $V$ be a bounded solution of the equations
\[\left\{\begin{array}{ll}
\textnormal{div}(t^{1-2s} \nabla V)  = 0 &\quad \textnormal{in} ~\kappa (\mathcal{C}_{\epsilon}),
\\
V>0&\quad \textnormal{in}~\kappa(\mathcal{C}_{\epsilon}),
\\
V=0 &\quad \textnormal{on}~ \kappa (\partial_L \mathcal{C}_{\epsilon}),
\\
\partial_{\nu}^{s} V (x,0)= g(x) V(x,0) + \epsilon \varphi(x) V(x,0)
&\quad \textnormal{on} ~\kappa (\Omega_{\epsilon} \times \{0\}).
\end{array}
\right.\]
Fix $\beta \in (1,\infty)$. Suppose that $\varphi$ satisfies $\|\varphi\|_{L^{\frac{n}{2s}}(\mathbb{R}^{n})} \leq Q_1$,
\[\int_{B_{n+1}(0,r)}t^{1-2s} V(x,t)^{\beta+1} dx dt +  \int_{B_{n}(0,r)} V(x,0)^{\beta+1} dx \leq Q_2\]
and
\[\int_{B_{n}(0,r)} |g(x)|^{q} dx \leq Q_3,\]
for some $r>0$ and $q> \frac{n}{2s}$.
Then, for any $J>1$, there exist constants $\epsilon_0 = \epsilon_0 (Q_1, J)>0$ and  $C= C(r, Q_1, Q_2, Q_3, J)>0$ depending on $r$, $Q_1$, $Q_2$, $Q_3$ and $J$ such that, if $0 < \epsilon < \epsilon_0$, then we have
\[\| V(\cdot,0)\|_{L^{J} (B_{n}(0,r/2))} \leq C.\]
\end{lem}
\begin{proof}
Let $\eta \in C_c^{\infty}(\mathbb{R}^{n+1})$. Then the same argument for \eqref{koeoh} gives
\begin{multline}\label{gxe2v2}
C_s \int_{\kappa(\Omega_{\epsilon} \times \{0\})} g(x) \eta^2 V^{\beta+1} (x,0)dx + \epsilon C_s \int_{\kappa (\Omega_{\epsilon} \times \{0\})}  \varphi(x)\eta^2 V^{\beta+1}(x,0) dx
\\
\qquad \geq \frac{\beta}{2} \frac{2}{(\beta+1)^2} \int_{\kappa(\mathcal{C}_{\epsilon})} t^{1-2s}|\nabla (V^{\frac{\beta+1}{2}} \eta)|^2 dz - \left( \frac{2}{\beta} + \frac{2\beta}{(\beta+1)^2}\right) \int_{\kappa(\mathcal{C}_{\epsilon})} t^{1-2s}| V^{\frac{\beta+1}{2}} \nabla \eta|^2 dz.
\end{multline}
Using H\"older's inequality we get
\[\epsilon \int_{\kappa (\Omega_{\epsilon} \times \{0\})} \varphi(x)\eta^2 V^{\beta+1}(x,0) dx \leq \epsilon \|\varphi\|_{L^{\frac{p+1}{p-1}}(\mathbb{R}^{n})} \| \eta^2 V^{\beta+1}(\cdot,0)\|_{L^{\frac{p+1}{2}}(\mathbb{R}^n)}.\]
If $\epsilon < \frac{\beta}{4(\beta+1)^2 \mathcal{S}_{n,s}^2Q_1}$, from the trace inequality, we obtain
\[\epsilon \|\varphi\|_{L^{\frac{p+1}{p-1}}(\mathbb{R}^n)} \| \eta^2 V^{\beta+1}(\cdot,0)\|_{L^{\frac{p+1}{2}}(\mathbb{R}^n)} \leq \frac{\beta}{4(\beta+1)^2} \int_{\kappa (\mathcal{C}_{\epsilon})} t^{1-2s} |\nabla (V^{\frac{\beta+1}{2}} \eta)|^2 dz.\]
Now we can follow the steps \eqref{t1t2b}-\eqref{eq-moser} of the previous lemma.
Moreover, we can iterate it with respect to $\beta$ as long as $\epsilon < \frac{\beta}{4(\beta+1)^2 \mathcal{S}_{n,s}^2Q_1}$ holds.
Thus, for $\epsilon < \frac{J}{4(J+1)^2 \mathcal{S}_{n,s}^2Q_1}$, we can find a constant $C= C(r, C_1, C_2, C_3, J)$ such that
\[\| V (\cdot,0) \|_{L^J (B_{n}(0,r/2))} \leq C.\]
It proves the lemma.
\end{proof}

To apply the previous lemma to get a bound of $\mu_{\epsilon}$ in terms of $\epsilon$, we also need to make the use of the Pohozaev identity of $U_{\epsilon}$:
\[\frac{1}{2C_s } \int_{\partial_L \mathcal{C}} t^{1-2s} |\nabla U_{\epsilon} (z)|^2 \langle z, \nu \rangle dS = \epsilon s \int_{\Omega \times \{0\}} U_{\epsilon}(x,0)^2 dx.\]
As a matter of fact, we will not use this identity directly, but instead we will utilize its local version to prove the following result.
\begin{prop}\label{prop-sub-estimate}
Suppose that $U \in H^s_{0,L}(\mathcal{C})$ is a solution of problem \eqref{u0inc} with $f$ such that $f$ has the critical growth and $f = F'$ for some function $F \in C^1 (\mathbb{R})$. Then, for each $ \delta >0$ and $q >\frac{n}{s}$ there is a constant $C= C(\delta, q) >0$ such that
\begin{equation}\label{poho-prop-estimate}
\begin{aligned}
&\min_{r \in [\delta, 2\delta]}\left|n \int_{\mathcal{I}(\Omega,{r/2})\times\{0\}} F(U) dx - \left(\frac{n-2s}{2}\right) \int_{ \mathcal{I}(\Omega,{r/2}) \times \{0\}} U f(U) dx\right|
\\
&\leq C \left[ \left(\int_{\mathcal{O}(\Omega,{2\delta})\times \{0\}} |f(U)|^{q} dx\right)^{2 \over q} + \int_{\mathcal{O}(\Omega,2\delta)\times \{0\}} |F(U)| dx + \left( \int_{\mathcal{I}(\Omega, \delta/2) \times \{0\}} |f(U)| dx \right)^2 \right]
\end{aligned}
\end{equation}
where $\mathcal{I}$ and $\mathcal{O}$ is defined in  \eqref{eq-I-omega} and \eqref{eq-O-omega}.
\end{prop}
\noindent We defer the proof of the proposition to Appendix \ref{sec_appen_a}. We remark that this kind of estimate was used in \cite{C} for~$s=1/2$.

\medskip
Now we can prove the following result.
\begin{lem}\label{eq-mu-epsilon}
There exist a constant $C >0$ and $\alpha >0$ such that
\[\mu_{\epsilon} \leq C \epsilon^{-\alpha}\quad \textrm{for all } ~\epsilon >0.\]
\end{lem}
\begin{proof}
We denote
\begin{equation}\label{fF}
f(u)= u^p + \epsilon u \quad \text{and } F(u) = \frac{1}{p+1} u^{p+1} + \frac{1}{2}\epsilon u^2 \quad \text{for } u > 0
\end{equation}
and fix a small number $\delta>0$ so that $\mathcal{I}(\Omega,{\delta})$ has the same topology as that of $\Omega$. For $r \in [\delta,2\delta]$ we see that
\begin{equation}\label{eq-u-l2}
\begin{aligned}
\epsilon \int_{\mathcal{I}(\Omega,{r})} u_{\epsilon}(x)^2 dx & = \epsilon \int_{\mathcal{I}(\Omega,r)} \mu_{\epsilon}^2 b_{\epsilon}\left(\mu_{\epsilon}^{\frac{p-1}{2s}} (x-x_{\epsilon})\right)^2 dx
=\epsilon \mu_{\epsilon}^2 \mu_{\epsilon}^{-\frac{p-1}{2s}n} \int_{\mu_{\epsilon}^{\frac{p-1}{2s}}(\mathcal{I}(\Omega,{r})-x_{\epsilon})} b_{\epsilon}(x)^2 dx
\\
&\geq \epsilon \mu_{\epsilon}^{-\frac{4s}{n-2s}} \int_{B_{n}(0,1)}  b_{\epsilon}^2 (x) dx \geq C\epsilon \mu_{\epsilon}^{-\frac{4s}{n-2s}},
\end{aligned}
\end{equation}
where we used the fact that $b_{\epsilon}$ converges to $w_1$ uniformly on any compact set (see Lemma \ref{lem-convergence}).
Since $U_{\epsilon}$ is a solution of \eqref{u0inc} with $f$ given in \eqref{fF}, we have
\begin{multline*}
\min_{r \in [\delta,2\delta]} \left|n \int_{\mathcal{I}(\Omega,r)\times\{0\}} F(U_{\epsilon}) dx - \left(\frac{n-2s}{2}\right) \int_{ \mathcal{I}(\Omega,{r}) \times \{0\}} U_{\epsilon} f(U_{\epsilon}) dx\right|\\
= \min_{r \in [\delta,2\delta]} \left|\epsilon s \int_{\mathcal{I}(\Omega,{r})} U_{\epsilon}(x,0)^2 dx \right| \geq C\epsilon \mu_{\epsilon}^{-\frac{4s}{n-2s}}.
\end{multline*}
This gives a lower bound of the left-hand side of \eqref{poho-prop-estimate}.

Now we shall find an upper bound of the right-hand side of \eqref{poho-prop-estimate}.
By Lemma \ref{lem-harnack-3}, for any $q<\infty$, we get $\| d_{\epsilon}\|_{L^q (B_{n}(0,1))} \leq C$ with a constant $C=C(q)>0$.
Using this we have
\begin{equation}\label{eq-b-a}
\begin{aligned}
C \ge \int_{\{|x|\leq 1\}} d^q_{\epsilon}(x)dx & = \int_{\{|x|\leq 1\}} |x|^{-(n-2s)q } b_{\epsilon}^q \left(\frac{x}{|x|^2}\right) dx
\\
&=\int_{\{|x| \geq 1\}} |x|^{(n-2s)q} b_{\epsilon}^{q} (x) |x|^{-2n} dx
\\
& = \int_{\{|x| \geq 1\}} |x|^{(n-2s)q -2n} \mu_{\epsilon}^{-q} u_{\epsilon}^q \left(\mu_{\epsilon}^{-\frac{p-1}{2s}} x + x_{\epsilon}\right) dx
\\
& = \int_{\big\{|x -x_{\epsilon}| \geq \mu_{\epsilon}^{-\frac{p-1}{2s}}\big\}} \mu_{\epsilon}^{\frac{p-1}{2s}[(n-2s)q -2n]} \mu_{\epsilon}^{-q} \mu_{\epsilon}^{\frac{p-1}{2s}n} |x-x_{\epsilon}|^{(n-2s) q-2n} u_{\epsilon}^q (x) dx
\\
&= \int_{\big\{|x-x_{\epsilon}| \geq \mu_{\epsilon}^{-\frac{p-1}{2s}}\big\}} \mu_{\epsilon}^{q-\frac{2n}{n-2s}} |x-x_{\epsilon}|^{(n-2s)q -2n} u_{\epsilon}^q (x) dx.
\end{aligned}
\end{equation}
First of all, we find a bound of $\int_{\Omega} u_{\epsilon}^{p} (x) dx.$ Using \eqref{eq-b-a} and H\"older's inequality we deduce that
\begin{align*}
&\int_{\big\{|x-x_{\epsilon}| \geq \mu_{\epsilon}^{-\frac{p-1}{2s}}\big\}} u_{\epsilon}^p (x) dx
\\
&\leq \left(\int_{\big\{|x-x_{\epsilon}| \geq \mu_{\epsilon}^{-\frac{p-1}{2s}}\big\}} u_{\epsilon}^{q} (x) |x-x_{\epsilon}|^{(n-2s)q -2n} dx \right)^{p \over q} \\
&\hspace{150pt} \times \left( \int_{\big\{|x-x_{\epsilon}| \geq \mu_{\epsilon}^{-\frac{p-1}{2s}}\big\}} |x-x_{\epsilon}|^{-[(n-2s) q-2n] \frac{p}{q-p} } dx \right)^{\frac{q-p}{q}}
\\
& \leq \mu_{\epsilon}^{-(q-\frac{n}{n-2s}) \frac{p}{q}} \mu_{\epsilon}^{\frac{p-1}{2s} [ ((n-2s)q -2n) \frac{p}{q-p} -n] \frac{q-p}{q}}.
\end{align*}
Note that if $q=\infty$, then the last term is equal to $\mu_{\epsilon}^{-p} \mu_{\epsilon}^{\frac{p-1}{2s}[(n+2s)-n]}= \mu_{\epsilon}^{-1}.$
Thus, for any  $\kappa >0$, we can find  $q = q(\kappa)$ sufficiently large so that the last term of the above estimate is bounded by $\mu_{\epsilon}^{-1+\kappa}$.
Then it follows that
\begin{equation}\label{eq-r-1}
\left(\int_{\big\{|x-x_{\epsilon}| \geq \mu_{\epsilon}^{-\frac{p-1}{2s}}\big\}} u_{\epsilon}^{p}(x) dx \right)^2 \leq \mu_{\epsilon}^{-2 + 2 \kappa}.
\end{equation}
On the other hand, because $u_{\epsilon} (x) \leq C\mu_{\epsilon}$, we have
\begin{equation}\label{eq-r-2}
\left(\int_{\big\{|x-x_{\epsilon}| \leq \mu_{\epsilon}^{-\frac{p-1}{2s}}\big\}} u_{\epsilon}^p (x) dx \right)^2 \leq C \mu_{\epsilon}^{2p} \mu_{\epsilon}^{-\frac{p-1}{2s} \cdot 2n} = C \mu_{\epsilon}^{\frac{4s -2n}{n-2s}}= C\mu_{\epsilon}^{-2}.
\end{equation}
These two estimates give us the bound of $\int_{\Omega} u_{\epsilon}^p (x) dx.$

Now we turn to bound $\|f(U_{\epsilon})(\cdot,0)\|_{L^q (\mathcal{O}(\Omega, 2\delta))}$. For this we again use inequality \eqref{eq-b-a} to have
\[\int_{\{|x-x_{\epsilon}| \geq \textrm{dist}(x_0, \partial \Omega)/2\}} u_{\epsilon}^{pq} (x) dx \leq C\mu_{\epsilon}^{-(pq- \frac{2n}{n-2s})} ~\textrm{for any}~ q >1.\]
Using this inequality for a sufficiently large $q$ and H\"older's inequality we can deduce that
\begin{equation}\label{eq-r-3}
\left( \int_{\{|x-x_{\epsilon}| \geq \textrm{dist}(x_0, \partial \Omega)/2\}} u_{\epsilon}^{pq}(x) dx \right)^{\frac{2}{q}} \leq C \left( \mu_{\epsilon}^{-(pq-\frac{2n}{n-2s})}\right)^{\frac{2}{q}} \leq C \mu_{\epsilon}^{-2p + \kappa}.
\end{equation}
Similarly we have
\begin{equation*}
\int_{\mathcal{O}(\Omega,2\delta)} |F(u_{\epsilon}(x))|dx \leq C \mu_{\epsilon}^{-(p+1)+\kappa}.
\end{equation*}
Combining this estimate with \eqref{eq-r-1}, \eqref{eq-r-2} and \eqref{eq-r-3} gives the bound
\[\left( \int_{\mathcal{O}(\Omega,2\delta)}  |f(U_{\epsilon})(x,0)|^{q} dx\right)^{2 \over q} + \int_{\mathcal{O}(\Omega,2\delta)} |F(U_{\epsilon}(x,0)|dx +\left( \int_{\Omega} f(U_{\epsilon} )(x,0) dx \right)^2 \leq C \mu_{\epsilon}^{-2 + 2\kappa}.\]
We put this bound and \eqref{eq-u-l2} into \eqref{poho-prop-estimate} in the statement of Proposition \ref{prop-sub-estimate}. Then we finally get
\begin{equation}\label{eq-m-bound-e}\epsilon \mu_{\epsilon}^{-\frac{4s}{n-2s}} \leq  C \mu_{\epsilon}^{-2 + 2 \kappa}
\end{equation}
which is equivalent to
\[\mu_{\epsilon}^{\frac{2n-8s}{n-2s} - 2 \kappa} \leq \frac{C}{\epsilon}.\]
Choose $\kappa >0$ such that $\alpha:= \frac{2n-8s}{n-2s} -2 \kappa$ is positive. Then the estimate \eqref{eq-m-bound-e} turns out to be
\[\mu_{\epsilon} \leq C \epsilon^{-\alpha},\]
which is the desired inequality.
\end{proof}

\begin{proof}[Proof of Proposition \ref{prop-uniform-bound}]
We know that
\begin{equation*}
\lim_{\epsilon \rightarrow 0} \int_{\mathcal{C}_{\epsilon}} t^{1-2s} |\nabla ( B_{\epsilon} - W_1)|^2 dx dt = 0.
\end{equation*}
By employing the Sobolev trace embedding, we find that
\begin{equation}\label{eq-b-bound-stable}
\lim_{\epsilon \rightarrow 0} \left[\int_{\Omega_{\epsilon}} | b_{\epsilon} (x) - w_1 (x)|^{p+1} dx \right]= 0.
\end{equation}
Since $p=\frac{n+2s}{n-2s}$, we have the scaling invariance
\[\int_{\mathbb{R}^n} |a (x)|^{p+1} dx = \int_{\mathbb{R}^n} |x|^{-2n}|a(\kappa (x))|^{p+1} dx,\]
and
\[\int_{\mathbb{R}^{n+1}} t^{1-2s}|\nabla A(z)|^2 dz \ge C \int_{\mathbb{R}^{n+1}} t^{1-2s}|\nabla[|z|^{-(n-2s)}A(\kappa(z))]|^2 dz\]
for arbitrary functions $a : \mathbb{R}^{n} \rightarrow \mathbb{R}$ and $A: \mathbb{R}^{n+1} \rightarrow \mathbb{R}$ which decay sufficiently fast.
Using these identities, we deduce from \eqref{eq-b-bound-stable} that
\[\lim_{\epsilon \rightarrow 0}\left[\int_{\kappa(\Omega_{\epsilon})} | d_{\epsilon} (x) - w_1 (x)|^{p+1} dx
+ \int_{\kappa(\mathcal{C}_{\epsilon})} t^{1-2s} |\nabla (D_{\epsilon} - W_1) (x,t)|^2 dxdt \right]= 0.\]
Using the Sobolev embedding theorem and H\"older's inequality, for $\beta_0 = \min\{p, {n+2 \over n}\} > 1$, we get
\begin{equation}\label{eq-D-bound}
\lim_{\epsilon \rightarrow 0} \int_{B_{n+1}(0,1)} t^{1-2s} |D_{\epsilon} (x,t)- W_1 (x,t)|^{\beta_0+1} dxdt =0.
\end{equation}
Finally, estimates \eqref{eq-D-bound} and \eqref{eq-e-bound-1} enable us to apply Lemma \ref{lem-harnack-1} so that we can find $\delta>0$ satisfying
\begin{equation}\label{eq-d-bound-2}
\int_{\kappa(\Omega_{\epsilon} \times \{0\}) \cap B_{n}(0,\delta)} \left( d_{\epsilon}^{p-1} \right)^{{n \over 2s}{\beta_0+1 \over 2}} dx
\leq C \quad \text{for any } \epsilon>0.
\end{equation}

Next, from Lemma \ref{eq-mu-epsilon} we may find $\alpha >0$ such that $\mu_{\epsilon} \leq \epsilon^{-\alpha}$. Then, for $\zeta>0$ small enough, we have
\begin{align*}
\| \epsilon \mu_{\epsilon}^{- p +1} |x|^{-4s} \|_{L^{\frac{n}{2s} + \zeta}(\kappa(\Omega_{\epsilon}))} &\leq \epsilon \left[ \int_{\big\{|x| \geq \mu_{\epsilon}^{-\frac{p-1}{2s}}\big\}} \mu_{\epsilon}^{-(p-1) (\frac{n}{2s} + \zeta)} |x|^{-2n-4s \zeta} dz \right]^{1 \over { \frac{n}{2s} + \zeta}}
\\
& \leq \epsilon \left[ \mu_{\epsilon}^{-(p-1)(\frac{n}{2s} + \zeta)} \mu_{\epsilon}^{\frac{p-1}{2s}(n+ 4s \zeta)} \right]^{1 \over { \frac{n}{2s} + \zeta}}
\\
& = \epsilon \mu_{\epsilon}^{\zeta (p-1) \over \frac{n}{2s} + \zeta}
\leq \epsilon \cdot \epsilon^{-\frac{\alpha \zeta (p-1)}{{n \over 2s} + \zeta}} \leq 1.
\end{align*}
Given this estimate and \eqref{eq-d-bound-2}, we can apply Lemma \ref{lem-harnack-2} to get
\[\| d_{\epsilon}\|_{L^{\infty}(B_{n}(0,\delta/2))} \leq C.\]
The proof is concluded.
\end{proof}

\subsection{Proof of Theorem \ref{thm-m-limit}}
We are now ready to prove Theorem \ref{thm-m-limit}.
\begin{proof}[Proof of Theorem \ref{thm-m-limit}]
By the definition of $\mu_{\epsilon}$ in \eqref{mue}, we have
\[\mathcal{A}_s(\|u_{\epsilon}\|_{L^{\infty}(\Omega)} u_{\epsilon})(x) =\mathfrak{c}_{n,s} \left[\mu_{\epsilon} u_{\epsilon}^{p}(x) + \epsilon \mu_{\epsilon} u_{\epsilon}(x)\right],\quad x \in \Omega.\]
Note from $p=\frac{n+2s}{n-2s}$ that
\begin{align*}
\int_{\Omega} \left(\mu_{\epsilon} u_{\epsilon}^{p}(x) + \epsilon \mu_{\epsilon} u_{\epsilon} (x)\right) dx
&= \int_{\Omega} \mu_{\epsilon}^{p+1} b_{\epsilon}^{p} \left(\mu_{\epsilon}^{\frac{p-1}{2s}} (x-x_{\epsilon}) \right) dx + \epsilon \mu_{\epsilon}^2 \int_{\Omega} b_{\epsilon} \left(\mu_{\epsilon}^{\frac{p-1}{2s}} (x-x_{\epsilon}) \right) dx
\\
&= \int_{\Omega_{\epsilon}} b_{\epsilon}^{p}(x) dx + \epsilon \mu_{\epsilon}^2 \mu_{\epsilon}^{-\frac{p-1}{2s} n} \int_{\Omega_{\epsilon}} b_{\epsilon}(x)dx.
\end{align*}
Note also that
\begin{align*}
\mu_{\epsilon}^2 \mu_{\epsilon}^{-\frac{p-1}{2s}n} \int_{\Omega_{\epsilon} } b_{\epsilon}(x) dx &\leq \mu_{\epsilon}^2 \mu_{\epsilon}^{-\frac{p-1}{2s} n} \int_{\big\{|x|\leq \mu_{\epsilon}^{\frac{p-1}{2s}}\big\}} \frac{C}{(1 +|x|)^{n-2s}} dx
\\
&\leq C \mu_{\epsilon}^{2 -\frac{p-1}{2s} n + \frac{p-1}{2s} 2s} \le C.
\end{align*}
Given the uniform bound \eqref{aexcu}, we use the Lebesgue dominated convergence theorem to obtain
\[\lim_{\epsilon \rightarrow 0} \int_{\Omega} \mathfrak{c}_{n,s}\mu_{\epsilon} u_{\epsilon}^{p}(x) = \int_{\mathbb{R}^n} \mathfrak{c}_{n,s} w_1^{p}(x) dx = \mathfrak{b}_{n,s},\]
where
\begin{equation}\label{eq-beta}
\mathfrak{b}_{n,s} := \frac{|S^{n-1}|}{2} \frac{ \Gamma \left( s\right) \Gamma \left( \frac{n}{2}\right)}{\Gamma \left( \frac{n+2s}{2}\right)} \mathfrak{c}_{n,s}^{p+1}.
\end{equation}
For $ x \neq x_0$, we have $\lim_{\epsilon \rightarrow 0} \mu_{\epsilon} u_{\epsilon}^{p}(x) =0$ by \eqref{uexcu}.
Therefore we may conclude that
\[\lim_{\epsilon \rightarrow 0} \mathcal{A}_{s}( \|u_{\epsilon}\|_{L^{\infty}(\Omega)} u_{\epsilon} ) (x) = \mathfrak{b}_{n,s} \delta_{x_0} (x) \quad \textrm{in}\quad C(\Omega)'.\]
Set $v_{\epsilon}:= \mathcal{A}_s (\|u_{\epsilon}\|_{L^{\infty}(\Omega)} u_{\epsilon})$. Then $\lim_{\epsilon \rightarrow 0} \int_{\Omega} v_{\epsilon} dx=\mathfrak{b}_{n,s}$ and $\lim_{\epsilon \rightarrow 0}  v_{\epsilon}(x) =0$ uniformly on any compact set of $\Omega \setminus \{x_0\}$. We observe the formula
\begin{equation}\label{eq-H-expression}
\|u_{\epsilon}\|_{L^{\infty}(\Omega)} U_{\epsilon}(x,t) = \int_{\Omega} \left[\frac{\mathfrak{a}_{n,s}}{|(x-y,t)|^{n-2s}} - H_{\mathcal{C}}(x,t,y) \right] v_{\epsilon}(y) dy.
\end{equation}
On the other hand we have $H_{\mathcal{C}}(x,t,\cdot)$ is in $C^{\infty}_{\textrm{loc}}(\Omega)$ and  $\|H_{\mathcal{C}}(x,t,\cdot)\|_{L^{\frac{2n}{n-2s}}(\Omega)} \leq C$ which holds uniformly on any compact set of $\Omega \setminus \{x_0\}$.
From this we conclude that
\[\|u_{\epsilon}\|_{L^{\infty}(\Omega)} U_{\epsilon}(x,t) \rightarrow  \mathfrak{b}_{n,s} G_{\mathcal{C}}(x,t,x_0)\quad \textrm{in}~
C^{0}_{\textrm{loc}}(\overline{\mathcal{C}} \setminus \{(x_0,0)\}).\]
Also, pointwise convergence in $\mathcal{C}$ is valid for the derivatives of $\|u_{\epsilon}\|_{L^{\infty}(\Omega)} U_{\epsilon}$ by elliptic regularity.
Especially, for $t=0$, the regularity property of the function $x\in \Omega \rightarrow H_{\mathcal{C}}(x,0,y)$ given in Lemma \ref{lem-property-H} proves that
\[\|u_{\epsilon}\|_{L^{\infty}(\Omega)} u_{\epsilon}(x) \rightarrow \mathfrak{b}_{n,s} G(x,x_0) \text{ in } \left\{\begin{array}{lll} C^{\alpha}_{\textrm{loc}}(\Omega \setminus \{x_0\}) &\ \text{for all } \alpha \in (0,2s) & \textrm{if } s \in (0,1/2],
\\
C^{1,\alpha}_{\textrm{loc}}(\Omega \setminus \{x_0\})&\ \text{for all } \alpha
\in (0,2s-1) & \textrm{if } s \in (1/2,1).
\end{array}\right.\]
This completes the proof.
\end{proof}

\section{Location of the blowup point}\label{sec_blowup}
The objective of this section is to prove Theorem \ref{thm-m-location}.
For this goal, we will derive several identities related to Green's function.
Throughout this section, we keep using the notations:
$X_0 = (x_0, 0)$, $B_r = B_{n+1} (X_0, r) \cap \mathbb{R}^{n+1}_{+}$, $\partial B_r^+ = \partial B_r \cap \mathbb{R}^{n+1}_{+}$ and $\Gamma_r = B_n(x_0,r)$ for $r > 0$ small.
We also use $G(z)$ (or $H(z)$) to denote $G_{\mathcal{C}}(z,x_0)$ (or $H_{\mathcal{C}}(z,x_0)$) for brevity.

\medskip
The first half of this section is devoted to proving the second statement of Theorem \ref{thm-m-location}.
\begin{proof}[Proof of Theorem \ref{thm-m-location} (2)]
According to Appendix \ref{sec_appen_b}, it holds
\begin{equation}\label{eq-main-limit}
\begin{aligned}
&\lim_{\epsilon \to 0} \epsilon s C_s \mu_{\epsilon}^{2(n-4s) \over n-2s} \delta \int_{\mathbb{R}^n}w_1^2(x)dx
\\
&= \mathfrak{b}_{n,s}^2\int_{\delta}^{2\delta}\left[\int_{\partial B_r^+} t^{1-2s}\left<(z-X_0, \nabla G(z)) \nabla G(z) - (z-X_0) \frac{|\nabla G(z)|^2}{2}, \nu\right> dS\right.
\\
&\ + \left.\left(\frac{n-2s}{2}\right) \int_{\partial B_r^+} t^{1-2s} G(z,x_0) \frac{\partial G(z)}{\partial \nu} dS\right] dr
\end{aligned}
\end{equation}
for an each $\delta > 0$ small enough.
We will now take a limit $\delta \rightarrow 0$.
Putting
\[G(z)= \frac{\mathfrak{a}_{n,s}}{|z-X_0|^{n-2s}} - H(z)
\quad \text{and} \quad
\nabla G(z) = - \mathfrak{a}_{n,s} (n-2s) \frac{z-X_0}{|z-X_0|^{n+2-2s}} - \nabla H(z)\]
into the right-hand side of \eqref{eq-main-limit} and applying $\nu = \frac{z-X_0}{r}$ on $\partial B^+_r$, we can derive
\begin{align*}
&2\lim_{\epsilon \to 0} \epsilon s C_s \mu_{\epsilon}^{2(n-4s) \over n-2s} \int_{\mathbb{R}^n}w_1^2(x)dx\\
&= (n-2s)^2 \mathfrak{a}_{n,s} \mathfrak{b}_{n,s}^2 \lim_{r \rightarrow 0} \left( 2 \int_{\partial B^+_{2r}} \frac{t^{1-2s}}{(2r)^{n+1-2s}} H(z) dS - \int_{\partial B^+_r} \frac{t^{1-2s}}{r^{n+1-2s}} H(z) dS \right)\\
&\ + \lim_{\delta \rightarrow 0} {1 \over \delta} \int_{\delta}^{2\delta} \int_{\partial B^+_r} t^{1-2s} O \left( \langle \nu, \nabla H(z) \rangle \left({1 \over r^{n-2s}} + H(z)\right) + r |\nabla H(z)|^2 \right) dSdr.
\end{align*}
Since $\partial_i H_{\mathcal{C}} (\cdot,x_0)$ has a bounded H\"older norm over a small neighborhood of $x_0$ for each $i = 1, \cdots, n$ (refer to \cite[Lemma 2.9]{CDDS}),
the second term in the right-hand side tends to 0. As a result,
\[2\lim_{\epsilon \to 0} \epsilon s C_s \mu_{\epsilon}^{2(n-4s) \over n-2s} \int_{\mathbb{R}^n}w_1^2(x)dx \to (n-2s)^2 {\mathcal{D}_{n,s}} \mathfrak{a}_{n,s}\mathfrak{b}_{n,s}^2 \tau(x_0)\]
as $\delta \to 0$, where
\[\mathcal{D}_{n,s} := \lim_{r \rightarrow 0} \int_{\partial B_{n+1}(0,r) \cap \mathbb{R}^{n+1}_+} \frac{t^{1-2s}}{r^{n+1-2s}} dS = \int_{B_n(0,1)} \frac{1}{(1-|x|^2)^s} dx = \frac{|S^{n-1}|}{2} B \left(1-s, \frac{n}{2}\right),\]
$B$ denoting the beta function.
This proves Theorem \ref{thm-m-location} (2).
We also know that the constant $\mathfrak{d}_{n,s}$ in the statement of the theorem is given by
\begin{equation}\label{eq-J}
\mathfrak{d}_{n,s} = \frac{\Gamma (n-2s)}{\pi^{n/s} \Gamma (\frac{n}{2}-2s)} \frac{(n-2s)^2}{2sC_s}{\mathcal{D}_{n,s}} \mathfrak{a}_{n,s}\mathfrak{b}_{n,s}^2\mathfrak{c}_{n,s}^{-\frac{4s}{n-2s}}.
\end{equation}
\end{proof}

Next, we prove the first statement of Theorem \ref{thm-m-location}, that is, $\tau'(x_0) = 0$.
\begin{proof}[Proof of Theorem \ref{thm-m-location} (1)]
If $U$ is a solution to \eqref{u0inc}, for each $1 \le k \le n$, we have
\begin{align*}
\int_{\partial B_r^+}t^{1-2s} |\nabla U|^2 \nu_k dS
&= \int_{B_r} t^{1-2s} \partial_k |\nabla U|^2 dz = 2 \int_{B_r} t^{1-2s} \nabla U \cdot \nabla \partial_k U dz\\
&= 2 \int_{\partial B_r^+} t^{1-2s} \langle \nabla U, \nu \rangle \partial_k U dS + 2C_s \int_{\partial \Gamma_r} F(U) \nu_k dS_x
\end{align*}
where $F(t) := \int_0^t f(t)dt$, $\nu_k$ is the $k$-th component of $\nu$, $\partial_k$ is the partial derivative with respect to the $k$-th variable and $r > 0$ small.
For the last equality, we used $\int_{\Gamma_r} f(U)\partial_kU dx = \int_{\Gamma_r} (\partial_kF)(U) dx = \int_{\partial \Gamma_r} F(U) \nu_k dS_x
$.
Therefore putting $\|U_{\epsilon}(\cdot, 0)\|_{L^{\infty}(\Omega)} U_{\epsilon}$ (see \eqref{dpuu}) in the place of $U$ in the above identity, integrating the result from $\delta$ to $2\delta$ in $r$ and taking $\epsilon \to 0$, we obtain
\begin{equation}\label{pgxpn}
\int_{\delta}^{2\delta}\int_{\partial B_r^+} t^{1-2s} |\nabla G|^2 \nu_k dSdr = 2 \int_{\delta}^{2\delta}\int_{\partial B_r^+} t^{1-2s} \langle \nabla G, \nu \rangle \partial_k G dSdr
\end{equation}
(cf. Appendix \ref{sec_appen_b}).
On the other hand, a direct calculation shows that
\begin{equation}\label{poho-g-nabla-1}
\begin{aligned}
&\lim_{\delta \rightarrow 0} {1 \over \delta} \int_{\delta}^{2\delta} \int_{\partial B_r^+} t^{1-2s} \langle \nabla G(z), \nu(z) \rangle \partial_kG(z) dS dr
\\
& = \lim_{\delta \rightarrow 0} {1 \over \delta} \int_{\delta}^{2\delta} \int_{\partial B_r^+} t^{1-2s} \left[ \frac{\mathfrak{a}_{n,s} (n-2s)}{r^{n-2s+1}} + \left\langle \frac{z-X_0}{r}, \nabla H(z) \right\rangle \right]
\\
& \hspace{190pt} \times \left[(x_k-x_{0,k})\frac{\mathfrak{a}_{n,s}(n-2s)}{r^{n+2-2s}} + \partial_k H(z) \right]dSdr
\\
& = \lim_{r \rightarrow 0} \int_{\partial B_r^+} t^{1-2s} \frac{\mathfrak{a}_{n,s} (n-2s)}{r^{n-2s+1}} \partial_k H(z) dS
\\
&\hspace{80pt} + \lim_{\delta \rightarrow 0} {1 \over \delta} \int_{\delta}^{2\delta} \int_{\partial B_r^+} t^{1-2s} \frac{\mathfrak{a}_{n,s}(n-2s)}{r^{n-2s+3}}(x_k-x_{0,k}) \langle z-X_0, \nabla H(z) \rangle dS dr
\\
& = (n-2s+3)(n-2s) \mathfrak{a}_{n,s} \mathcal{E}_{n,s} \partial_k \tau(x_0),
\end{aligned}
\end{equation}
where $x_k$ and $x_{0,k}$ mean the $k$-th component of $x$ and $x_0$, respectively, and
\[\mathcal{E}_{n,s} := \lim_{r \rightarrow 0} \int_{\partial B_{n+1}(0,r) \cap \mathbb{R}^{n+1}_+} \frac{t^{1-2s}}{r^{n-2s+3}}  x_k^2 dS
= \frac{1}{n} \int_{B_n(0,1)} \frac{|x|^2}{(1-|x|^2)^s} dx = \frac{|S^{n-1}|}{2n} B\left(1-s, \frac{n+2}{2}\right).\]
In particular, $\mathcal{D}_{n,s} = (n-2s+2)\mathcal{E}_{n,s}$.
Moreover we observe
\begin{equation}\label{ngx2c}
\begin{aligned}
&\lim_{\delta \rightarrow 0} {1 \over \delta} \int_{\delta}^{2\delta} \int_{\partial B_r^+} t^{1-2s} |\nabla G(z)|^2 \nu_k(z) dS dr
\\
&= 2 \lim_{r \rightarrow 0} \int_{\partial B_r^+} t^{1-2s}
\frac{\mathfrak{a}_{n,s}(n-2s)}{r^{n-2s+3}}(x_k -x_{0,k})^2 \partial_k H(z) dS
+ \lim_{\delta \rightarrow 0} {1 \over \delta} \int_{\delta}^{2\delta} \int_{\partial B_r^+} t^{1-2s} |\nabla H(z)|^2 \nu_k(z) dSdr
\\
&= 2 (n-2s) \mathfrak{a}_{n,s} \mathcal{E}_{n,s} \partial_k\tau(x_0).
\end{aligned}
\end{equation}
Taking $\delta \to 0$ in \eqref{pgxpn} with \eqref{poho-g-nabla-1} and \eqref{ngx2c} in hand gives our desired result.
\end{proof}

\section{Construction of solutions for \eqref{a12uu} concentrating at multiple points}\label{sec_reduction}
In this section we prove Theorem \ref{thm-m-multipeak} by applying the Lyapunov-Schmidt reduction method to the extended problem
\begin{equation}\label{equation-construction}
\left\{\begin{array}{ll} \text{div}\left(t^{1-2s} \nabla U\right) = 0 &\quad \text{in}~ \mathcal{C} = \Omega \times (0,\infty),\\
U>0 & \quad \text{in}~\mathcal{C},\\
U=0 & \quad \text{on}~ \partial_L \mathcal{C} = \partial \Omega \times (0,\infty),\\
\partial_{\nu}^s U = U^p + \epsilon U &\quad \text{on}~ \Omega \times \{0\},
\end{array}\right.
\end{equation}
where $0 < s < 1$ and $p = {n+2s \over n-2s}$.
We remind that the functions $w_{\lambda, \xi}$ and $W_{\lambda, \xi}$ are defined in \eqref{bubble} and \eqref{wlyxt}.
By the result of D\'avila, del Pino and Sire \cite{DDS}, it is known that the space of the bounded solutions for the linearized equation of \eqref{entire_nonlocal} at $w_{\lambda,\xi}$, namely,
\begin{equation}\label{dsppo1}
\mathcal{A}_{s} \phi = pw_{\lambda,\xi}^{p-1} \phi \quad \text{in}~\mathbb{R}^{n}
\end{equation}
is spanned by
\begin{equation}\label{dsppo2}
{\partial w_{\lambda,\xi} \over \partial \xi_1},\ \cdots,\ {\partial w_{\lambda,\xi} \over \partial \xi_n} \quad \text{and} \quad {\partial w_{\lambda,\xi} \over \partial \lambda}
\end{equation}
where $\xi = (\xi_1, \cdots, \xi_n)$ represents the variable in $\mathbb{R}^n$.
From this, it also follows that the solutions of the extended problem of \eqref{dsppo1}
\begin{equation}\label{extended_entire}
\left\{ \begin{array}{ll} \text{div}(t^{1-2s} \nabla \Phi )= 0 &\quad \text{in}~ \mathbb{R}^{n+1}_{+} = \mathbb{R}^n \times (0,\infty),\\
\partial_{\nu}^{s} \Phi = p w_{\lambda,\xi}^{p-1} \Phi &\quad \text{on}~ \mathbb{R}^n \times \{0\},
\end{array} \right.
\end{equation}
which are bounded on $\Omega \times \{0\}$, consist of the linear combinations of
\[{\partial W_{\lambda,\xi} \over \partial \xi_1},\ \cdots,\ {\partial W_{\lambda,\xi} \over \partial \xi_n} \quad \text{and} \quad {\partial W_{\lambda,\xi} \over \partial \lambda}.\]

In the proof of Theorem \ref{thm-m-multipeak}, we will often consider the dilated equation
\begin{equation}\label{equation-dilated}
\left\{\begin{array}{ll}
\text{div}(t^{1-2s} \nabla U) = 0 &\quad \text{in}~ \mathcal{C}_{\epsilon} = \Omega_{\epsilon} \times (0,\infty),\\
U>0 &\quad \text{in}~ \mathcal{C}_{\epsilon},\\
U=0&\quad \text{on}~ \partial_L \mathcal{C}_{\epsilon} = \partial \Omega_{\epsilon} \times (0,\infty),\\
\partial_{\nu}^{s} U = U^p + \epsilon^{1+2s\alpha_0} U &\quad \text{on}~ \Omega_{\epsilon}\times \{0\},
\end{array}\right.
\end{equation}
where
\[\mathcal{C}_{\epsilon} = \frac{\mathcal{C}}{\epsilon^{\alpha_0}} = \left\{ \frac{(x,t)}{\epsilon^{\alpha_0}} : (x,t) \in \mathcal{C} \right\}\]
and
\[\Omega_{\epsilon} = \frac{\Omega}{\epsilon^{\alpha_0}} = \left\{ \frac{x}{\epsilon^{\alpha_0}}: x \in \Omega \right\}\]
for some $\alpha_0 > 0$ to be determined later. If $U$ is a solution of \eqref{equation-dilated}, then $U_{\epsilon}(z):= \epsilon^{-\frac{(n-2s)}{2}\alpha_0} U (\epsilon^{-\alpha_0} z)$ for $z \in \Omega$ becomes a solution of problem \eqref{equation-construction}.

Since we want solutions to be positive, we use a well-known trick that replaces the nonlinear term $U^p$ in \eqref{equation-construction} with its positive part $U^p_+$.
Namely, we consider the following modified equation of \eqref{equation-dilated}
\begin{equation}\label{equation-dilated-m}
\left\{\begin{array}{ll}
\text{div}(t^{1-2s} \nabla U) = 0 &\quad \text{in}~ \mathcal{C}_{\epsilon},\\
U=0&\quad \text{on}~ \partial_L \mathcal{C}_{\epsilon},\\
\partial_{\nu}^{s} U = f_{\epsilon}(U) := U_+^p + \epsilon^{1+2s\alpha_0} U &\quad \text{on}~ \Omega_{\epsilon}\times \{0\}.
\end{array}\right.
\end{equation}

\subsection{Finite dimensional reduction}
In order to construct a $k$-peak solution of \eqref{equation-construction} ($k \in \mathbb{N}$),
we define the admissible set
\begin{multline}\label{admissible}
O^{\delta_0} = \left\{(\bls) := ((\lambda_1, \cdots, \lambda_k),(\sigma_1, \cdots, \sigma_k)) \in (\mathbb{R}^{+})^k \times \Omega^k: \sigma_i = (\sigma_i^1, \cdots, \sigma_i^n), \right.\\
\left. \text{dist} (\sigma_i, \partial \Omega) > \delta_0,\ \delta_0 < \lambda_i < \frac{1}{\delta_0},\ |\sigma_i - \sigma_j| > \delta_0,\ i \neq j,\ i, j = 1, \cdots, k \right\}
\end{multline}
with some small $\delta_0 > 0$ fixed, which recodes the information of the concentration rate and the locations of points of concentration.

Let the map
\[i_{\epsilon}^{*} : L^{\frac{2n}{n+2s}} (\Omega_{\epsilon}) \rightarrow H_{0,L}^s (\mathcal{C}_{\epsilon})\]
be the adjoint operator of the Sobolev trace embedding
\[i_{\epsilon}:H_{0,L}^s(\mathcal{C}_{\epsilon}) \rightarrow L^{\frac{2n}{n-2s}} (\Omega_{\epsilon}) \quad \text{defined by} \quad i_{\epsilon}(U)
:= \text{tr}|_{\Omega_{\epsilon} \times \{0\}}(U) \quad \text{for } U \in H_{0,L}^s(\mathcal{C}_{\epsilon}),\]
which comes from the inequality \eqref{eq-sharp-trace} (for the definition of $H_{0,L}^s(\mathcal{C}_{\epsilon})$, see Subsection \ref{subsec_frac_Sob}).
From its definition, $i_{\epsilon}^{*}(u)=V$ for some $u \in L^{\frac{2n}{n+2s}} (\Omega_{\epsilon})$ and $V \in H_{0,L}^s(\mathcal{C}_{\epsilon})$ if and only if
\[\left\{\begin{array}{ll}
\text{div}(t^{1-2s} \nabla V) = 0 &\quad \text{in}~ \mathcal{C}_{\epsilon},\\
V=0&\quad \text{on}~ \partial_L \mathcal{C}_{\epsilon},\\
\partial_{\nu}^{s} V = C_s^{-1}u &\quad \text{on}~ \Omega_{\epsilon}\times \{0\}.
\end{array}\right.\]
where $C_s > 0$ is the constant defined in \eqref{cs}.
Therefore finding a solution $U \in H^s_{0,L}(\mathcal{C}_{\epsilon})$ of \eqref{equation-dilated} is equivalent to solving the relation
\begin{equation}\label{equation_compact}
i_{\epsilon}^*(f_{\epsilon}(i_{\epsilon}(U))) = C_s^{-1}U.
\end{equation}
It is valuable to note that from \eqref{eq_Sobo_trace} we have in fact  $i_{\epsilon}:H_{0,L}^s(\mathcal{C}_{\epsilon}) \rightarrow H_0^s(\Omega_{\epsilon}) \subset L^{\frac{2n}{n-2s}} (\Omega_{\epsilon})$ and so $\mathcal{A}_s(i_{\epsilon}(U))$ makes sense.
See also Sublemma \ref{sol_pos}.

We introduce the functions
\begin{equation}\label{psis}
\Psi^{0}_{\lambda,\xi} = \frac{\partial W_{\lambda,\xi}}{\partial \lambda}, \quad \Psi_{\lambda,\xi}^{j} = \frac{\partial W_{\lambda,\xi}}{\partial \xi^{j}}, \quad
\psi^{0}_{\lambda,\xi} = \frac{\partial w_{\lambda,\xi}}{\partial \lambda}, \quad \psi_{\lambda,\xi}^{j} = \frac{\partial w_{\lambda,\xi}}{\partial \xi^{j}}
\end{equation}
where $\xi = \left(\xi^1, \cdots, \xi^n\right) \in \mathbb{R}^n$ and $j = 1, \cdots, n$, and
\begin{equation}\label{def_P_epsilon}
P_{\epsilon}W_{\lambda,\xi} = i_{\epsilon}^{*}\left(w_{\lambda,\xi}^p\right),
\quad P_{\epsilon} \Psi^j_{\lambda,\xi} = i_{\epsilon}^{*} \left(p w_{\lambda, \xi}^{p-1} \psi_{\lambda, \xi}^j\right) \quad \text{for } j=0,1,\cdots, n.
\end{equation}
Furthermore, we let the functions $P_{\epsilon}w_{\lambda,\xi}$ and $P_{\epsilon}\psi_{\lambda,\xi}^j$ be
\begin{equation}\label{def_P_epsilon2}
P_{\epsilon}w_{\lambda,\xi} = i_{\epsilon}(P_{\epsilon}W_{\lambda,\xi}) \quad
\text{and} \quad P_{\epsilon}\psi^j_{\lambda,\xi} = i_{\epsilon}\left(P_{\epsilon}\Psi^j_{\lambda,\xi}\right) \quad \text{for } j = 0, \cdots, n
\end{equation}
which vanish on $\partial\Omega_{\epsilon}$ and solve the equations $\mathcal{A}_{s} u = w_{\lambda,\xi}^p$ and $\mathcal{A}_{s} u = p w_{\lambda, \xi}^{p-1} \psi_{\lambda,\xi}^j$ in $\Omega_{\epsilon}$, respectively.
Also, whenever $(\bls) \in O^{\delta_0}$ is chosen, we denote
\begin{equation}\label{W_i}
W_i = W_{\lambda_i, \sigma_i\epsilon^{-\alpha_0}},
\quad P_{\epsilon} W_i = P_{\epsilon} W_{\lambda_i, \sigma_i\epsilon^{-\alpha_0}}
\quad \text{and}\quad P_{\epsilon} \Psi_i^j = P_{\epsilon}\Psi^j_{\lambda_i, \sigma_i\epsilon^{-\alpha_0}}
\end{equation}
and similarly define $P_{\epsilon} w_i$ and $P_{\epsilon} \psi_i^j$ ($i = 1, \cdots, k$ and $j = 0, 1, \cdots, n$) for the sake of simplicity. Set also
\begin{equation}\label{ortho}
K_{\bls}^{\epsilon} = \left\{ u \in H_{0,L}^{1} (\mathcal{C}_{\epsilon}) : \left(u, P_{\epsilon} \Psi_i^{j}\right)_{\mathcal{C}_{\epsilon}} = 0, ~ i=1,2,\cdots,k, ~j=0,1,\cdots, n\right\}
\end{equation}
for $\epsilon >0$ and $(\bls) \in O^{\delta_0}$
and define the orthogonal projection operator $\Pi_{\bls}^{\epsilon}: H_{0,L}^s(\mathcal{C}_{\epsilon}) \rightarrow K_{\bls}^{\epsilon}$.

Now, if we set
$L_{\bls}^{\epsilon}: K_{\bls}^{\epsilon} \rightarrow K_{\bls}^{\epsilon}$ by
\begin{equation}\label{linear_operator}
L_{\bls}^{\epsilon}(\Phi) = C_s^{-1}\Phi - \Pi_{\bls}^{\epsilon} i_{\epsilon}^{*} \left[ f_{\epsilon}' \left(\sum_{i=1}^{k} P_{\epsilon} w_i\right) \cdot i_{\epsilon}(\Phi) \right],
\end{equation}
then we can obtain the following lemma from the nondegeneracy result of \cite{DDS}.

\begin{lem}\label{lem_linear_nondeg}
Suppose that $(\bls)$ is contained in $O^{\delta_0}$.
Then there exists a positive constant $C=C(n, \delta_0)$ such that
\[\| L_{\bls}^{\epsilon} (\Phi)\|_{\mathcal{C}_{\epsilon}} \geq C \|\Phi\|_{\mathcal{C}_{\epsilon}} \quad \text{for all } \Phi \in K_{\bls}^{\epsilon} \text{ and sufficiently small } \epsilon > 0.\]
\end{lem}
\begin{proof}
Assume the contrary.
Then there exist sequences $\epsilon_l >0$, $\Phi_l \in K_{\bl_l, \bs_l}^{\epsilon_l}$, $H_l = L_{\bl_l, \bs_l}^{\epsilon_l} (\Phi_l)$ and $(\bl_l, \bs_l) = ((\lambda_{1l}, \cdots, \lambda_{kl}), (\sigma_{1l}, \cdots, \sigma_{kl})) \in O^{\delta_0}$ ($l \in \mathbb{N}$) satisfying
\begin{equation}\label{en0ln}
\lim_{l \to \infty} \epsilon_l = 0, \quad \|\Phi_l\|_{\mathcal{C}_{\epsilon_l}} = 1,
\quad \lim_{l \to \infty} \|H_l\|_{\mathcal{C}_{\epsilon_l}} = 0, \quad \lim_{l \to \infty} (\bl_l, \bs_l) = (\bl_{\infty}, \bs_{\infty}) \in O^{\delta_0}.
\end{equation}
Set $\mathcal{C}_l = \mathcal{C}_{\epsilon_l}$, $\Omega_l = \Omega_{\epsilon_l}$, $P_l w_{il} = P_{{\epsilon_l}}w_{\lambda_{il}, \sigma_{il}}$ and $P_l \Psi_{il}^{j} = P_{{\epsilon_l}} \Psi^{j}_{\lambda_{il}, \sigma_{il}}$ for $i = 1, \cdots, k$ and $j = 1, \cdots, n$.
If we further denote $\phi_l = i_{\epsilon_l}(\Phi_l)$, then we have
\begin{equation}\label{pnief}
C_s^{-1} \Phi_l - i_{\epsilon_l}^{*} \left[f'_{\epsilon_l}\left(\sum_{i=1}^{k} P_l w_{il}\right) \phi_l \right] = H_l + Q_l \quad \text{in } H_{0,L}^{1} (\mathcal{C}_l)
\end{equation}
where $Q_l := \sum_{i=1}^k \sum_{j=1}^n c_{ij}^{l} P_l \Psi^{j}_{il}$ for some constants $c_{ij}^{l} \in \mathbb{R}$.
By our assumptions above and the relation
\begin{equation}\label{Psi_ortho}
\lim_{l \to \infty}\left(P_l\Psi_{i_1l}^{j_1}, P_l\Psi_{i_2l}^{j_2}\right)_{\mathcal{C}_l}
= p \cdot C_s \lim_{l \to \infty} \int_{\Omega_l} U_{i_1}^{p-1}\psi_{i_1l}^{j_1}P_l\psi_{i_2l}^{j_2}
= c_{j_1} \delta_{i_1i_2}\delta_{j_1j_2}
\end{equation}
for some constant $c_{j_1} > 0$ depending on $j_1$ ($i_1, i_2 = 1, \cdots, k$ and $j_1, j_2 = 0, \cdots, n$), it holds that $\|Q_l\|_{\mathcal{C}_l}$ is bounded and so is $\left|c_{ij}^{l}\right|$.

\medskip
First we claim that
\[\lim_{l \rightarrow \infty} \|Q_l\|_{\mathcal{C}_l} = 0.\]
Indeed, testing \eqref{pnief} with $Q_l$, and employing Lemmas \ref{lem_red_appen1}, \ref{lem_red_appen2} and \ref{lem_red_appen3}, the definition of the operator $i^*_{\epsilon_l}$,
and the relation $\left(\Phi_l, P_l\Psi_{il}^j\right)_{\mathcal{C}_l} = C_s  \int_{\Omega_l} w_{il}^{p-1}\psi_{il}^j\phi_l = 0$ which comes from $\Phi_l \in K_{\bl_l, \bs_l}^{\epsilon_l}$ and \eqref{def_P_epsilon}, we can deduce
\begin{align*}
&\ \|Q_l\|_{\mathcal{C}_l}^2\\
&= - \int_{\Omega_l} f_{\epsilon_l}'\left(\sum_{i=1}^k P_l w_{il}\right) \phi_l q_l - (H_l, Q_l)_{\mathcal{C}_l}\\
& \le \left[ \left( p \left\| \left(\sum_{i=1}^k P_l w_{il}\right)^{p-1} - \sum_{i=1}^k  w_{il}^{p-1} \right\|_{L^{n \over 2s}(\Omega_l)} + \epsilon^{1 + 2s\alpha_0}|\Omega_l|^{2s \over n}\right) \|\Phi_l\|_{L^{2n \over n-2s}(\Omega_l)} + \|H_l\|_{\mathcal{C}_l} \right] \|Q_l\|_{\mathcal{C}_l}\\
&\ + \left(\sum_{i=1}^k \left\|f_{\epsilon_l}'(w_{il})\right\|_{L^{n \over 2s}(\Omega_l)}\right) \|\Phi_l\|_{L^{2n \over n-2s}(\Omega_l)} \left(\sum_{i,j} \left|c_{ij}^l\right| \Big\|P_l \psi_{il}^j - \psi_{il}^j\Big\|_{L^{2n \over n-2s}(\Omega_l)}\right)\\
&= o(1) \|Q_l\|_{\mathcal{C}_l} + o(1) = o(1)
\end{align*}
where $q_l := i_{\epsilon_l}(Q_l)$.

Choose now a smooth function $\chi: \mathbb{R} \rightarrow [0,1]$ such that $\chi(x)=1 $ if $|x| \leq \delta_0/2$ and $\chi(x)=0$ if $|x| \geq \delta_0$ (where $\delta_0$ is the small number chosen in \eqref{admissible}), and set
\[\chi_l (x) = \chi(\epsilon_l^{\alpha_0} x), \quad \Phi_{hl}(x,t) = \Phi_l (x + \epsilon_l^{-\alpha_0} \sigma_{hl}, t) \chi_l (x) \quad \text{for } (x,t) \in \mathcal{C}_l\]
and $\phi_{hl} := i_{\epsilon_l} (\Phi_{hl})$ for each $h = 1, \cdots, k$.
Since $\|\Phi_{hl}\|_{\mathbb{R}^{n+1}_+} $ is bounded for each $h$, $\Phi_{hl}$ converges to $\Phi_{h\infty}$ weakly in $\mathcal{D}^{s}(\mathbb{R}_+^{n+1})$ up to a subsequence.
Using the same arguments of \cite{MP}, we can conclude that
$\Phi_{h \infty}$ is a weak solution of \eqref{extended_entire} with $(\lambda, \xi) = (\lambda_{h\infty}, 0)$ and
\[\int_{\mathbb{R}^{n+1}_+} t^{1-2s} \nabla \Phi_{h\infty} \cdot \nabla \Psi_{\lambda_{h\infty},0}^j = 0 \quad \text{for all } j=0,1,\cdots,n.\]
In order to use the result of \cite{DDS} to show $\Phi_{h\infty} = 0$,
we also need to know that $\phi_{h\infty}$ is bounded where $\phi_{h \infty}(x) := \Phi_{h \infty}(x, 0)$ for any $x \in \mathbb{R}^n$, and it is the next step we will be concerned with.
Define $\widetilde{\Phi}_L = \min\{|\Phi_{h \infty}|, L\}$ and $\widetilde{\phi}_L = \text{tr}|_{\mathbb{R}^{n+1}}\widetilde{\Phi}_L$ for any $L > 0$,
and select the test function $\widetilde{\Phi}_L^{\beta} \in \mathcal{D}^s(\mathbb{R}^{n+1})$ for \eqref{pnief} with any $\beta > 1$ to obtain
\[{4\beta \over (\beta+1)^2} \Big\|\widetilde{\Phi}_L^{\beta+1 \over 2}\Big\|^2_{\mathbb{R}^{n+1}_+} = C_s \int_{\mathbb{R}^n} f'_0(w_{\lambda_{h\infty}}) \widetilde{\phi}_L^{\beta+1} dx.\]
Then by applying the Sobolev trace embedding and taking $L \to \infty$, we can get
\begin{equation}\label{eq-boot-estimate}
\Big\| \phi_{h\infty}^{\frac{\beta+1}{2}} \Big\|_{L^{\frac{2n}{n-2s}}(\mathbb{R}^{n+1}_+)}
\leq C_{\beta} \left\| \phi_{h\infty} \right\|_{L^{\beta+1}(\mathbb{R}^{n+1}_+)}^{\frac{\beta+1}{2}}
\end{equation}
with a constant $C_{\beta} > 0$ which depends only on $\beta$.
Since we already have that $\|\phi_{h\infty}\|_{L^{\frac{2n}{n-2s}}(\mathbb{R}^{n+1}_+)}$ is finite,
we may deduce from \eqref{eq-boot-estimate} that for any $q > 1$, there is a constant $C_q > 0$ which relies only on the choice of $q$ such that
\[\| \phi_{h \infty}\|_{L^q(\mathbb{R}^{n+1}_+)} \leq C_q.\]
Now we note the expression
\begin{align*}
&\ \phi_{h \infty} (x)\\
&= \int_{\mathbb{R}^{n}} \frac{\mathfrak{a}_{n,s}}{|x-y|^{n-2s}} f'_0( w_{\lambda_{h\infty}})(y) \phi_{h \infty} (y) dy\\
&= \int_{\{|x-y| \leq 1\}} \frac{\mathfrak{a}_{n,s}}{|x-y|^{n-2s}} f'_0 (w_{\lambda_{h\infty}})(y) \phi_{h\infty}(y) dy
+ \int_{\{|x-y| > 1\}} \frac{\mathfrak{a}_{n,s}}{|x-y|^{n-2s}} f'_0 (w_{\lambda_{h\infty}})(y) \phi_{h \infty}(y) dy\\
&:= I_1(x) + I_2(x) \qquad \text{for } x \in \mathbb{R}^n.
\end{align*}
As for $I_1$, we take a very large number $q$ so that $r:=\frac{q}{q-1}$ is sufficiently close to $1$. Then we get
\begin{equation}\label{eq-A}
\begin{aligned}
I_1 & \leq C \left( \int_{\{|x-y| \leq 1\}} \frac{1}{|x-y|^{(n-2s)r}} dy \right)^{1 \over r} \left( \int_{\{|x-y| \leq 1\}} \left| f'_0 (w_{\lambda_{h\infty}})(y) \phi_{h \infty}(y)\right|^q dy \right)^{1 \over q}\\
& \leq C \| \phi_{h \infty}\|_{L^q(\mathbb{R}^{n+1})} \leq C.
\end{aligned}
\end{equation}
Considering $I_2$ we take $r$ such that $r = \frac{n}{n-2s} + \zeta$ for a small number $\zeta >0$.
Then $q$ is close to $\frac{n}{2s}$.
We further find numbers $q_1$ slightly less than $\frac{n}{2s}$ and $q_2$ such that $\frac{1}{q} = \frac{1}{q_1} + \frac{1}{q_2}$.
Then we get
\begin{equation}\label{eq-B}
I_2 \leq C \left( \int_{\{|x-y| > 1\}} \frac{1}{|x-y|^{(n-2s)r}} dy \right)^{1 \over r} \left\| f'_0 (w_{\lambda_{h\infty}})\right\|_{L^{q_1}(\mathbb{R}^{n+1})} \left\| \phi_{h\infty}\right\|_{L^{q_2}(\mathbb{R}^{n+1})} \leq C.
\end{equation}
The estimates \eqref{eq-A} and \eqref{eq-B} show that $\phi_{h \infty}$ is bounded.
Now we may achieve that $\Phi_{h \infty} = 0$ by the classification of the solutions for the linear problem \eqref{extended_entire} obtained in \cite{DDS}.
In summary, we proved that
\begin{equation}\label{phn0w}
\lim_{l \to \infty} \Phi_{hl} = 0 \quad \text{weakly in } \mathcal{D}^{s} (\mathbb{R}_{+}^{n+1}) \quad \text{and} \quad \lim_{l \to \infty} \phi_{hl} = 0 \quad \text{strongly in } L^q(\Omega) \text{ for } 1 \le q < p+1
\end{equation}
($h = 1,\cdots, k$).

\medskip
Consequently, \eqref{phn0w} yields
\[\lim_{l \to \infty} \int_{\Omega_l} f'_{\epsilon}\left(\sum_{i=1}^k P_l w_{il}\right) \phi_l^2 =0.\]
Hence by testing $\Phi_l$ to \eqref{pnief} we may deduce that
\[\lim_{l \to \infty} \|\Phi_l\|_{\mathcal{C}_l} =0.\]
However it contradicts to \eqref{en0ln}. This proves the validity of the lemma.
\end{proof}

For each $\epsilon > 0$ sufficiently small and $(\bls) \in O^{\delta_0}$ fixed,
the linear operator $L_{\bls}^{\epsilon}: K_{\bls}^{\epsilon} \rightarrow K_{\bls}^{\epsilon}$ has the form $\text{Id} + \mathcal{K}$ where Id is the identity operator and $\mathcal{K}$ is a compact operator on $K_{\bls}^{\epsilon}$,
because the trace operator $i_{\epsilon}: H^s_{0, L}(\mathcal{C}_{\epsilon}) \to L^q(\Omega_{\epsilon}) \subset L^{p+1}(\Omega_{\epsilon})$ is compact whenever $q \in [1, p+1)$.
Therefore, by the Fredholm alternative, it is a Fredholm operator of index 0.
However Lemma \ref{lem_linear_nondeg} implies that it is also an injective operator.
Consequently, we have the following result.

\begin{prop}\label{prop_fredholm}
The inverse $(L_{\bls}^{\epsilon})^{-1}$ of $L_{\bls}^{\epsilon}: K_{\bls}^{\epsilon} \rightarrow K_{\bls}^{\epsilon}$ exists for any $\epsilon > 0$ small and $(\bls) \in O^{\delta_0}$.
Besides, its operator norm is uniformly bounded in $\epsilon$ and $(\bls) \in O^{\delta_0}$, if $\epsilon$ is small enough.
\end{prop}

The previous proposition gives us that
\begin{prop}\label{prop_remainder}
For any sufficiently small $\delta_0 >0$ chosen fixed, we can select $\epsilon_0 > 0$ such that for any $\epsilon \in (0,\epsilon_0)$ and $(\bl, \bs) \in O^{\delta_0}$, there exists a unique $\Phi_{\bl, \bs}^{\epsilon} \in K_{\bl, \bs}^{\epsilon}$ satisfying
\[\Pi_{\bl, \bs}^{\epsilon}\left\{C_s^{-1}\left( \sum_{i=1}^{k} P_{\epsilon} W_i + \Phi_{\bl, \bs}^{\epsilon}\right) - i_{\epsilon}^{*} \left[ f_{\epsilon} \left(\sum_{i=1}^{k} P_{\epsilon} w_i + i_{\epsilon} \left(\Phi_{\bl, \bs}^{\epsilon}\right) \right) \right] \right\} =0\]
and
\begin{equation}\label{Phi_est}
\| \Phi_{\bl, \bs}^{\epsilon} \|_{\mathcal{C}_{\epsilon}} \leq C \epsilon^{\eta_0} \quad \text{with} \quad
\eta_0 := \left\{\begin{array}{ll} \frac{1}{2} + 2s \alpha_0 &\text{if } n \geq 6s,\\
\frac{1}{2} + (1+\delta)s\alpha_0 &\text{if } 4s < n < 6s,
\end{array}\right.
\end{equation}
where $\delta > 0$ is chosen to satisfy $(4+2\delta)s < n$.
{Furthermore, the map $(\bls) \mapsto \Phi_{\bl, \bs}^{\epsilon}$ is $C^1(O^{\delta_0})$}.
\end{prop}
\begin{proof}
Define
\begin{align*}
N_{\epsilon}(\Phi) &= \Pi_{\bls}^{\epsilon} \circ i_{\epsilon}^* \left[f_{\epsilon}\left(\sum_{i=1}^{k} P_{\epsilon} w_i + i_{\epsilon} (\Phi)\right) - f_{\epsilon}\left(\sum_{i=1}^{k} P_{\epsilon} w_i\right) - f'_{\epsilon}\left(\sum_{i=1}^{k} P_{\epsilon} w_i\right) i_{\epsilon} (\Phi)\right],\\
R_{\epsilon} &= \Pi_{\bls}^{\epsilon} \left( i_{\epsilon}^*\left[ f_{\epsilon}\left(\sum_{i=1}^{k} P_{\epsilon} w_i\right) \right] - C_s^{-1}\sum_{i=1}^{k} P_{\epsilon} W_i\right)\\
\intertext{and}
T_{\epsilon}(\Phi) &= (L_{\bls}^{\epsilon})^{-1}(N_{\epsilon}(\Phi) + R_{\epsilon}) \quad \text{for } \Phi \in K_{\bls}^{\epsilon},
\end{align*}
where the set $K_{\bls}^{\epsilon}$ and the operator $\Pi_{\bls}^{\epsilon}$ are defined in \eqref{ortho} and the sentence following it.
Also, the well-definedness of the inverse of the operator $L_{\bls}^{\epsilon}$ is guaranteed by Proposition \ref{prop_fredholm}.
By Lemmas \ref{lem-projection-robin}, \ref{lem_red_appen1} and \ref{lem_red_appen3}, we have $\|R_{\epsilon}\|_{\mathcal{C}_{\epsilon}} = O\left(\epsilon^{\eta_0}\right)$ as $\epsilon \to 0$, and
from this we can conclude that $T_{\epsilon}$ is a contraction mapping on
$\mathcal{K}_{\bls}^{\epsilon} := \{\Phi \in K_{\bls}^{\epsilon}: \|\Phi\|_{\mathcal{C}_\epsilon} \le C \epsilon^{\eta_0}\}$  for some small $C > 0$, which implies the existence of a unique fixed point of $T_{\epsilon}$ on $\mathcal{K}_{\bls}^{\epsilon}$.
It is easy to check that this fixed point is our desired function $\Phi_{\bls}^{\epsilon}$.
For the detailed treatment of the argument, we refer to \cite[Proposition 1.8]{MP} (see also \cite[Proposition 3]{DDM}).
\end{proof}

\subsection{The reduced problem}
We set $\alpha_0 =\frac{1}{n-4s}$.
Notice that equation \eqref{equation-dilated} for each fixed $\epsilon > 0$ has the variational structure, that is, $U \in H_{0,L}^s(\Omega)$ is a weak solution of the equation if and only if it is a critical point of the energy functional
\begin{equation}\label{energy_functional}
E_{\epsilon}(U) := {1 \over 2C_s} \int_{\mathcal{C}_{\epsilon}} t^{1-2s}|\nabla U|^2 - \int_{\Omega_{\epsilon} \times \{0\}} F_{\epsilon}(i_{\epsilon}(U))
\end{equation}
where $F_{\epsilon}(t) := \int_0^t f_{\epsilon}(t) dt$.
In fact, thanks to the Sobolev trace embedding $i_{\epsilon}: H_{0,L}^s(\mathcal{C}_{\epsilon}) \to L^{p+1}(\Omega_{\epsilon})$,
we can obtain that $E_{\epsilon}: H_{0,L}^s(\mathcal{C}_{\epsilon}) \to \mathbb{R}$ is a $C^1$-functional and
\[E'_{\epsilon}(U)\Phi = {1 \over C_s}\int_{\mathcal{C}_{\epsilon}} t^{1-2s} \nabla U \cdot \nabla \Phi - \int_{\Omega_{\epsilon} \times \{0\}} f_{\epsilon}(i_{\epsilon}(U))i_{\epsilon}(\Phi) \quad \text{for any } \Phi \in H_{0,L}^s(\mathcal{C}_{\epsilon}).\]
Using Proposition \ref{prop_remainder}, we can define a localized energy functional defined in the admissible set $O^{\delta_0}$ in \eqref{admissible}:
\begin{equation}\label{energy_functional_local}
\widetilde{E}_{\epsilon}(\bls) := E_{\epsilon} \left(\sum_{i=1}^k P_{\epsilon} W_{\lambda_i, \frac{\sigma_i}{\epsilon^{\alpha_0}}} +
\Phi_{\bls}^{\epsilon}\right)
\end{equation}
for $(\bls) = ((\lambda_1, \cdots, \lambda_k), (\sigma_1, \cdots, \sigma_k)) \in O^{\delta_0}$. Then we can obtain the following important properties of $\widetilde{E}_{\epsilon}$.

\begin{prop}\label{prop_main_est}
Suppose $\epsilon > 0$ is sufficiently small.

\noindent (1) If $\widetilde{E}_{\epsilon}'(\bl^{\epsilon}, \bs^{\epsilon}) = 0$ for some $(\bl^{\epsilon}, \bs^{\epsilon}) \in O^{\delta_0}$, then the function $U_{\epsilon} := \sum_{i=1}^k P_{\epsilon} W_{\lambda_i^{\epsilon}, \frac{\sigma_i^{\epsilon}}{\epsilon^{\alpha_0}}} + \Phi_{\bl^{\epsilon}, \bs^{\epsilon}}^{\epsilon}$ is a solution of \eqref{equation-dilated-m}.
Hence one concludes that a dilated function $V_{\epsilon}(z) := \epsilon^{-{n-2s \over 2(n-4s)}} U_{\epsilon}(\epsilon^{-{1 \over n-4s}}z)$ defined for $z \in \mathcal{C}$ is a solution of \eqref{equation-construction}.

\noindent
(2) Recall the number $\eta_0$ chosen in \eqref{Phi_est}. Then it holds that
\begin{equation}\label{energy_local_expansion}
\widetilde{E}_{\epsilon}(\bls) = {ks \over n}c_0 + {1 \over 2} \Upsilon_k(\bls) \epsilon^{n-2s \over n-4s} + o(\epsilon^{n-2s \over n-4s})
\end{equation}
in $C^1$-uniformly in $(\bls) \in O^{\delta_0}$. Here $\Upsilon_k$ is the function introduced in \eqref{upsilon_brezis} and
\begin{equation}\label{constant_C}
c_0 = \int_{\mathbb{R}^n} w_{1,0}^{p+1}(x)dx
\end{equation}
(recall that $w_{1,0}$ is the function obtained by taking $(\lambda, \xi) = (1,0)$ in \eqref{bubble}).
\end{prop}

\noindent We postpone its proof in Appendix \ref{subsec_proof_of_main_est}.

\subsection{Definition of stable critical sets and conclusion of the proofs of Theorems \ref{thm-m-multipeak}, \ref{domain-construction}}\label{subsec_proof_reduction}
We recall the definition of stable critical sets which was introduced by Li \cite{Li2}. \begin{defn}\label{stable_critical}
Suppose that $D \subset \mathbb{R}^n$ is a domain and $g$ is a $C^1$ function in $D$.
We say that a bounded set $\Lambda \subset D$ of critical points of $f$ is a stable critical set if there is a number $\delta > 0$ such that $\|g-h\|_{L^{\infty}(\Lambda)} + \|\nabla (g - h)\|_{L^{\infty}(\Lambda)} < \delta$ for some $h \in C^1(D)$ implies the existence of a critical point of $h$ in $\Lambda$.
\end{defn}

Now we are ready to prove Theorem \ref{thm-m-multipeak}.
\begin{proof}[Proof of Theorem \ref{thm-m-multipeak}]
By the virtue of Proposition \ref{prop_main_est} (2) and Definition \ref{stable_critical},
we can find a pair $(\bl^{\epsilon}, \bs^{\epsilon}) \in \Lambda_k$ which is a critical point of the reduced energy functional $\widetilde{E}_{\epsilon}$ (defined in \eqref{energy_functional}) given $0 < \epsilon < \epsilon_0$ for some $\epsilon_0$ small enough.
From this fact and Proposition \ref{prop_main_est} (1), we obtain a solution $v_{\epsilon} := i_{\epsilon}(V_{\epsilon})$ of \eqref{a12uu} for $\epsilon \in (0, \epsilon_0)$.

Also, by using the dilation invariance of \eqref{entire_nonlocal} and the trace inequality \eqref{eq-sharp-trace}, we see that
$v_{\epsilon} = \sum_{i=1}^k P_1 w_{\epsilon^{\alpha_0}\lambda_i^{\epsilon}, \sigma_i^{\epsilon}} + \widetilde{\phi}_{\bl^{\epsilon}, \bs^{\epsilon}}^{\epsilon}$ in $\Omega$
where $\|\widetilde{\phi}_{\bl^{\epsilon}, \bs^{\epsilon}}^{\epsilon}\|_{L^{2n \over n-2s}(\Omega)} \le C \|\Phi_{\bl^{\epsilon}, \bs^{\epsilon}}^{\epsilon}\|_{\mathcal{C}_{\epsilon}} = O(\epsilon^{\eta_0})$ ($\eta_0 > 0$ is chosen in \eqref{Phi_est}).
From this fact, if we test \eqref{a12uu} with $\widetilde{\phi}_{\bls}^{\epsilon}$ and use \eqref{def_P_epsilon2}, we can deduce $\|\widetilde{\phi}_{\bls}^{\epsilon}\|_{H^s(\Omega)} = o(1)$.
Furthermore, it is obvious that there exists a point
$(\bl^0, \bs^0) \in \Lambda_k$ such that
$(\bl^{\epsilon}, \bs^{\epsilon}) \to (\bl^0, \bs^0)$ up to a subsequence.
This completes the proof of Theorem \ref{thm-m-multipeak}.
\end{proof}

\begin{proof}[Proof of Theorem \ref{domain-construction}]
We recall that $G$ and $\tau$ are Green's function and the Robin function of $\mathcal{A}_{s}$ in $\Omega$ with the zero Dirichlet boundary condition, respectively (see \eqref{Green_Omega} and \eqref{Robin_Omega}).
To emphasize the dependence of $G$ and $\tau$ on the domain $\Omega$,
we append the subscript $\Omega$ in $G$ and $\tau$ so that $G = G_{\Omega}$ and $\tau = \tau_{\Omega}$.

If a sequence of domains $\{\Omega_{\epsilon}: \epsilon > 0\}$ satisfies $\lim_{\epsilon \rightarrow 0} \Omega_{\epsilon} = \Omega$ and $\Omega_{\epsilon_1} \subset \Omega_{\epsilon_2}$ for any $\epsilon_1 < \epsilon_2$, then $\tau_{\Omega_{\epsilon}}$ converges to $\tau_{\Omega}$ in $C^{1}_{\text{loc}}(\Omega)$.
In order to prove this statement, we first note that the maximum principle (Lemma \ref{lem-maximum}, cf. \cite[Lemma 3.3]{T2}) ensures that $\tau_{\Omega_{\epsilon}}$ is monotone increasing as $\epsilon \to 0$ and tends to $\tau_{\Omega}$ pointwise.
Then we can deduce from {Lemma \ref{lem-property-H}} that it converges also in $C^{1}$ on any compact set of $\Omega$.
Similar arguments also apply to show that $G_{\Omega_\epsilon}(x,y)$ converges to $G_{\Omega}(x,y)$ in $C^{1}$ locally on $\{(x,y) \in \Omega_{\epsilon}^2: x \neq y\}$.
The rest part of the proof goes along the same way to \cite{MP} or \cite{EGP},
where the authors considered domains $\Omega_{\epsilon}$ consisting of $k$ disjoint balls and thin strips liking them whose widths are $\epsilon$.
\end{proof}

\section{The subcritical problem}\label{sec_subcrit}
We are now concerned in the proofs of Theorem \ref{thm-m-limit-sub} and Theorem \ref{thm-m-multipeak2}.
Since many steps of the proofs for the previous theorems can be modified easily for problem \eqref{eqtn-subcritical}, we only stress the parts where some different arguments should be introduced.

\medskip
Remind that $\mu_{\epsilon} = \mathfrak{c}_{n,s}^{-1} \sup_{x \in \Omega} u_{\epsilon}(x)$ and $x_{\epsilon} \in \Omega$ is a point which satisfies $\mu_{\epsilon}= \mathfrak{c}_{n,s}^{-1}u_{\epsilon}(x_{\epsilon})$. (See Lemma \ref{lem_mu_e}.)
We also define the  functions $b_{\epsilon}$ and $B_{\epsilon}$ with their domains $\Omega_{\epsilon}$ and $\mathcal{C}_{\epsilon}$ as in \eqref{aexme} and \eqref{aezme}, replacing the scaling factor ${2 \over n-2s} = {p-1 \over 2s}$ by ${p-1-\epsilon \over 2s}$.
Then $b_{\epsilon}$ converges to $w_1$ pointwisely.

In order to get the uniform boundedness result, we first need the following bound of $\mu_{\epsilon}$.
\begin{lem}\label{forso}
There exists a constant $C>0$ such that
\begin{equation}\label{dmn1e}
C \leq \mu_{\epsilon}^{-(\frac{n}{2s}-1)\epsilon} \quad \text{for all } \epsilon >0.
\end{equation}
\end{lem}
\begin{proof}
Since $b_{\epsilon}$ converges to $w_1$ pointwise, we have
\[\int_{B_n (0,1)} b_{\epsilon}^{p+1-\epsilon} \geq C\]
by Fatou's lemma. Note that
\begin{align*}
\int_{B_n (0,1)} b_{\epsilon}^{p+1-\epsilon}(x) dx &= \int_{B_n (0,1)} \mu_{\epsilon}^{-(p+1-\epsilon)} u_{\epsilon}^{p+1-\epsilon} \left(\mu_{\epsilon}^{-\frac{p-1-\epsilon}{2s}} x + x_{\epsilon}\right) dx
\\
&\leq \int_{\Omega} \mu_{\epsilon}^{-(p+1-\epsilon)} \mu_{\epsilon}^{\frac{n}{2s}(p-1-\epsilon)} u_{\epsilon}^{p+1-\epsilon} (x) dx
\\
&\leq C \mu_{\epsilon}^{-(\frac{n}{2s}-1)\epsilon}.
\end{align*}
Combining these two estimates completes the proof.
\end{proof}
Next, as before we denote by $d_{\epsilon}$ and $D_{\epsilon}$ the Kelvin transforms of  $b_{\epsilon}$ and $B_{\epsilon}$ (see \eqref{eq-kelvin-b} and \eqref{eq-kelvin-B}). Then the function $D_{\epsilon}$ satisfies
\[\left\{ \begin{array}{ll} \textrm{div}(t^{1-2s}\nabla D_{\epsilon} )(z) = 0 &\quad \text{in} ~\kappa( \mathcal{C}_{\epsilon}),
\\
\partial_{\nu}^s D_{\epsilon} = |x|^{-\epsilon (n-2s)} D_{\epsilon}^{p-\epsilon} &\quad \text{in}~ \kappa(\Omega_{\epsilon} \times \{0\}),
\\
D_{\epsilon} > 0 &\quad \text{in}~ \kappa(\mathcal{C}_{\epsilon}),
\\
D_{\epsilon} =0 &\quad \text{on}~\kappa( \partial_L \mathcal{C}_{\epsilon}).
\end{array}
\right.\]
From \eqref{aexme}, we have $|x| \geq C\mu_{\epsilon}^{-\frac{p-1-\epsilon}{2s}}$ for $x \in \kappa (\Omega_{\epsilon})$, hence Lemma \ref{forso} yields
\[|x|^{-\epsilon (n-2s)} \leq \mu_{\epsilon}^{\frac{(p-1-\epsilon)}{2s}(n-2s)\epsilon} \leq  C \quad \text{for all } x \in \kappa (\Omega_{\epsilon}).\]
By this fact we may use Lemma \ref{lem-harnack-1} and Lemma \ref{lem-harnack-2} and the proof of Proposition \ref{prop-uniform-bound} to find $C>0$ such that
\begin{equation}\label{eqtn-sub-bound}
u_{\epsilon}(x) \leq C w_{\mu_{\epsilon}^{-{2 \over n-2s}},x_{\epsilon}} (x) \quad \text{for all } \epsilon > 0 \text{ and } x \in \Omega.
\end{equation}
Now we need to get a sharpened bound of $\mu_{\epsilon}$.
Considering both \eqref{eqtn-sub-bound} with Proposition \ref{prop-sub-estimate} simultaneously,
we can prove the following lemma.
\begin{lem}\label{lem-subcritical-control}
(1) There exists a constant $C >0$ such that
\[\epsilon \leq C \mu_{\epsilon}^{-2-2 \epsilon} \quad \text{for } \epsilon >0 \text{ small}.\]
\noindent (2) We have
\begin{equation}\label{mee10}
\lim_{\epsilon \rightarrow 0} \mu_{\epsilon}^{\epsilon} = 1.
\end{equation}
\end{lem}
\begin{proof}
As in the proof of Lemma \ref{eq-mu-epsilon}, we take a small number $\delta>0$.
Recall also the definition of $\mathcal{I}(\Omega,r)$ and $\mathcal{O}(\Omega,r)$ (see \eqref{eq-I-omega} and \eqref{eq-O-omega}).
Then we see that the left-hand side of \eqref{poho-prop-estimate} is bounded below, i.e.,
\begin{equation}\label{lcxnu}
\left(\frac{n}{p+1-\epsilon} - \frac{n-2s}{2}\right) \int_{\mathcal{I}(\Omega,{\delta}) \times \{0\}} |U_{\epsilon}(x,0)|^{p+1-\epsilon} dx \geq C \epsilon
\end{equation}
for some constant $C>0$.
\

On the other hand, using \eqref{eqtn-sub-bound} we deduce
\begin{equation}\label{oodup}
\begin{aligned}
\left( \int_{ \Omega} u_{\epsilon}^{p-\epsilon} dx \right)^2 \leq&~ C \left( \int_{\mathbb{R}^n} {w_{\mu_{\epsilon}^{-{2 \over n-2s}},x_{\epsilon}}^{p-\epsilon}} (x) dx \right)^2
\\
\leq&~ C \mu_{\epsilon}^{-2(p-\epsilon)} \left( \int_{\mathbb{R}^n} \left( \mu_{\epsilon}^{-\frac{4}{n-2s}}  + |x|^2 \right)^{-\frac{n-2s}{2}(p-\epsilon)} dx \right)^2
\\
\leq& ~C \mu_{\epsilon}^{-2 - 2 \epsilon}.
\end{aligned}
\end{equation}
Since $ x_0 \notin \mathcal{O}(\Omega,{2\delta})$ we have $w_{\mu_{\epsilon}, x_{\epsilon}} (x) \leq C \mu_{\epsilon}^{-1}$ for $x \in \mathcal{O}(\Omega,\delta)$.
It yields, for a fixed large number $q>0$, that
\begin{equation}\label{odu2p}
\left(\int_{\mathcal{O}(\Omega,\delta)} u_{\epsilon}^{(p-\epsilon)q} dx\right)^{2/q} \leq C \mu_{\epsilon}^{-2(p-\epsilon)}.
\end{equation}
Now we inject the estimates \eqref{lcxnu}, \eqref{oodup} and \eqref{odu2p} to the inequality in the statement of Proposition \ref{prop-sub-estimate} to get
\[C \epsilon \leq C\left[ \mu_{\epsilon}^{-2 - 2\epsilon} + \mu_{\epsilon}^{-2(p-\epsilon)}\right] \leq 2 C \mu_{\epsilon}^{-2 - 2 \epsilon},\]
which proves the first statement of the lemma.
Using Taylor's theorem, we get
\[|\mu_{\epsilon}^{\epsilon} -1| \leq \sup_{0\leq t \leq 1}\epsilon \mu_{\epsilon}^{t\epsilon} \log (\mu_{\epsilon}) = O (\mu_{\epsilon}^{-1-2\epsilon}  \log (\mu_{\epsilon}) ).\]
It proves $\lim_{\epsilon \rightarrow 0} \mu_{\epsilon}^{\epsilon} = 1$ because  $\mu_{\epsilon}$ goes to infinity. Now the proof is complete.
\end{proof}

We now prove Theorems \ref{thm-m-limit-sub} and \ref{thm-m-multipeak2}.
\begin{proof}[Proof of Theorem \ref{thm-m-limit-sub}]
By definition we have
\[\mathcal{A}_{s} (\|u_{\epsilon}\|_{L^{\infty}(\Omega)} u_{\epsilon})(x) = \mathfrak{c}_{n,s} \mu_{\epsilon} u_{\epsilon}^{p-\epsilon}(x).\]
Note from $p=\frac{n+2s}{n-2s}$ that
\[\int_{\Omega}\mathfrak{c}_{n,s} \mu_{\epsilon} u_{\epsilon}^{p-\epsilon}(x)dx
= \int_{\Omega} \mathfrak{c}_{n,s}\mu_{\epsilon}^{p+1-\epsilon} b_{\epsilon}^{p-\epsilon}\left( \mu_{\epsilon}^{p-1-\epsilon \over 2s} (x-x_{\epsilon})\right)dx
= \int_{\Omega_{\epsilon}} \mathfrak{c}_{n,s} \mu_{\epsilon}^{({n \over 2s} -1)\epsilon}  b_{\epsilon}^{p-\epsilon}(x) dx.\]
Here, from Lemma \ref{lem-subcritical-control} and the dominated convergence theorem with the fact that $b_\epsilon$ converges to $w_1$ pointwise, we conclude that
\[\lim_{\epsilon \rightarrow 0} \int_{\Omega} \mathfrak{c}_{n,s}\mu_{\epsilon} u_{\epsilon}^{p-\epsilon}(x)dx = \int_{\mathbb{R}^n} \mathfrak{c}_{n,s} w_1^{p}(x) dx = \mathfrak{b}_{n,s}\]
(see \eqref{eq-beta}). Now the first statement follows as in the proof of Theorem \ref{thm-m-limit}.

The proof of the second statement can be performed similarly to the proof of Theorem \ref{thm-m-location}.
The constant $\mathfrak{g}_{n,s}$ is given by
\begin{equation}\label{eq-constant-L}
\mathfrak{g}_{n,s} = \frac{4n}{2C_s} \mathcal{S}_{n,s}^{n/s} \mathcal{D}_{n,s} \mathfrak{a}_{n,s}\mathfrak{b}_{n,s}^2.
\end{equation}
The proof is complete.
\end{proof}

\begin{proof}[Proof of Theorem \ref{thm-m-multipeak2}]
This theorem can be proved in a similar way to the proof of Theorem \ref{thm-m-multipeak}.
In this case, if we take $\alpha_0 = \frac{1}{n-2s}$ and $f_{\epsilon}(U) = U_+^{p-\epsilon}$, then an analogous result of Proposition \ref{prop_main_est} holds with $\widetilde{\Upsilon}$ (refer to \eqref{upsilon_subcritical}).
Therefore there exists a family of solutions which concentrate at a critical point of $\widetilde{\Upsilon}$.
\end{proof}

\medskip
\noindent \textbf{Acknowledgments}

\medskip
The authors thank to the referee for the comments which improved the exposition of the paper.
W. Choi was supported by the Global Ph.D Fellowship of the Government of South Korea 300-20130026.
He wants to express gratitude for his advisor Prof. Ponge for the support and encouragement.
S. Kim was partially supported by FONDECYT Grant 3140530, Chile.
He is also grateful to Prof. Byeon and Department of Mathematical Sciences in KAIST for their financial support through research grant G04130010 during his study in KAIST, and to Prof. Pistoia for her support and hospitality on his visit to La Sapienza - Universit\`a di Roma.
Ki-Ahm Lee was supported by the IT R\&D program of MSIP/KEIT. [No.10047212, Development of homomorphic encryption supporting arithmetics on ciphertexts of size less than 1 kB and its applications].

\appendix
\section{Proof of Proposition \ref{prop-sub-estimate}}\label{sec_appen_a}
This section is devoted to present the proof of Proposition \ref{prop-sub-estimate}, namely, the following proposition.
\begin{prop}
Suppose that $U \in {H^s_{0,L}(\mathcal{C})}$ is a solution of problem \eqref{u0inc} with $f$ such that $f$ has the critical growth and $f = F'$ for a function $F \in C^1 (\mathbb{R})$.
Then, for each $ \delta >0$ and $q >\frac{n}{s}$ there is a constant $C= C(\delta, q) >0$ such that
\begin{equation}\label{eq_local_poho}
\begin{aligned}
&\min_{r \in [\delta, 2\delta]}\left|n \int_{\mathcal{I}(\Omega,{r/2})\times\{0\}} F(U) dx - \left(\frac{n-2s}{2}\right) \int_{ \mathcal{I}(\Omega,{r/2}) \times \{0\}} U f(U) dx\right|
\\
&\leq C \left[ \left(\int_{\mathcal{O}(\Omega,{2\delta})\times \{0\}} |f(U)|^{q} dx\right)^{2 \over q} + \int_{\mathcal{O}(\Omega,2\delta)\times \{0\}} |F(U)| dx + \left( \int_{\mathcal{I}(\Omega, \delta/2) \times \{0\}} |f(U)| dx \right)^2 \right]
\end{aligned}
\end{equation}
where $\mathcal{I}$ and $\mathcal{O}$ is defined in  \eqref{eq-I-omega} and \eqref{eq-O-omega}.
\end{prop}
\begin{proof}
Recall the local form of the Pohozaev identity
\begin{equation}\label{eq-poho-root}
\textrm{div} \biggl\{ t^{1-2s} \langle z,\nabla U \rangle \nabla U - t^{1-2s} \frac{|\nabla U|^2}{2} z \biggr\} + \biggl( \frac{n-2s}{2} \biggr) t^{1-2s}|\nabla U |^2=0
\end{equation}
and define the following sets:
\begin{align*}
D_r &= \left\{ z \in \mathbb{R}^{n+1}_{+} : \textrm{dist}(z,\mathcal{I}(\Omega,r) \times \{0\}) \leq r/2\right\},\\
\partial D_r^{+} &= \partial D_r \cap \left\{(x,t) \in \mathbb{R}^{n+1}: t>0\right\} \quad \text{and} \quad E_{\delta} = \bigcup_{r = \delta}^{2\delta} \partial D_{r}^{+}.
\end{align*}
Note that $\partial D_r = \partial D_r^{+} \cup (\mathcal{I}(\Omega,r/2)\times \{0\})$.
Fix a small number $\delta >0$.
We integrate the identity \eqref{eq-poho-root} over $D_r$ for each $r \in (0,2\delta]$ to derive
\begin{multline}\label{dxdy0}
 \int_{ \partial D_{r}^{+}} t^{1-2s}\left\langle \langle z, \nabla U \rangle \nabla U - z \frac{|\nabla U|^2}{2}, \nu\right\rangle dS
+ C_s \int_{\mathcal{I}(\Omega,{r/2}) \times \{0\}} \langle x, \nabla_x U \rangle \partial_{\nu}^s U  dx
\\
\quad = - \left( \frac{n-2s}{2}\right) \int_{ D_r} t^{1-2s}|\nabla U|^2 dx dt.
\end{multline}
In view of Lemmas 4.4 and 4.5 of \cite{CS2}, one can deduce that the $i$-th component $\partial_{x_i} U$ of $\nabla_x U$ is H\"older continuous in $\overline{D_r}$ for each $i = 1, \cdots, n$, which justifies the above formula.
By using $\partial_{\nu}^s U = f(U)$ and performing integration by parts, we derive
\begin{align*}
\int_{\mathcal{I}(\Omega,{r/2}) \times \{0\}} \langle x, \nabla_x U \rangle \partial_{\nu}^s U dx
&=\int_{\mathcal{I}(\Omega,{r/2}) \times\{0\}} \langle x, \nabla_x F(U) \rangle dx
\\
&= -n \int_{\mathcal{I}(\Omega,{r/2})\times\{0\}} F(U) dx +\int_{\partial \mathcal{I}(\Omega,{r/2})\times \{0\}} \langle x,\nu \rangle F(U) dS_x
\end{align*}
and
\[\int_{D_r} t^{1-2s} |\nabla U|^2 dx dt= C_s\int_{ \mathcal{I}(\Omega,{r/2}) \times \{0\}} U f(U) dx + \int_{\partial D_r^{+}} t^{1-2s} U \frac{\partial U}{\partial \nu} dS.\]
Then \eqref{dxdy0} is written as
\begin{equation}\label{eq-local-poho}
\begin{aligned}
&\ C_s \left\{n \int_{\mathcal{I}(\Omega,{r/2})\times\{0\}} F(U) dx - \left(\frac{n-2s}{2}\right) \int_{ \mathcal{I}(\Omega,{r/2}) \times \{0\}} U f(U) dx\right\}\\
& \qquad \qquad = \int_{ \partial D_{r}^{+}} t^{1-2s} \left[ \left< \langle z, \nabla U \rangle \nabla U - z \frac{|\nabla U|^2}{2}, \nu\right> dS + \left(\frac{n-2s}{2}\right) U \frac{\partial U}{\partial \nu} \right] dS
\\
& \qquad \qquad \  +  \int_{\partial \mathcal{I}(\Omega,{r/2})\times \{0\}} \langle x,\nu \rangle F(U) dS_x.
\end{aligned}
\end{equation}
From this identity we get
\begin{multline*}
\left|n \int_{\mathcal{I}(\Omega,{r/2})\times\{0\}} F(U) dx-\left(\frac{n-2s}{2}\right) \int_{\mathcal{I}(\Omega,{r/2}) \times \{0\}} U f(U) dx\right|
\\
\leq C \int_{\partial D_r^{+}} t^{1-2s} (|\nabla U|^2 + U^2 ) dS + \int_{\partial \mathcal{I}(\Omega,{r/2}) \times\{0\}} \langle x,\nu \rangle  F(U) dS_x.
\end{multline*}
We integrate this identity with respect to $r$ over an interval $[\delta, 2\delta]$ and then use the Poincar\'e inequality. Then we observe
\begin{multline*}
\min_{r \in [\delta ,2\delta]} \left|n \int_{\mathcal{I}(\Omega,{r/2})\times\{0\}}  F(U) dx-\left(\frac{n-2s}{2}\right) \int_{ \mathcal{I}(\Omega,{r/2}) \times \{0\}} U f(U) dx \right|\\
\leq C \int_{E_{\delta}}t^{1-2s} |\nabla U|^2 dz  + C \int_{\mathcal{O}(\Omega,{\delta})}|F(U)(x,0)| dx.
\end{multline*}
We only need to estimate the first term of the right-hand side of the previous inequality since the second term is already one of the terms which constitute the right-hand side of \eqref{eq_local_poho}. Note that
\begin{equation}\label{eq-U-decom}
\nabla_z U(z) = \int_{\Omega} \nabla_z  G_{\mathbb{R}^{n+1}_{+}} (z,y) f(U)(y,0) dy - \int_{\Omega} \nabla_z H_{\mathcal{C}}(z,y) f(U)(y,0) dy
\end{equation}
for $z \in E_{\delta}$.

Let us deal with the last term of \eqref{eq-U-decom} first.
Admitting the estimation
\begin{equation}\label{eq-appendix-H-pre}
\sup_{y \in \Omega} \int_{E_{\delta}} t^{1-2s} |\nabla_z H_{\mathcal{C}}(z,y)|^2 dz \le C
\end{equation}
for a while and using H\"older's inequality, we get
\begin{equation}\label{eq-appendix-H}
\begin{aligned}
&\ \int_{E_{\delta}} t^{1-2s} \left( \int_{\Omega} |\nabla_z H_{\mathcal{C}}(z,y) f(U)(y,0)| dy \right)^2 dz
\\
&\le \left(\sup_{y \in \Omega} \int_{E_{\delta}} t^{1-2s} |\nabla_z H_{\mathcal{C}}(z,y)|^2 dz\right) \left( \int_{\Omega} |f(U)(y,0)| dy\right)^2 \le C \left( \int_{\mathcal{I}(\Omega,{\delta}) \cup \mathcal{O}(\Omega,{\delta})} |f(U)(y,0)| dy\right)^2
\\
&\le C \left[\left(\int_{\mathcal{O}(\Omega,{2\delta})} |f(U)(y,0)|^q dy\right)^{2 \over q} + \left( \int_{\mathcal{I}(\Omega, \delta/2) } |f(U)(y,0)| dy \right)^2\right],
\end{aligned}
\end{equation}
which is a part of the right-hand side of \eqref{eq_local_poho}.

The validity of \eqref{eq-appendix-H-pre} can be reasoned as follows.
First of all, if $y$ is a point in $\Omega$ such that $\dist(y, E_{\delta}) \leq \delta/ 2$,
then it automatically satisfies that $\dist (y, \partial \Omega) \geq \delta/2$ from which we know
\[\sup_{\dist(y,\partial \Omega) \ge \delta/2} \left(\int_{E_{\delta}} t^{1-2s} |\nabla_z H_{\mathcal{C}}(z,y)|^2 dz\right) \le \sup_{\dist(y, \partial \Omega) \geq \delta/2} \left(\int_{\mathcal{C}} t^{1-2s} |\nabla_z H_{\mathcal{C}} (z,y) |^2 dz\right) \le C.\]
See the proof of Lemma \ref{lem-H-existence} for the second inequality.
Meanwhile, in the complementary case $\dist (y, E_{\delta}) > \delta /2$, we can assert that
\begin{equation}\label{eq-h-nabla}
\int_{E_{\delta}} t^{1-2s} |\nabla_z H_{\mathcal{C}}(z,y)|^2 dz
\leq C \left(\int_{N(E_{\delta}, \delta/4)} t^{1-2s} |H_{\mathcal{C}}(z,y)|^2 dz\right)
\end{equation}
where $N(E_{\delta}, \delta/4) := \{ z \in \mathcal{C}: \dist (z, E_{\delta}) \leq \delta /4\}$.
To show this, we recall that $H_{\mathcal{C}}$ satisfies
\begin{equation}\label{eq-appendix-h}
\left\{\begin{array}{ll}
\textrm{div}(t^{1-2s} \nabla H_{\mathcal{C}}(\cdot, y)) = 0 &\quad \textrm{in}~\mathcal{C},
\\
\partial_{\nu}^{s} H_{\mathcal{C}}(\cdot, y) = 0 &\quad \textrm{on}~\Omega \times \{0\}.
\end{array}\right.
\end{equation}
Fix a smooth function $\phi \in C^{\infty}_{0} (N(E_{\delta}, \delta/4))$ such that $\phi =1$ on $E_{\delta}$ and $|\nabla \phi|^2 \leq C_0 \phi$ holds for some $C_0 > 0$,
and multiply $H_{\mathcal{C}}(\cdot,y) \phi (\cdot)$ to \eqref{eq-appendix-h}.
Then we have
\[\int_{\mathcal{C}} t^{1-2s}  | \nabla H_{\mathcal{C}} (z,y)|^2 \phi (z) + \int_{\mathcal{C}} t^{1-2s}[\nabla H_{\mathcal{C}}(z,y) \cdot \nabla \phi (z)] H_{\mathcal{C}}(z,y) dz = 0.\]
From this we deduce that
\begin{align*}
&\ \int_{\mathcal{C}} t^{1-2s}  | \nabla H_{\mathcal{C}} (z,y)|^2 \phi (z) dz
\\
&= - \int_{\mathcal{C}} t^{1-2s}[\nabla H_{\mathcal{C}}(z,y) \cdot \nabla \phi (z)] H_{\mathcal{C}}(z,y) dz
\\
& \leq \frac{1}{2 C_0} \int_{\mathcal{C}}t^{1-2s} |\nabla H_{\mathcal{C}} (z,y)|^2 |\nabla \phi (z)|^2 dz + 2C_0 \int_{N(E_{\delta}, \delta/4)} t^{1-2s} |H_{\mathcal{C}}(z,y)|^2 dz.
\end{align*}
Using the property $|\nabla \phi|^2 \leq C_0 \phi$ we derive that
\[\int_{\mathcal{C}} t^{1-2s}  | \nabla H_{\mathcal{C}} (z,y)|^2 \phi (z) dz \leq 4C_0 \int_{N(E_{\delta}, \delta/4)} t^{1-2s} |H_{\mathcal{C}}(z,y)|^2 dz.\]
It verifies inequality \eqref{eq-h-nabla}.
Since the assumption $\dist (y, E_{\delta}) > \delta /2$ implies $\dist(y, N(E_{\delta}, \delta/4)) > \delta /4$, it holds
\[\sup_{\dist (y, E_{\delta}) > \delta/2} \sup_{z \in N (E_{\delta},\delta/4)} |H_{\mathcal{C}}(z,y)|
\le \sup_{\dist (y, E_{\delta}) > \delta/2} \sup_{z \in N (E_{\delta},\delta/4)} |G_{\mathbb{R}^{n+1}_{+}} (z,y)| \leq C.\]
Combination of this and \eqref{eq-h-nabla} gives
\[\sup_{\dist(y,E_{\delta}) > \delta/2} \left(\int_{E_{\delta}} t^{1-2s} |\nabla_z H_{\mathcal{C}}(z,y)|^2 dz\right) \le
C \left(\int_{N(E_{\delta}, \delta/4)} t^{1-2s} dz\right) \le C.\]
This concludes the derivation of the desired uniform bound \eqref{eq-appendix-H-pre}.

It remains to take into consideration of the first term of \eqref{eq-U-decom}. We split the term as
\begin{align*}
&\ \int_{\Omega} \nabla_z G_{\mathbb{R}^{n+1}_{+}} (z,y) f(U) (y,0) dy \\
&= \int_{\mathcal{O}(\Omega,2\delta)} \nabla_z G_{\mathbb{R}^{n+1}_{+}} (z,y) f(U)(y,0) dy + \int_{\mathcal{I}(\Omega,{2\delta})} \nabla_z G_{\mathbb{R}^{n+1}_{+}} (z,y) f(U)(y,0) dy\\
&:= A_1 (z) + A_2 (z).
\end{align*}
Take $q > \frac{n}{s}$ and  $r>1$ satisfying $\frac{1}{q}+ \frac{1}{r} =1$. Then
\[|A_{1}(z)| \leq \left(\int_{\mathcal{O}(\Omega,2\delta)} |\nabla_z G_{\mathbb{R}^{n+1}_{+}}(z,y)|^{r} dy \right)^{1 \over r} \| f(U)(\cdot,0)\|_{L^{q}(\mathcal{O}(\Omega, 2\delta))}.\]
In light of the definition of $G_{\mathbb{R}^{n+1}_{+}}$, it holds that
\begin{align*}
\left(\int_{\mathcal{O}(\Omega,2\delta)} |\nabla_z G_{\mathbb{R}^{n+1}_{+}}(z,y)|^{r} dy\right)^{1 \over r} &\leq C \left(\int_{\mathcal{O}(\Omega,2\delta)}  \frac{1}{ |(x-y,t)|^{(n-2s+1)r}} dy\right)^{1 \over r}
\\
& \leq C \max\left\{ t^{{n \over r}-(n-2s+1)}, 1\right\}
= C\max \left\{ t^{- \frac{n}{q} +2s-1}, 1 \right\}.
\end{align*}
Thus we have
\[|A_1 (z)| \leq C \max\left\{ t^{-\frac{n}{q}+2s-1} ,1\right\}\| f(U)(\cdot,0)\|_{L^{q}(\mathcal{O}(\Omega,{2\delta}))}.\]
Using this we see
\begin{equation}\label{eq-appendix-G-1}
\begin{aligned}
\int_{E_{\delta}} t^{1-2s}|A_1(z)|^2 dz &\leq C \int_{0}^{1} \max\left\{ t^{1-2s} t^{-\frac{2n}{q}+4s-2}, t^{1-2s}\right\} \| f(U)(\cdot,0)\|_{L^{q}(\mathcal{O}(\Omega,{2\delta}))}^2 dt
\\
&  = \int_{0}^1 \max\left\{ t^{2s- \frac{2n}{q} -1}, t^{1-2s}\right\} \| f(U)(\cdot,0)\|_{L^{q}(\mathcal{O}(\Omega,{2\delta}))}^2 dt.
\\
& \leq C \|f(U)(\cdot,0)\|_{L^{q}(\mathcal{O}(\Omega,{2\delta}))}^2.
\end{aligned}
\end{equation}
Concerning the term $A_2$, we note that $E_{\delta}$ is away from $\mathcal{I}(\Omega,2\delta) \times \{0\}$. Thus we have
\[\sup_{z \in E_{\delta}, y \in  \mathcal{I}(\Omega,{2\delta})} |\nabla_z G_{\mathbb{R}^{n+1}_{+}} (z,y) |\leq C.\]
Hence
\[|A_2 (z)| \leq C \int_{\mathcal{I}(\Omega,{2\delta})} |f(U)(y,0)|dy, \quad z \in E_{\delta}.\]
Using this we find
\begin{equation}\label{eq-appendix-G-2}
\int_{E_{\delta}}t^{1-2s} |A_2 (z)|^2 dz \leq C \left(\int_{\mathcal{I}(\Omega,{2\delta})} |f(U)(y,0)|dy\right)^2.
\end{equation}
We have obtained the desired bound of $\int_{E_{\delta}} t^{1-2s} |\nabla U|^2 dz$ through the estimates \eqref{eq-appendix-H}, \eqref{eq-appendix-G-1} and \eqref{eq-appendix-G-2}. The proof is complete.
\end{proof}

\begin{rem}
Estimate \eqref{eq-appendix-H-pre} can be generalized to
\begin{equation}\label{eq-appendix-H-pre-1}
\sup_{y \in \Omega} \int_{E_{\delta}} t^{1-2s} |\nabla_z \partial_y^I H_{\mathcal{C}}(z,y)|^2 dz \le C,
\end{equation}
for any multi-index $I \in (\mathbb{N} \cup \{0\})^n$.
The proof of this fact follows in the same way as the derivation of \eqref{eq-appendix-H-pre} with an observation that $\partial_y^I H_{\mathcal{C}}(\cdot, y)$ satisfies equation \eqref{eq-derive-h}.
\end{rem}

\section{Proof of \eqref{eq-main-limit}}\label{sec_appen_b}
The aim of this section is to provide the derivation of \eqref{eq-main-limit}.
Due to a technical issue, we shall use an identity derived from integrating the local Pohozaev identity \eqref{eq-local-poho} (actually, its slight modification) with respect to $r\in(\delta, 2\delta)$,
and then apply the Lebesgue dominated convergence theorem to it.
The notations defined in Section \ref{sec_blowup} will be used here.

\medskip
In Section \ref{sec_asymp} we proved that $Q_{\epsilon}(x) := \|U_{\epsilon}(x,0)\|_{L^{\infty}(\Omega)} U_{\epsilon}(x,0)$ is of the form
\begin{equation}\label{eq-prop-v}
Q_{\epsilon} (z) = \int_{\Omega} G_{\mathcal{C}} (z,y) v_{\epsilon} (y) dy,
\end{equation}
where $v_{\epsilon} \in C^{0}(\Omega) $ satisfies
\begin{equation}\label{eq-prop-v2}
\lim_{\epsilon \rightarrow 0} \int_{\Omega} v_{\epsilon}(x) dx = \mathfrak{b}_{n,s} > 0 \quad \textrm{and} \quad \lim_{\epsilon \rightarrow 0} v_{\epsilon}(y) = 0\quad \textrm{in} ~C_{\text{loc}}^{0}(\Omega \setminus \{x_0\})
\end{equation}
(the number $\mathfrak{b}_{n,s}$ is defined in \eqref{eq-beta}).
Using an equivalent form of the local Pohozaev identity \eqref{eq-poho-root},
\begin{equation}\label{2divx}
\textrm{div}\left[2 t^{1-2s} \left( \langle z-X_0,\nabla V\rangle\nabla V - |\nabla V|^2 (z-X_0)\right) \right] + (n-2s) t^{1-2s} |\nabla V|^2 =0 \quad \text{in } D
\end{equation}
which holds for a function $V \in C^1(D)$ for some subset $D \subset \mathbb{R}^{n+1}_+$ satisfying div$(t^{1-2s}\nabla V) = 0$ in $D$,
we can obtain an identity corresponding to \eqref{eq-local-poho} with $\mathcal{I}(\Omega,r/2)$ and $D_r$ changed into $\Gamma_r$ and $B_r$, respectively.
After integrating it for the function $Q_{\epsilon}$ in $r$ from $\delta$ to $2\delta$, we have
\begin{equation}\label{eq-delta-2delta}
\begin{aligned}
&\int_{\delta}^{2\delta} \left[ \epsilon s C_s \int_{\Gamma_r} Q_{\epsilon}^2 (x,0) dx\right]dr \\
&= \int_{\delta}^{2\delta} \int_{\partial B_r^+} t^{1-2s}\left< \langle z-X_0, \nabla Q_{\epsilon} \rangle \nabla Q_{\epsilon} - (z-X_0) \frac{|\nabla Q_{\epsilon}|^2}{2}, \nu\right> dS dr
\\
&\ + \left(\frac{n-2s}{2}\right) \int_{\delta}^{2\delta} \int_{\partial B_r^+} t^{1-2s} Q_{\epsilon} \frac{\partial Q_{\epsilon}}{\partial \nu} dS dr
\\
&\ + \left[\int_{\delta}^{2\delta}{ \int_{\partial \Gamma_r} \langle x-x_0,\nu \rangle \|U_{\epsilon}(x,0)\|_{L^{\infty}}^2 \left({\epsilon \over 2} U_{\epsilon}^2 + {n-2s \over 2n} U_{\epsilon}^{2n \over n-2s}\right) dS_x}\right]dr.
\end{aligned}
\end{equation}
We shall apply the dominated convergence theorem to the right-hand side of this identity.
For this we need to find an integrable dominating function.
We only concern the first term in the right-hand side of  \eqref{eq-delta-2delta} because the other terms can be handled in similar fashion.
Set $E_{\delta}' = \cup_{r=\delta}^{2\delta} \partial B_r^+$ for some sufficiently small $\delta > 0$.
Then we bound $|\nabla Q_{\epsilon}|$ using \eqref{eq-prop-v} and \eqref{eq-prop-v2} as follows.
\[|\nabla Q_{\epsilon} (z)| = \left| \int_{\Omega} \nabla_z G_{\mathcal{C}} (z,y) v_{\epsilon}(y) dy\right|
\leq 2\mathfrak{b}_{n,s} \sup_{y \in \Gamma_{r/2}} |\nabla_z G_{\mathcal{C}} (z,y)| + \left|\int_{\Omega \setminus \Gamma_{r/2}} \nabla_z G_{\mathcal{C}}(z,y) v_{\epsilon} (y) dy \right|.\]
We will take $2\mathfrak{b}_{n,s} \sup_{y \in \Gamma_{r/2}} |\nabla_z G_{\mathcal{C}}(z,y)|$ as a dominating function
and prove that the quantity $|\int_{\Omega \setminus \Gamma_{r/2}} \nabla_z G_{\mathcal{C}}(z,y) v_{\epsilon} (y) dy|$ is negligible in the sense that its contribution tends to zero as $\epsilon \rightarrow 0$.
Note that
\[\left|\left\langle \langle z-X_0, \nabla Q_{\epsilon} \rangle \nabla Q_{\epsilon} - (z-X_0) \frac{|\nabla Q_{\epsilon}|^2}{2}, \nu\right\rangle \right| \leq C |\nabla Q_{\epsilon}(z)|^2 \quad \text{on } E_{\delta}'.\]
Thus it is enough to show
\begin{equation}\label{eq-ba}
\int_{E_{\delta}'}  t^{1-2s}\left(\sup_{y \in \Gamma_{r/2}}|\nabla_z G_{\mathcal{C}}(z,y)|\right)^2 dz \leq C
\end{equation}
and
\begin{equation}\label{eq-0ba}
\lim_{\epsilon \rightarrow 0} \int_{E_{\delta}'} t^{1-2s} \left(\int_{\Omega \setminus \Gamma_{r/2}} \nabla_z G_{\mathcal{C}}(z,y) v_{\epsilon} (y) dy\right)^2 dz = 0.
\end{equation}
\begin{proof}[Proof of \eqref{eq-ba} and \eqref{eq-0ba}]
We note that $\nabla_z G_{\mathbb{R}^{n+1}_{+}}(z,y)$ is uniformly bounded for $z \in E_{\delta}'$ and $y \in \Gamma_{r/2}$ since $\dist(E_{\delta}', \Gamma_{r/2}) > 0$.
Thus we only need to prove that
\begin{equation}\label{eq-basup}
\int_{E_{\delta}'} t^{1-2s} \left( \sup_{y \in \Gamma_{r/2}} |\nabla_z H_{\mathcal{C}}(z,y)|\right)^2 dz \leq C.
\end{equation}
Using the Sobolev embedding $H^{n}(\Omega) \hookrightarrow L^{\infty}(\Omega)$ and \eqref{eq-appendix-H-pre-1}, we obtain
\begin{equation*}
\textrm{(LHS) of \eqref{eq-basup}} \leq \sum_{|\alpha| \leq n} \int_{\Gamma_{r/2}}\left( \int_{E_{\delta}'} t^{1-2s} |\nabla_z \partial_y^I H_{\mathcal{C}} (z,y)|^2 dz\right) dy \leq C.
\end{equation*}
It proves \eqref{eq-ba}.

In order to deduce \eqref{eq-0ba} it suffices to show that
\begin{equation}\label{eq-g-zero}
\lim_{\epsilon \rightarrow 0} \int_{E_{\delta}'} t^{1-2s} \left( \int_{\Omega \setminus \Gamma_{r/2}} |\nabla_z G_{\mathbb{R}^{n+1}_{+}}(z,y) v_{\epsilon}(y)| dy \right)^2 dz = 0,
\end{equation}
and
\begin{equation}\label{eq-h-zero}
\lim_{\epsilon \rightarrow 0} \int_{E_{\delta}'} t^{1-2s} \left( \int_{\Omega \setminus \Gamma_{r/2}} |\nabla_z H_{\mathcal{C}}(z,y) v_{\epsilon}(y)| dy \right)^2 dz = 0.
\end{equation}
To show \eqref{eq-g-zero} we note that
\begin{align*}
\int_{\Omega \setminus \Gamma_{r/2}} |\nabla_z G_{\mathbb{R}^{n+1}_{+}} (z,y) v_{\epsilon}(y)| dy &\leq C \sup_{y \in \Omega \setminus \Gamma_{r/2}} |v_{\epsilon}(y)| \left( \int_{\Omega \setminus \Gamma_{r/2}} \frac{1}{|z-y|^{n+1-2s}} dy\right)
\\
&\leq \left\{ \begin{array}{ll}
C t^{2s-1} \sup_{y \in \Omega \setminus \Gamma_{r/2}}|v_{\epsilon}(y)| &\textrm{if}~ s \in (0, 1/2),
\\
C |\log t| \sup_{y \in \Omega \setminus \Gamma_{r/2}}|v_{\epsilon}(y)| &\textrm{if}~ s = 1/2,
\\
C\sup_{y \in \Omega \setminus \Gamma_{r/2}}|v_{\epsilon}(y)| &\textrm{if}~ s \in (1/2,1).
\end{array}
\right.
\end{align*}
Thus we have
\begin{align*}
\textrm{(LHS) of \eqref{eq-g-zero}} &\leq \lim_{\epsilon \rightarrow 0} \left(\sup_{y \in \Omega \setminus \Gamma_{r/2}} |v_{\epsilon}(y)|^2 \int_{E_{\delta}'} t^{1-2s} \max \{t^{4s-2}, |\log t|, 1\} dz\right)
\\
&\leq C \lim_{\epsilon \rightarrow 0} \left(\sup_{y \in \Omega \setminus \Gamma_{r/2}} |v_{\epsilon}(y)|^2\right) = 0.
\end{align*}
In order to prove \eqref{eq-h-zero} we use H\"older's inequality and \eqref{eq-appendix-H-pre} to obtain
\begin{align*}
\textrm{(LHS) of \eqref{eq-h-zero}} &\leq \int_{E_{\delta}'} t^{1-2s} \left( \int_{\Omega \setminus \Gamma_{r/2}} |\nabla_z H_{\mathcal{C}}(z,y)|^2 |v_{\epsilon}(y)| dy \right)\left( \int_{\Omega \setminus \Gamma_{r/2}} |v_{\epsilon}(y)| dy \right) dz
\\
&\leq \sup_{y \in \Omega} \left(\int_{E_{\delta}'} t^{1-2s} |\nabla_z H_{\mathcal{C}} (z,y)|^2 dz \right) \left(\int_{\Omega \setminus \Gamma_{r/2}} |v_{\epsilon}(y)|dy\right)^2
\\
&\leq C \left(\int_{\Omega \setminus \Gamma_{r/2}} |v_{\epsilon}(y)|dy\right)^2 \rightarrow 0\quad \textrm{as}~ \epsilon \rightarrow 0.
\end{align*}
It proves \eqref{eq-0ba}.
\end{proof}
Now we can take a limit $\epsilon \rightarrow 0$ to get
\begin{multline*}
\lim_{\epsilon \rightarrow 0} \int_{E_{\delta}'} t^{1-2s}\left\langle \langle z-X_0, \nabla Q_{\epsilon} \rangle \nabla Q_{\epsilon} - (z-X_0) \frac{|\nabla Q_{\epsilon}|^2}{2}, \nu\right\rangle dz
\\
= \mathfrak{b}_{n,s}^2 \int_{E_{\delta}'} t^{1-2s}\left\langle \langle z-X_0, \nabla G (z) \rangle \nabla G(z) - (z-X_0) \frac{|\nabla G(z)|^2}{2}, \nu\right\rangle dz.
\end{multline*}
In a similar way one can deduce that
\begin{align*}
&\lim_{\epsilon \rightarrow 0} \left[\left(\frac{n-2s}{2}\right) \int_{E_{\delta}'} t^{1-2s} Q_{\epsilon} \frac{\partial Q_{\epsilon}}{\partial \nu} dz \right.
\\
&\left. \hspace{80pt} + \int_{\delta}^{2\delta} \int_{\partial \Gamma_r} \langle x-x_0,\nu \rangle \|U_{\epsilon}(x,0)\|_{L^{\infty}}^2 \left({\epsilon \over 2} U_{\epsilon}^2 + {n-2s \over 2n} U_{\epsilon}^{2n \over n-2s}\right) dS_xdr \right]\\
&= \mathfrak{b}_{n,s}^2\left(\frac{n-2s}{2}\right) \int_{E_{\delta}'} t^{1-2s} G(z) \frac{\partial G(z)}{\partial \nu} dz.
\end{align*}
Combining the above two identities gives
\begin{equation}\label{eq-q2}
\begin{aligned}
& \lim_{\epsilon \to 0} \epsilon s C_s \int_{\delta}^{2\delta}\left[\int_{\Gamma_r} Q_{\epsilon}^2 (x,0) dx\right]dr \\
&= \mathfrak{b}_{n,s}^2\int_{\delta}^{2\delta}\left[\int_{\partial B_r^+} t^{1-2s}\left<(z-X_0, \nabla G(z)) \nabla G(z) - (z-X_0) \frac{|\nabla G(z)|^2}{2}, \nu\right> dS\right.
\\
&\ + \left.\left(\frac{n-2s}{2}\right) \int_{\partial B_r^+} t^{1-2s} G(z) \frac{\partial G(z)}{\partial \nu} dS\right] dr.
\end{aligned}
\end{equation}
On the other hand, since $x_0 \in \Gamma_r$,
we can derive that
\begin{equation}\label{eq-limit-l3}
\epsilon s \int_{\Gamma_r} Q_{\epsilon}^2 (x,0) dx
= \epsilon s \mu_{\epsilon}^{2} \int_{\Gamma_r} \mu_{\epsilon}^2 b_{\epsilon}^2\left(\mu_{\epsilon}^{\frac{p-1}{2s}} x\right) dx
=  \epsilon s \mu_{\epsilon}^{\frac{2n-8s}{n-2s}} \int_{(\Gamma_r)_{\epsilon}} b_{\epsilon}^2 (x) dx
\end{equation}
and
\begin{equation}\label{eq-limit-l3-1}
\lim_{\epsilon \rightarrow 0} \int_{(\Gamma_r)_{\epsilon}} b_{\epsilon}^2 (x) dx = \int_{\mathbb{R}^n} w_1^2 (x) dx = \mathfrak{c}_{n,s}^2 \frac{ \pi^{n/s} \Gamma (\frac{n}{2}-2s)}{\Gamma (n-2s)},
\end{equation}
where $(\Gamma_r)_{\epsilon} := (\Gamma_r - x_{0})/\epsilon$.
Now \eqref{eq-q2}, \eqref{eq-limit-l3}, and \eqref{eq-limit-l3-1} shows the validity of \eqref{eq-main-limit}.

\section{Technical computations in the proof of Theorem \ref{thm-m-multipeak}}\label{sec_appen_c}
In this section, we collect technical lemmas which are necessary during the proof of Theorem \ref{thm-m-multipeak}.

\subsection{Estimation of $P_{\epsilon} w_{\lambda, \xi}$ and $P_{\epsilon}\psi_{\lambda, \xi}^j$ ($j = 0, \cdots, n$)}
We recall the functions $w_{\lambda, \xi}$, $\psi^j_{\lambda, \xi}$, $P_{\epsilon}w_{\lambda, \xi}$ and $P_{\epsilon}\psi^j_{\lambda, \xi}$ defined in \eqref{bubble}, \eqref{psis} and \eqref{def_P_epsilon} for any $\lambda > 0$, $\xi \in \mathbb{R}^n$ and $j = 0, \cdots, n$.

\medskip
In the next two lemmas, we estimate the difference between $w_{\lambda, \xi}$ and $P_{\epsilon}w_{\lambda, \xi}$ (or $\psi^j_{\lambda, \xi}$ and $P_{\epsilon}\psi^j_{\lambda, \xi}$)
in terms of Green's function $G$ and its regular part $H$ of the fractional Laplacian $\mathcal{A}_s$ (see \eqref{Green_Omega} and \eqref{Robin_Omega} for their definitions).
\begin{lem}\label{lem-projection-robin}
Let $\lambda > 0$ and $\sigma = (\sigma^1, \cdots, \sigma^n) \in \Omega$. Then we have
\begin{align*}
P_{\epsilon} w_{\lambda, \sigma\epsilon^{-\alpha_0}} (x)
&= w_{\lambda, \sigma\epsilon^{-\alpha_0}} (x) - c_1 \lambda^{\frac{n-2s}{2}} H(\epsilon^{\alpha_0}x, \sigma) \epsilon^{(n-2s)\alpha_0} + o (\epsilon^{(n-2s)\alpha_0}),
\\
P_{\epsilon}\psi_{\lambda, \sigma\epsilon^{-\alpha_0}}^j (x) &= \psi_{\lambda, \sigma\epsilon^{-\alpha_0}}^j (x) - c_1 \lambda^{\frac{n-2s}{2}} \frac{\partial H}{\partial \sigma^j} (\epsilon^{\alpha_0}x, \sigma) \epsilon^{(n-2s+1)\alpha_0} + o(\epsilon^{(n-2s+1)\alpha_0}),
\\
P_{\epsilon}\psi_{\lambda, \sigma\epsilon^{-\alpha_0}}^{0} (x) &= \psi_{\lambda, \sigma\epsilon^{-\alpha_0}}^0 (x) - \frac{(n-2s)c_1}{2} \lambda^{\frac{n-2s-2}{2}} H(\epsilon^{\alpha_0}x, \sigma) \epsilon^{(n-2s)\alpha_0} + o(\epsilon^{(n-2s)\alpha_0})
\end{align*}
for all $x \in \Omega_{\epsilon}$ where $c_1 > 0$ is the constant defined in \eqref{constant_AB}.
Here the little $o$ terms tend to zero as $\epsilon \to 0$ uniformly in $x \in \Omega_{\epsilon}$ and $\sigma \in \Omega$ provided $\textrm{dist}(\sigma, \partial \Omega) > \overline{C}$ for some constant $\overline{C} > 0$.
\end{lem}
\begin{proof}
For fixed $\lambda > 0$ and $\sigma \in \Omega$, let $\Phi_{\lambda, \sigma\epsilon^{-\alpha_0}} = W_{\lambda, \sigma\epsilon^{-\alpha_0}}- P_{\epsilon}W_{\lambda, \sigma\epsilon^{-\alpha_0}}$.  Then $\Phi_{\lambda, \sigma\epsilon^{-\alpha_0}}$ satisfies
\[\left\{ \begin{array}{ll}
\text{div}\left(t^{1-2s}\nabla \Phi_{\lambda, \sigma\epsilon^{-\alpha_0}}\right) = 0 & \text{in } \mathcal{C}_{\epsilon},\\
\Phi_{\lambda, \sigma\epsilon^{-\alpha_0}} =W_{\lambda, \sigma\epsilon^{-\alpha_0}}& \text{on } \partial_L \mathcal{C}_{\epsilon},\\
\partial_{\nu}^{s} \Phi_{\lambda, \sigma\epsilon^{-\alpha_0}} = 0 &\text{on } \Omega_\epsilon \times \{0\}.
\end{array}\right.\]
On the other hand, the function $\mathcal{F}(z):= c_1 \lambda^{\frac{n-2s}{2}} H_{\mathcal{C}}(\epsilon^{\alpha_0}z, \sigma) \epsilon^{(n-2s)\alpha_0}$ defined for $z \in \mathcal{C}_{\epsilon}$ solves
\[\left\{ \begin{array}{ll}
\text{div}(t^{1-2s}  \nabla \mathcal{F})= 0 &\text{in } \mathcal{C}_{\epsilon},
\\
\mathcal{F}(z) = \epsilon^{(n-2s)\alpha_0} c_1 \lambda^{\frac{n-2s}{2}}G_{\mathbb{R}^{n+1}_{+}}(\epsilon^{\alpha_0} z,\sigma) & \text{on } \partial_L \mathcal{C}_{\epsilon},
\\
\partial_{\nu}^{s} \mathcal{F} = 0 &\text{on } \Omega_{\epsilon} \times \{0\}.
\end{array}\right.\]
Note that
\begin{align*}
W_{\lambda, \sigma\epsilon^{-\alpha_0}}(x,t)
& = \int_{\mathbb{R}^{n}} G_{\mathbb{R}^{n+1}_+}(x,t,y)  W_{\lambda, \sigma\epsilon^{-\alpha_0}}^{p} (y,0) dy
\\
&= \mathfrak{c}_{n,s}^p\int_{\mathbb{R}^n} G_{\mathbb{R}^{n+1}_{+}} (x,t,y) \frac{ \lambda^{\frac{n+2s}{2}}}{|(y-\sigma\epsilon^{-\alpha_0}, \lambda)|^{n+2s}}dy
\\
&= \mathfrak{c}_{n,s}^p\int_{\mathbb{R}^{n}} \lambda^{\frac{n-2s}{2}} G_{\mathbb{R}^{n+1}_{+}}(x,t, \lambda y+ \sigma\epsilon^{-\alpha_0}) \frac{1}{|(y,1)|^{n+2s}} dy \quad \text{for } (x,t) \in \mathbb{R}^{n+1}_+.
\end{align*}
For $(x, t) \in \partial_{L} \mathcal{C}$, we calculate
\begin{align*}
W_{\lambda, \sigma\epsilon^{-\alpha_0}}(x\epsilon^{-\alpha_0} ,t\epsilon^{-\alpha_0}) &=  \mathfrak{c}_{n,s}^p\int_{\mathbb{R}^{n}} \lambda^{\frac{n-2s}{2}} G_{\mathbb{R}^{n+1}_{+}}((x-\sigma)\epsilon^{-\alpha_0},t\epsilon^{-\alpha_0}, \lambda y) \frac{1}{|(y,1)|^{n+2s}} dy
\\
&=\epsilon^{(n-2s)\alpha_0}\mathfrak{c}_{n,s}^p\int_{\mathbb{R}^{n}} \lambda^{\frac{n-2s}{2}} G_{\mathbb{R}^{n+1}_{+}}(x-\sigma, t, \lambda y) \frac{\epsilon^{-n\alpha_0}}{|(y\epsilon^{-\alpha_0},1)|^{n+2s}} dy
\\
&=\epsilon^{(n-2s)\alpha_0}c_1\lambda^{\frac{n-2s}{2}} G_{\mathbb{R}^{n+1}_{+}}(x-\sigma,t, 0) + o (\epsilon^{(n-2s)\alpha_0}).
\end{align*}
As $\epsilon > 0$ goes to 0, the term $o (\epsilon^{(n-2s)\alpha_0})$ above converges to 0 uniformly in $(x,t) \in \partial_L\mathcal{C}$ and $\sigma \in \Omega$ satisfying $\textrm{dist}(\sigma, \partial \Omega) > \overline{C}$.

On the other hand, we have
\[\mathcal{F}(x\epsilon^{-\alpha_0},t\epsilon^{-\alpha_0}) = \epsilon^{(n-2s)\alpha_0} c_1 \lambda^{\frac{n-2s}{2}} G_{\mathbb{R}^{n+1}_{+}}(x,t, \sigma)
=  \epsilon^{(n-2s)\alpha_0} c_1 \lambda^{\frac{n-2s}{2}} G_{\mathbb{R}^{n+1}_{+}}(x-\sigma,t,0) .\]
Thus
\[\sup_{(x\epsilon^{-\alpha_0},t\epsilon^{-\alpha_0}) \in \partial_L \mathcal{C}}
|\Psi_{\lambda, \sigma\epsilon^{-\alpha_0}}(x\epsilon^{-\alpha_0} ,t\epsilon^{-\alpha_0})-\mathcal{F}(x\epsilon^{-\alpha_0},t\epsilon^{-\alpha_0}, \sigma)| = o (\epsilon^{(n-2s)\alpha_0}).\]
By the maximum principle (Lemma \ref{lem-maximum}), we get
\[\sup_{z \in \mathcal{C}_{\epsilon}}|\Psi_{\lambda, \sigma\epsilon^{-\alpha_0}}(z)-\mathcal{F}(z)| = o (\epsilon^{(n-2s)\alpha_0}).\]
By taking $z = (x,0)$ for $x \in \Omega_{\epsilon}$ we obtain $\sup_{x \in \Omega_{\epsilon}}|w_{\lambda, \sigma\epsilon^{-\alpha_0}}(x) -Pw_{\lambda, \sigma\epsilon^{-\alpha_0}}(x)-  \mathcal{F}(x,0)| = o (\epsilon^{(n-2s)\alpha_0})$. Now the first identity follows from the definition of $\mathcal{F}$.

The second and third estimation can be proved similarly.
\end{proof}

From the above lemma, we immediately get the following lemma.
\begin{lem}\label{lem-projection-green}
For any $\lambda > 0$ and $\sigma = (\sigma^1, \cdots, \sigma^n) \in \Omega$, we have
\begin{align*}
P_{\epsilon} w_{\lambda, \sigma\epsilon^{-\alpha_0}} (x) &= c_1 \lambda^{\frac{n-2s}{2}} G(\epsilon^{\alpha_0}x, \sigma) \epsilon^{(n-2s)\alpha_0} + o (\epsilon^{(n-2s)\alpha_0})
\\
P_{\epsilon}\psi_{\lambda, \sigma\epsilon^{-\alpha_0}}^j (x) &= c_1 \lambda^{\frac{n-2s}{2}} \frac{\partial G}{\partial \sigma^j} (\epsilon^{\alpha_0}x, \sigma) \epsilon^{(n-2s+1)\alpha_0} + o( \epsilon^{ (n-2s+1)\alpha_0})
\\
P_{\epsilon}\psi_{\lambda, \sigma\epsilon^{-\alpha_0}}^{0} (x) &=\left(\frac{n-2s}{2}\right) c_1 \lambda^{\frac{n-2s-2}{2}} G(\epsilon^{\alpha_0}x, \sigma) \epsilon^{(n-2s)\alpha_0} + o( \epsilon^{(n-2s)\alpha_0}),
\end{align*}
where the little $o$ terms tend to zero uniformly in $x \in \Omega_{\epsilon}$ and $\sigma \in \Omega$ provided $|\epsilon^{\alpha_0} x - \sigma|>C$ and $\textrm{dist}(\epsilon^{\alpha_0}x,\partial \Omega) >C$ for a fixed constant $C>0$.
As the previous lemma, $c_1 > 0$ is the constant given in \eqref{constant_AB}.
\end{lem}

\subsection{Basic estimates}
Let $w_i$ and $\psi_i^j$ (for $i = 1, \cdots, k$ and $j = 0, \cdots, n$) be the functions given in \eqref{W_i}.
Then applying the definition of $w_{\lambda, \xi}$ in \eqref{bubble}, Lemma \ref{lem-projection-robin} and the Sobolev trace inequality \eqref{eq-sharp-trace}, we can deduce the following estimates.
For the details, we refer to \cite{MP} in which the authors deal with the case $s = 1$.

\begin{lem}\label{lem_red_appen1}
It holds that
\[\| P_{\epsilon} w_i \|_{L^{\frac{2n}{n-2s}}(\Omega_{\epsilon})} \leq \| w_i \|_{L^{\frac{2n}{n-2s}}(\Omega_{\epsilon})} \leq C.\] Also we have
\[\|P_{\epsilon} w_i \|_{L^{\frac{2n}{n+2s}}(\Omega_{\epsilon})} \leq \left\{\begin{array}{ll} C &\quad \text{if}~ n>6s,\\
C \epsilon^{-(6s-n) \alpha_0/2} |\log \epsilon|& \quad \text{if}~ n \leq 6s.
\end{array}\right.\]
Similarly,
\begin{align*}
& \big\|P_{\epsilon}\psi_i^j \big\|_{L^{\frac{2n}{n-2s}}(\Omega_{\epsilon})} \leq C,
\quad \big\|P_{\epsilon}\psi_i^j \big\|_{L^{\frac{2n}{n+2s}}(\Omega_{\epsilon})} \leq C \quad \text{if } j = 1, \cdots, n,
\intertext{and}
& \big\|P_{\epsilon}\psi_i^0 \big\|_{L^{\frac{2n}{n+2s}}(\Omega_{\epsilon})} \leq
\left\{ \begin{array}{ll} C &\quad \text{if} ~ n > 6s,\\
C\epsilon^{-(6s-n)\alpha_0 /2 } |\log \epsilon| &\quad\text{if}~ n \leq 6s.
\end{array}\right.
\end{align*}
\end{lem}

\begin{lem}\label{lem_red_appen2}
For $i=1,\cdots,k$, we have
\[\big\| P_{\epsilon}\psi_i^{j} - \psi_i^{j} \big\|_{L^{\frac{2n}{n-2s}}(\Omega_{\epsilon})} \leq C \epsilon^{\alpha_0 \left(\frac{n}{2}-s+1\right)}\quad \text{if}~ j=1,\cdots, n\]
and
\[\big\| P_{\epsilon}\psi_i^{0}-\psi_i^{0} \big\|_{L^{\frac{2n}{n-2s}}(\Omega_{\epsilon})} \leq C \epsilon^{\alpha_0 \frac{n-2s}{2}}.\]
\end{lem}

\begin{lem}\label{lem_red_appen3}
It holds that
\[\left\| \left(\sum_{i=1}^k P_{\epsilon} w_i \right)^p - \sum_{i=1}^k w_i^p \right\|_{L^{\frac{2n}{n+2s}}(\Omega_{\epsilon})} \leq
\left\{\begin{array}{ll} C\epsilon^{\frac{n+2s}{2}\alpha_0} &\quad \text{if}~ n> 6s,\\
C\epsilon^{(n-2s) \alpha_0}|\log \epsilon| &\quad \text{if}~ n \leq 6s.
\end{array}\right.\]
Besides,
\[\left\| \left(\sum_{i=1}^k P_{\epsilon} w_i\right)^{p-1} - \sum_{i=1}^k  w_i^{p-1}\right\|_{L^{\frac{n}{2s}}(\Omega_{\epsilon})} \leq C\epsilon^{2s\alpha_0}.\]
and
\[\left\| \left[ \left(\sum_{i=1}^k P_{\epsilon} w_i\right)^{p-1} - \sum_{i=1}^k w_i^{p-1} \right] P_{\epsilon}\psi_h^{j} \right\|_{L^{\frac{2n}{n+2s}}(\Omega_{\epsilon})} \leq C \epsilon^{\alpha_0 \frac{n+2s}{2}}\]
for $h = 1, \cdots, k$ and $j = 0, 1, \cdots, n$.
\end{lem}

\subsection{Proof of Proposition \ref{prop_main_est}}\label{subsec_proof_of_main_est}
This subsection is devoted to give a proof of Proposition \ref{prop_main_est}.
\begin{proof}[Proof of Proposition \ref{prop_main_est}]
We first prove (1).
Applying $\widetilde{E}_{\epsilon}'(\bl^{\epsilon}, \bs^{\epsilon}) = 0$,
we can obtain after some computations that
\[{\partial \over \partial \varrho} \widetilde{E}_{\epsilon}'(\bl^{\epsilon}, \bs^{\epsilon})
= \sum_{h=1}^k\sum_{j=0}^n c_{hj} \left[\left(P_{\epsilon}\Psi_h^j,  \sum_{i=1}^k P_{\epsilon}{\partial W_i \over \partial \varrho} \right)_{\mathcal{C}_{\epsilon}} - \left( P_{\epsilon} {\partial \Psi_h^j \over \partial \varrho}, \Phi_{\bl^{\epsilon}, \bs^{\epsilon}}^{\epsilon} \right)_{\mathcal{C}_{\epsilon}} \right] = 0\]
where $\varrho$ is one of $\lambda_i$ or $\sigma_i^j$ with $i = 1, \cdots, k$ and $j = 0, \cdots, n$ (see \eqref{admissible}).
Using \eqref{Psi_ortho} and \eqref{Phi_est}, we can conclude that $c_{hj} = 0$ for all $h$ and $j$,
which implies that the function $U_{\epsilon}$ defined in the statement of the proposition is a solution of \eqref{equation-dilated-m}.
The assertion that $V_{\epsilon}$ is a solution of \eqref{equation-construction} is justified by the following sublemma provided $\epsilon > 0$ small.
\begin{slem}\label{sol_pos}
Suppose that $U \in H_{0,L}^{s}(\mathcal{C})$ is a solution of problem \eqref{equation-construction} with $U^p$ substituted by $U^p_+$ (here, the condition $U > 0$ in $\mathcal{C}$ is ignored).
If $\epsilon$ is small, then there is a constant $C>0$ depending only on $n$ and $s$, such that the function $U$ is positive.
\end{slem}
\begin{proof}
We multiply $U_{-}$ by equation  \eqref{equation-dilated-m} with $\epsilon = 1$. Then we have
\[\int_{\mathcal{C}} t^{1-2s} |\nabla U_{-}|^2 = \epsilon C_s \int_{\Omega \times \{0\}} U_{-}^2\]
(refer to \eqref{cs}).
By utilizing the Sobolev trace inequality and H\"older's inequality, we get
\[\| U_{-}(\cdot, 0) \|_{L^{\frac{2n}{n-2s}}(\Omega)} \leq \epsilon C \| U_{-}(\cdot, 0) \|_{L^{\frac{2n}{n-2s}}(\Omega)}\]
for some $C > 0$ independent of $U$.
Hence $U_{-}$ should be zero given that $\epsilon$ is sufficiently small.
The lemma is proved.
\end{proof}
\noindent The first part (1) of Proposition \ref{prop_main_est} is proved.

\medskip
We continue our proof by considering the second part (2). By \eqref{Phi_est}, there holds
\begin{align}
\widetilde{E}_{\epsilon}(\bls) &= E_{\epsilon} \left(\sum_{i=1}^k P_{\epsilon} W_{\lambda_i, \frac{\sigma_i}{\epsilon^{\alpha_0}}}\right) + O\left(\epsilon^{2\eta_0}\right) = E_{\epsilon} \left(\sum_{i=1}^k P_{\epsilon} W_i \right) + o\left(\epsilon^{(n-2s)\alpha_0}\right) \nonumber \\
&= {1 \over 2C_s}\int_{\mathcal{C}_{\epsilon}} t^{1-2s}\left|\nabla \left(\sum_{i=1}^k P_{\epsilon}W_i\right)\right|^2 - \int_{\Omega_{\epsilon}} F_{\epsilon}\left(i_{\epsilon}\left(\sum_{i=1}^kP_{\epsilon}W_i\right)\right) + o\left(\epsilon^{(n-2s)\alpha_0}\right) \label{decomp}
\end{align}
so it suffices to estimate each of the two terms that appear in \eqref{decomp} above.

Setting $B_i = B_{n}(\sigma_i, \delta_0/2) \subset \Omega$ where $\delta_0$ is a small number chosen in the definition \eqref{admissible} of $O^{\delta_0}$,
and applying Lemma \ref{lem-projection-robin} and Lemma \ref{lem-projection-green}, we find that
\begin{align*}
\int_{\Omega_{\epsilon}} w_i^p P_{\epsilon} w_i &= \int_{\Omega_{\epsilon}} w_1^{p+1} + \int_{\Omega_{\epsilon}} w_i^p(P_{\epsilon}w_i-w_i) = c_0 - c_1^2\lambda_i^{n-2s}H(\sigma_i, \sigma_i)\epsilon^{(n-2s)\alpha_0} + o(\epsilon^{(n-2s)\alpha_0}),
\\
\int_{\Omega_{\epsilon}} w_h^p P_{\epsilon}w_i &= \int_{B_i \over \epsilon^{\alpha_0}} w_h^p P_{\epsilon}w_i + o(\epsilon^{(n-2s)\alpha_0}) = c_1^2(\lambda_h\lambda_i)^{n-2s \over 2}G(\sigma_h,\sigma_i)\epsilon^{(n-2s)\alpha_0} + o(\epsilon^{(n-2s)\alpha_0}),
\\
\int_{\Omega_{\epsilon}} w_iP_{\epsilon} w_i &= \int_{\Omega_{\epsilon}} w_i^2 + o(1) = c_2 \lambda_i^{2s} + o(1) \quad \text{(if } n > 4s),\\
\int_{\Omega_{\epsilon}} w_hP_{\epsilon} w_i &= o(1) \quad \text{(if } n > 4s),
\end{align*}
for $i, h = 1, \cdots, k$ and $i \ne h$,
where $G$ and $H$ are the functions defined in \eqref{Green_Omega} and \eqref{Robin_Omega}, and $c_1$ and $c_2$ are positive constants given in \eqref{constant_AB} while $c_0$ is defined in \eqref{constant_C}.

Then the estimates obtained in the previous paragraph yield that
\begin{align*}
&\ {1 \over 2C_s}\int_{\mathcal{C}_{\epsilon}} t^{1-2s}\left|\nabla \left(\sum_{i=1}^k P_{\epsilon}W_i\right)\right|^2
= {1 \over 2} \sum_{i=1}^k \int_{\Omega_{\epsilon}} w_i^pP_{\epsilon}w_i + {1 \over 2} \sum_{\substack{i, h = 1 \\ i \ne h }}^k \int_{\Omega_{\epsilon}} w_h^pP_{\epsilon}w_i
\\
& = {kc_0 \over 2} + \left[{c_1^2 \over 2} \left\{\sum_{i \ne h}G(\sigma_i, \sigma_h)(\lambda_h\lambda_i)^{n-2s \over 2} - \sum_{i=1}^k H(\sigma_i,\sigma_i)\lambda_i^{n-2s} \right\} + o(1) \right]\epsilon^{(n-2s)\alpha_0}
\end{align*}
and
\begin{align*}
&\ \int_{\Omega_{\epsilon}} F_{\epsilon}\left(\sum_{i=1}^kP_{\epsilon}w_i\right)
\\
&= \sum_{h=1}^k \bigg[\int_{B_h \over \epsilon^{\alpha_0}} F_{\epsilon} \bigg(w_h + (P_{\epsilon}w_h - w_h) + \sum_{\substack{i, h = 1 \\ i \ne h }}^k P_{\epsilon}w_i \bigg) - F_{\epsilon}(w_h) \bigg] + \sum_{h=1}^k \int_{B_h \over \epsilon^{\alpha_0}} F_{\epsilon}(w_h) + o(\epsilon^{(n-2s)\alpha_0})
\\
&= \sum_{h=1}^k \left[\int_{B_h \over \epsilon^{\alpha_0}} F_{\epsilon}(w_h) + \int_{B_h \over \epsilon^{\alpha_0}} f_{\epsilon}(w_h)(P_{\epsilon}w_h-w_h)\right] + \sum_{i \ne h} \int_{B_h \over \epsilon^{\alpha_0}} f_{\epsilon}(w_h)P_{\epsilon}w_i + o(\epsilon^{(n-2s)\alpha_0})
\\
&= {kc_0 \over p+1} + \left[c_1^2 \left\{\sum_{i \ne h}G(\sigma_i, \sigma_h)(\lambda_h\lambda_i)^{n-2s \over 2} - \sum_{i=1}^k H(\sigma_i,\sigma_i)\lambda_i^{n-2s} \right\} + {c_2 \over 2} \sum_{i=1}^k \lambda_i^{2s} + o(1) \right]\epsilon^{(n-2s)\alpha_0}
\end{align*}
Note that here we also used that $1+2s\alpha_0 = (n-2s)\alpha_0$ which holds owing to our choice $\alpha_0 = {1 \over n-4s}$.
As a consequence, \eqref{energy_local_expansion} holds $C^0$-uniformly in $O^{\delta_0}$.
Similarly, with Lemmas \ref{lem_red_appen1}, \ref{lem_red_appen2} and \ref{lem_red_appen3}, one can conclude that \eqref{energy_local_expansion} has its validity in $C^1$-sense (see \cite[Section 7]{GMP} and \cite[Proposition 2.2]{MP}). This completes the proof.
\end{proof}

\end{document}